%% file: AugmentLagr2026_Ver2.tex
\documentclass[11pt]{article}

\input{inputusepackages}

\newcommand{\OUTPUT}{\item[\textbf{Output:}]}

\usepackage{float}
\usepackage{graphicx}
\usepackage{adjustbox}
\usepackage{booktabs}
\usepackage[table]{xcolor}

\definecolor{TablePink}{HTML}{F28FA1}

\newcommand{\pinktablehead}[1]{%
    \textcolor{white}{\textbf{#1}}%
}

\begin{document}



\title{
\href{http://orion.math.uwaterloo.ca/~hwolkowi/henry/reports/ABSTRACTS.html}{
Optimal Nonergodic Primal-Dual Complexity of Efficient Inexact Parameter-Free Augmented Lagrangian Methods
}
   \footnote{
Emails respectively: a3sujana@uwaterloo.ca, 
sghadimi@uwaterloo.ca,
hwolkowicz@uwaterloo.ca}
}

\author{
\href{https://asujanani6.github.io/}
 {Arnesh Sujanani}\thanks{
\href{http://www.math.uwaterloo.ca/co/}{Department of Combinatorics and Optimization}, University of Waterloo,  ON, Canada}
\and 
\href{https://uwaterloo.ca/management-science-engineering/contacts/saeed-ghadimi}{Saeed
Ghadimi}\thanks{Department of Management Science and Engineering, University of
Waterloo, ON, Canada}
 \and
\href{https://www.math.uwaterloo.ca/~hwolkowi/}
{Henry Wolkowicz}\footnotemark[2]
}


\date{
Revision of \today, \currenttime
}

\maketitle
\tableofcontents



%
%
\listofalgorithms
\listoftables

\begin{abstract}
Augmented Lagrangian (AL) methods are a classical framework for constrained optimization, but for directly verifiable approximate KKT points, known first-order complexity bounds for standard inexact AL methods are suboptimal, while the best known proximal augmented Lagrangian (PAL) bounds retain an additional logarithmic factor. We consider linearly constrained convex composite problems with a smooth convex term and a possibly nonsmooth closed proper convex term with compact domain. We develop three inexact AL schemes that preserve the standard AL subproblem structure and attain the optimal primal-dual complexity $\mathcal O(\epsilon^{-1})$ in the convex setting, improving prior AL bounds of $\mathcal O(\epsilon^{-4/3})$, $\mathcal O(\epsilon^{-7/4})$, and $\mathcal O(\epsilon^{-2})$, and removing the logarithmic factor from PAL guarantees. Two variants are parameter-free, and all three admit nonergodic guarantees, including a stronger last-iterate guarantee for one variant. 

These results show that proximal regularization, ergodic averaging, and prior knowledge of problem-dependent constants are not intrinsic requirements for attaining optimal verifiable primal-dual complexity within the standard AL framework. A key ingredient is a parameter-free accelerated method that computes verifiable stationarity certificates for the standard, unregularized AL subproblems with optimal complexity. In the strongly convex setting, our methods attain near-optimal complexity $\mathcal O(\epsilon^{-1/2}\log(\epsilon^{-1}))$, with two parameter-free variants. Numerical experiments on six problem classes, including elastic-net least-squares regression, group-sparse Huberized support vector machines, and a quantum semidefinite program (SDP), demonstrate substantial computational advantages over a representative PAL method, with speedups frequently ranging from $5$ to $50$ times.
\end{abstract}

\section{Introduction}\label{sec:intro}
\phantomsection
\index{$\psi(z)$, objective function}
\index{objective function, $\psi(z)$}
\index{$\psi_s(z)$, objective function smooth part}
\index{objective function smooth part, $\psi_s(z)$}
\index{$\psi_n(z)$, objective function nonsmooth part}
\index{objective function nonsmooth part, $\psi_n(z)$}
\phantomsection

Augmented Lagrangian (AL) methods are a fundamental tool for constrained optimization. By replacing difficult projections onto the feasible region with a sequence of penalized subproblems and multiplier updates, they provide a flexible framework that is particularly well suited to large-scale linearly constrained models, including game theory applications, radar communication waveform design, conic and semidefinite optimization problems, and recent low-rank SDP approaches \cite{AL1Lu, AL16Necoara, AL22Ding, AL42Monteiro, AL43Aguirre, AL26Hou, AL12Tang, AL25Wang, AL57Kananian, AL59Bazzi, AL62Deng}.

In this paper, we investigate whether the classical inexact AL architecture is itself sufficient to attain optimal and directly verifiable first-order guarantees for linearly constrained convex composite optimization problems of the form
\begin{equation}
\label{LinearConstraints}
\textdef{$\pLC$} :=\min \left\{\psi(z):=\psi_s(z)+\psi_n(z): \cA z=b\right\},
\end{equation}
where \textdef{$\cA:\E\to \Rm$} is a nonzero linear transformation, $\mathbb E$ is a Euclidean space, and \textdef{$\psi: \mathbb E\rightarrow \R$} is convex or $\bar \mu$-strongly convex. The nonsmooth term \textdef{$\psi_n:\cN\subset \mathbb E\rightarrow \R$} is closed, proper, and convex, and its domain is assumed to be a \textdef{compact convex set, $\Dn$}; in particular, its diameter is finite:
\begin{equation}\label{diameter}
D:=\sup \{ \|z-z'\| : z, z' \in \Dn \} < \infty.
\end{equation}
The smooth term \textdef{$\psi_s:\mathbb E\rightarrow \R$} is convex, differentiable, and \textdef{$\bar L$-smooth}, meaning that
\index{$\Dn$, compact convex domain}
\begin{equation}
\label{ineq:uppercurvature1}
\|\nabla \psi_s(x)-\nabla \psi_s(y)\|\leq \bar L\|x-y\|.
\end{equation}

Given a nondecreasing sequence of positive penalty parameters $\{c_k\}$, an inexact AL method proceeds as follows. At iteration $k$, given the current multiplier $p_{k-1}$, it approximately minimizes the \textdef{augmented Lagrangian function} $\mathcal L_{c_k}(z,p_{k-1})$, where
\begin{equation}\label{lagrangian2}
{\cal L}_{c}(z,p):=\cL^s_{c}(z,p)+\psi_n(z),
\end{equation}
and $\cL_c^s$ denotes its smooth part, namely,
\begin{equation}\label{smooth part}
\textdef{$\cL^s_c(z,p):=\psi_s(z)+\left\langle p,\cA z-b\right\rangle+\frac{c}{2}\|\cA z-b\|^2$}.
\end{equation}
It then performs the full multiplier update
\begin{equation}\label{eq:dual_update}
p_{k}=p_{k-1}+c_k(\cA z_{k}-b).
\end{equation}
Thus, each outer iteration consists of an inexact primal minimization followed by a dual ascent step.

Inexact proximal augmented Lagrangian (PAL) methods modify this template by adding a proximal regularization term. Given positive scalars $\{\lambda_k\}$, these methods approximately minimize $\mathcal P^{\lambda_k}_{c_k}(z,p_{k-1})$, where
\begin{equation}\label{PALFunction}
    \mathcal P^{\lambda}_{c}(z,p):=\mathcal L_c(z,p)+\frac{1}{2\lambda}\|z-z_{k-1}\|^2.
\end{equation}
The proximal term can simplify the complexity analysis and improve conditioning, but it also changes the subproblem and can be computationally undesirable when the unregularized AL subproblem has exploitable structure. This distinction is structural: PAL methods solve a different sequence of primal subproblems, so their complexity guarantees do not by themselves determine what is achievable within the standard, unregularized AL framework.

Despite substantial progress, existing first-order inexact AL methods do not simultaneously provide all of the following features for \Cref{LinearConstraints}: an optimal primal-dual/KKT complexity bound, a fully verifiable termination criterion, no proximal regularization of the AL subproblems, and no prior knowledge of problem-dependent constants. Some AL complexity guarantees target primal objective and feasibility criteria involving the unknown optimal value \cite{AL18Aybat, AL2Xu, AL16Necoara, AL17Patrascu, AL23Yurtsever, AL30Nedelcu}. Although important theoretically, such guarantees do not directly yield implementable stopping rules. For verifiable primal-dual/KKT certificates, the known first-order AL bounds remain suboptimal: Xu obtains an $\mathcal O(\epsilon^{-4/3})$ primal-dual complexity bound \cite{AL2Xu}; Necoara obtains an $\mathcal O(\epsilon^{-3/2})$ bound \cite{AL16Necoara}; Liu et al. establish an $\mathcal O(\epsilon^{-2})$ nonergodic primal-dual bound \cite{AL14Liu}; and both Lan and Monteiro's method and the standard first-order I-AL method of Lu and Zhou attain $\mathcal O(\epsilon^{-7/4})$ KKT-type bounds \cite{AL13Lan,AL1Lu}. Necoara's fast inexact AL method, which incorporates Nesterov-type acceleration into the outer scheme rather than following the traditional inexact AL framework, achieves an $\mathcal O(\epsilon^{-4/3})$ primal-dual complexity bound \cite{AL16Necoara}. Lu and Zhou's adaptively regularized PAL method improves the KKT complexity to $\mathcal O(\epsilon^{-1}\log(\epsilon^{-1}))$ \cite{AL1Lu}, and related PAL schemes obtain similar logarithmically suboptimal bounds \cite{LuM, AL8Lin, AL7Li, AL5Burns}. Thus, the remaining issue is not merely to improve a complexity exponent, but to determine the capabilities of the standard AL framework. 

This leads to the central question addressed in this paper: can an unregularized inexact AL method attain the \textbf{optimal} $\mathcal O(\epsilon^{-1})$ first-order complexity for \textbf{directly verifiable approximate KKT points} of \Cref{LinearConstraints}, while avoiding both proximal regularization of the AL subproblems and problem-dependent parameter tuning? We answer this question \textbf{affirmatively}. For merely convex objectives, we develop inexact AL methods with the \textbf{optimal} $\mathcal O(\epsilon^{-1})$ first-order complexity for computing \textbf{approximate primal-dual/KKT solutions}; two of the proposed variants are \textbf{parameter-free.} For strongly convex objectives, we establish a near-optimal $\mathcal O(\epsilon^{-1/2}\log(\epsilon^{-1}))$ complexity bound without requiring knowledge of the strong convexity parameter or any other problem-dependent parameters. 

Two ideas are central to the analysis. First, approximate stationarity of the standard AL subproblems, together with the full multiplier update, yields a telescoping dual-distance estimate that directly controls primal feasibility without ergodic averaging. Second, parameter-free accelerated regularization and restart schemes minimize the unregularized AL subproblems and produce verifiable composite-stationarity certificates at optimal or near-optimal inner complexity. For the adaptive last-iterate variant, a uniform multiplier bound further allows the increasing penalty parameter to directly control feasibility of the current iterate.

We also report numerical results for \textbf{six} important classes of linearly constrained optimization problems: \textbf{quadratic programs}, \textbf{logistic regression}, \textbf{elastic-net least-squares regression}, \textbf{group-sparse Huberized support vector machines (SVMs)}, \textbf{Bayesian D-optimal experimental design}, and a \textbf{Huberized quantum semidefinite program (SDP)}. These experiments show that the proposed AL methods can substantially outperform a representative PAL scheme \cite{LuM}, with \textbf{speedups} frequently exceeding \textbf{5--50x}.

\subsection{Literature Review}
\subsubsection{Parameter-Free Optimization Methods for Unconstrained Minimization}
When $\cA=0$ in \Cref{LinearConstraints}, the problem reduces to unconstrained composite minimization. Parameter-free first-order methods are attractive in this setting because they spare users from estimating quantities such as global Lipschitz constants, growth constants, and strong convexity moduli \cite{AL35Suh, AL44Liang, AL45Liang, AL39Sujanani, AL36Zhou, AL27Lan, AL37Li, AL34Wu, AL60Jiang, AL61Wu}. In terms of objective-value accuracy, when $\psi$ is convex, backtracking FISTA and related accelerated composite-gradient methods attain the optimal $\mathcal O(\epsilon^{-1/2})$ complexity for computing a point $z$ such that $\psi(z)-\psi_{*}\leq\epsilon$ \cite{AL46Beck, AL47Nesterov, AL49Scheinberg}.

For verifiable termination, however, objective residuals are less useful than stationarity measures such as gradient norms or composite gradient mappings. When $\psi$ is convex, standard accelerated schemes generally attain only the suboptimal $\mathcal O(\epsilon^{-2/3})$ complexity for reducing such stationarity measures below $\epsilon$ \cite{AL48Kim, AL51Shi, AL52Kim, AL53Li, AL54Monteiro}. Nesterov's proximal regularization technique yields a near-optimal $\mathcal O(\epsilon^{-1/2}\log(\epsilon^{-1}))$ bound \cite{AL29Nesterov}, while subsequent parameter-dependent methods remove the logarithmic factor under additional algorithmic requirements, such as a prescribed iteration horizon or consecutive method calls \cite{AL33Lee, AL50Nesterov}.

Lan et al. \cite{AL27Lan} recently developed accumulative-regularization methods that attain the optimal $\mathcal O(\epsilon^{-1/2})$ stationarity complexity and include parameter-free variants. Their results substantially clarify the complexity landscape for unconstrained stationarity. The linearly constrained composite setting considered here is more demanding: a verifiable certificate must simultaneously control Lagrangian stationarity and feasibility, while the AL subproblems themselves vary with both the multiplier and the penalty parameter.

\subsubsection{Inexact Augmented Lagrangian Methods}
The exact AL method dates back to Hestenes and Powell \cite{AL55Hestenes, AL56Powell}, and Rockafellar established convergence of inexact AL schemes through the proximal-point interpretation \cite{AL41Rockafellar}. Modern first-order analyses quantify both the accuracy with which the AL subproblems must be solved and the total number of inner first-order iterations required to reach a prescribed target accuracy.

One line of work focuses on primal accuracy. For convex objectives, Aybat and Iyengar \cite{AL18Aybat} obtained an $\mathcal O(\epsilon^{-1}\log(\epsilon^{-1}))$ complexity bound, and Xu \cite{AL2Xu} improved this result to the optimal $\mathcal O(\epsilon^{-1})$ gradient-evaluation complexity for computing $z\in\Dn$ satisfying
\begin{equation}\label{PrimalSolution}
|\psi(z)-\pLC|\leq \epsilon, \quad \|\cA z-b\|\leq \epsilon,
\end{equation}
where $\pLC$ is defined in \Cref{LinearConstraints}. Such guarantees are theoretically important, but \eqref{PrimalSolution} is generally not a practical stopping rule because $\pLC$ is unknown.

A more implementable target is an approximate KKT point, namely, a triple $(z,p,r)$ satisfying
\begin{equation}\label{KKTPoint}
 r \in \nabla \psi_s(z)+\partial \psi_n(z)+\cA^{*}p, \quad \|r\|\leq
\epsilon, \quad \|\cA z-b\|\leq \epsilon.
\end{equation}
This certificate can be checked directly within an AL method. The main first-order AL complexity results for primal-dual or KKT-type criteria therefore provide the natural benchmarks for this paper. Lan and Monteiro \cite{AL13Lan} analyze a first-order inexact AL method for computing a primal-dual solution according to the verifiable criterion in \Cref{KKTPoint}. They establish an $\mathcal O(\epsilon^{-7/4})$ complexity bound, although their criterion for terminating the AL subproblem solves is not directly verifiable because it depends on the unknown optimal values of the corresponding augmented Lagrangian subproblems. Lu and Zhou \cite{AL1Lu} and Necoara \cite{AL16Necoara} establish KKT complexity bounds of $\mathcal O(\epsilon^{-7/4})$ and $\mathcal O(\epsilon^{-3/2})$, respectively, for standard first-order I-AL methods that also seek points satisfying \eqref{KKTPoint}. Like the method of Lan and Monteiro, these methods are not parameter-free and use AL-subproblem termination criteria that depend on unknown optimal values of the AL objectives.

Xu \cite{AL2Xu} and Liu et al. \cite{AL14Liu} also study the complexity of inexact AL methods for computing primal-dual solutions, although their solution criteria differ slightly from \Cref{KKTPoint}. Xu establishes an ergodic primal-dual complexity bound of $\mathcal O(\epsilon^{-4/3})$, whereas Liu et al. establish a nonergodic primal-dual complexity bound of $\mathcal O(\epsilon^{-2})$. These results demonstrate steady progress toward verifiable AL certificates, but their complexities remain worse than $\mathcal O(\epsilon^{-1})$, and they do not provide the parameter-free, genuine-AL guarantee pursued here. More recently, easily implementable relative-error AL frameworks have been developed, but they attain a suboptimal $\mathcal O(\epsilon^{-2})$ complexity bound for \eqref{KKTPoint} \cite{AL4Qu}.

PAL and other proximal or regularized AL frameworks constitute a second line of work. Lu and Mei \cite{LuM} establish an $\mathcal O(\epsilon^{-1}\log(\epsilon^{-1}))$ complexity bound in the convex case and an $\mathcal O(\epsilon^{-1/2}\log(\epsilon^{-1}))$ bound in the strongly convex case for KKT-type solutions. Related PAL and inexact proximal-point analyses appear in \cite{AL7Li, AL8Lin, AL5Burns}. Although powerful, these methods introduce the proximal term in \Cref{PALFunction}, thereby changing the primal subproblem, and their known convex-case bounds retain a logarithmic dependence on the target accuracy. Moreover, strongly convex variants typically require knowledge of the strong convexity parameter.

Accordingly, our contribution is not merely a sharper complexity exponent. It establishes that, within the problem class studied here, the standard, unregularized AL architecture is sufficient to combine directly verifiable KKT termination with optimal nonergodic primal-dual complexity. In two variants, this is achieved without knowledge of problem-dependent constants.


\label{sect:litrev}

\subsection{Summary of Contributions}
We now summarize our main contributions.
\begin{itemize}
    \item[1.] We establish that standard inexact AL methods can attain the optimal $\mathcal O(\epsilon^{-1})$ first-order complexity for computing directly verifiable primal-dual/KKT solutions in the merely convex setting. We realize this result through three new inexact AL schemes. This improves upon the $\mathcal O(\epsilon^{-4/3})$ primal-dual bound of Xu \cite{AL2Xu}, the $\mathcal O(\epsilon^{-3/2})$ primal-dual bound of Necoara \cite{AL16Necoara}, the $\mathcal O(\epsilon^{-2})$ nonergodic primal-dual bound of Liu et al. \cite{AL14Liu}, and the $\mathcal O(\epsilon^{-7/4})$ primal-dual/KKT-type bounds of Lan--Monteiro \cite{AL13Lan} and the standard I-AL method of Lu--Zhou \cite{AL1Lu}. Our complexity bound also improves upon logarithmically suboptimal $\mathcal O(\epsilon^{-1}\log(\epsilon^{-1}))$ guarantees available for regularized PAL type methods \cite{AL1Lu, LuM, AL8Lin, AL7Li, AL5Burns}.

    \item[2.] We provide a detailed nonergodic convergence analysis of all three proposed AL methods, yielding guarantees for individual generated iterates rather than ergodic averages. Two variants yield best-iterate guarantees, whereas the third establishes the stronger last-iterate guarantee under mildly stronger assumptions. The last-iterate variant is also more adaptive and employs a novel mechanism for updating the tolerances used to solve its AL subproblems.
    
    \item[3.] In the convex setting, two of the three AL variants are parameter-free. To the best of our knowledge, they are the first parameter-free inexact AL methods that combine verifiable KKT termination with the optimal $\mathcal O(\epsilon^{-1})$ primal-dual complexity.

    \item[4.] A key ingredient in these bounds is a parameter-free regularization method with the optimal $\mathcal O(\epsilon^{-1/2})$ complexity for solving the unconstrained composite AL subproblems. The method extends the parameter-free ideas of \cite{AL27Lan} from smooth unconstrained minimization to the unconstrained composite setting and uses a termination criterion specifically tailored to the AL analysis. Beyond its role in the AL analysis, PF-AR therefore provides an optimal parameter-free method for computing directly verifiable stationarity certificates in convex composite minimization with compact nonsmooth domain.
    
    \item[5.] When $\psi$ is $\bar \mu$-strongly convex with $\bar \mu>0$, we establish a near-optimal $\mathcal O(\epsilon^{-1/2}\log(\epsilon^{-1}))$ complexity bound. Unlike the parameter-dependent methods of \cite{AL11Li, LuM}, two of our methods are fully parameter-free. The AL subproblems are solved using a simplified restarted FISTA scheme developed in this paper and inspired by \cite{AL39Sujanani}. Compared with the scheme in \cite{AL39Sujanani}, our method uses a simpler restart condition and permits a more flexible choice of restart point.

    \item[6.] Finally, numerical experiments on six important classes of linearly constrained problems demonstrate that the proposed AL methods can substantially outperform a representative PAL method by Lu and Mei \cite{LuM}. On many instances, the observed speedups often range from 5x to 50x.
\end{itemize}

The two tables below summarize our six main contributions further.

\begin{table}[H]
\centering

\caption{Comparison of Methods for Convex Linearly-Constrained Optimization}

\begin{adjustbox}{max width=\linewidth}
\scalebox{0.7}{%
\begin{tabular}{|lccc|}
\toprule

\rowcolor{TablePink}
\pinktablehead{Name}
&
\pinktablehead{Primal-Dual Complexity}
&
\pinktablehead{Parameter-Free}
&
\pinktablehead{AL or PAL Method}
\\

\midrule

Lan and Monteiro \cite{AL13Lan}
&
$\mathcal{O}\left(\epsilon^{-7/4}\right)$
&
No
&
AL
\\

Liu et al.~\cite{AL14Liu}\footnotemark
&
$\mathcal{O}\left(\epsilon^{-2}\right)$
&
No
&
AL
\\

Lu and Zhou \cite{AL1Lu}
&
$\mathcal{O}\left(\epsilon^{-7/4}\right)$
&
No
&
AL
\\

Necoara \cite{AL16Necoara}
&
$\mathcal{O}\left(\epsilon^{-3/2}\right)$
&
No
&
AL
\\

Xu \cite{AL2Xu}\footnotemark[\value{footnote}]
&
$\mathcal{O}\left(\epsilon^{-4/3}\right)$
&
No
&
AL
\\

\hline

This Work (\Cref{alg:O-IAL})
&
$\mathcal{O}\left(\epsilon^{-1}\right)$
&
No
&
AL
\\

This Work (\Cref{alg:OPF-IAL})
&
$\mathcal{O}\left(\epsilon^{-1}\right)$
&
Yes
&
AL
\\

This Work (\Cref{alg:APF-IAL})
&
$\mathcal{O}\left(\epsilon^{-1}\right)$
(stronger last-iterate)
&
Yes
&
AL
\\

\hline

Burns and Liang \cite{AL5Burns}
&
$\mathcal{O}\left(
    \epsilon^{-1}\log(\epsilon^{-1})
\right)$
&
No
&
PAL
\\

Li and Qu \cite{AL7Li}
&
$\mathcal{O}\left(
    \epsilon^{-1}\log(\epsilon^{-1})
\right)$
&
No
&
PAL
\\

Lin and Xu \cite{AL8Lin}
&
$\mathcal{O}\left(
    \epsilon^{-1}\log(\epsilon^{-1})
\right)$
&
No
&
PAL
\\

Lu and Zhou \cite{AL1Lu}
&
$\mathcal{O}\left(
    \epsilon^{-1}\log(\epsilon^{-1})
\right)$
&
No
&
PAL
\\

Lu and Mei \cite{LuM}
&
$\mathcal{O}\left(
    \epsilon^{-1}\log\left(\epsilon^{-1}\right)
\right)$
&
Yes
&
PAL
\\

\bottomrule
\end{tabular}%
}
\end{adjustbox}

\end{table}

\footnotetext[\value{footnote}]{%
It should be noted that the notions of approximate primal--dual solution employed by
Liu et al.~\cite{AL14Liu} and Xu~\cite{AL2Xu} are closely related to,
but formally distinct from, those employed by the other works listed
in the table.%
}

\begin{table}[H]
\centering

\begingroup
\setlength{\tabcolsep}{4pt}
\renewcommand{\arraystretch}{1}

\setlength{\aboverulesep}{0pt}
\setlength{\belowrulesep}{0pt}

\arrayrulecolor{black}

\begin{adjustbox}{max width=\linewidth}
\scalebox{0.7}{%
\begin{tabular}{|lccc|}
\toprule

\rowcolor{TablePink}
\pinktablehead{Name}
&
\pinktablehead{Primal-Dual Complexity}
&
\pinktablehead{Parameter-Free}
&
\pinktablehead{AL or PAL Method}
\\

\midrule

This Work (\Cref{alg:O-IAL})
&
$\mathcal{O}\left(
    \epsilon^{-1/2}\log(\epsilon^{-1})
\right)$
&
No
&
AL
\\

This Work (\Cref{alg:OPF-IAL})
&
$\mathcal{O}\left(
    \epsilon^{-1/2}\log(\epsilon^{-1})
\right)$
&
Yes
&
AL
\\

This Work (\Cref{alg:APF-IAL})
&
$\mathcal{O}\left(
    \epsilon^{-1/2}\log(\epsilon^{-1})
\right)$
&
Yes
&
AL
\\

\hline

Li and Xu \cite{AL11Li}
&
$\mathcal{O}\left(
    \epsilon^{-1/2}\log(\epsilon^{-1})
\right)$
&
No
&
AL
\\

Lu and Mei \cite{LuM}
&
$\mathcal{O}\left(
    \epsilon^{-1/2}\log(\epsilon^{-1})
\right)$
&
No
&
PAL
\\

\bottomrule
\end{tabular}%
}
\end{adjustbox}

\endgroup

\caption{Comparison of Methods for Strongly-Convex Linearly-Constrained Optimization}
\label{ALMethodsComparison2}
\end{table}

\subsection{Basic Definitions and Notation}\label{subsec:notation}

\index{Euclidean space, $\E$}
\index{$\E$, Euclidean space}

We introduce the notation used throughout the paper. Let $\mathbb{R}_+$ and
$\mathbb{R}_{++}$ denote the sets of nonnegative and strictly positive real
numbers, respectively. We work in a finite-dimensional inner product space
$\E$ with inner product $\langle\cdot,\cdot\rangle$ and associated norm
$\|\cdot\|$. Let $A:\E\to\mathbb{R}^m$ be a nonzero linear transformation,
and let $\sigma_A^+$ denote its smallest positive singular value.
\index{$\sigma^+_A$, smallest positive singular value}
\index{smallest positive singular value, $\sigma^+_A$}

For a closed convex set $Z\subseteq\E$, we denote its boundary by $\bd Z$,
\index{boundary, $\bd Z$}
\index{$\bd Z$, boundary}
the distance from $z\in\E$ to $Z$ by
\[
    \textdef{$\dist(z,Z):=\inf_{u\in Z}\|z-u\|$},
\]
and its indicator function by $\iota_Z:\E\to(-\infty,+\infty]$, where
\[
    \iota_Z(z):=
    \begin{cases}
        0, & z\in Z,\\
        +\infty, & \text{otherwise}.
    \end{cases}
\]
For $h:\E\to(-\infty,+\infty]$, define
\[
    \dom h:=\{z\in\E:h(z)<\infty\}.
\]
The function $h$ is proper if $\dom h\neq\emptyset$, and
\textdef{$\cConv(\E)$} denotes the class of proper, lower semicontinuous,
convex functions on $\E$.
\index{indicator function, $\iota_Z$}
\index{$\iota_Z$, indicator function}
\index{proper function $h$}

For a proper function $h$, its $\varepsilon$-subdifferential at $z\in\E$ is
\[
    \partial_\varepsilon h(z):=
    \left\{
        u\in\E:
        h(z')\ge h(z)+\langle u,z'-z\rangle-\varepsilon,
        \quad \forall z'\in\E
    \right\},
\]
and $\partial h(z):=\partial_0h(z)$. For a closed convex set $C$ and
$z\in C$, its $\varepsilon$-normal cone is
\[
    \textdef{$N_C^\varepsilon(z)$}:=
    \{\xi\in\E:\langle \xi,u-z\rangle\le\varepsilon,\quad \forall u\in C\},
\]
and $N_C(z):=N_C^0(z)$.

If $f:\E\to\mathbb{R}$ is differentiable, its affine approximation at
$\bar z$ is
\begin{equation}\label{eq:defell}
    \ell_f(z;\bar z):=
    f(\bar z)+\inner{\nabla f(\bar z)}{z-\bar z},
    \quad \forall z\in\E .
\end{equation}
If $f$ is convex, then $\ell_f(\cdot;\bar z)$ is the unique affine minorant
of $f$ at $\bar z$.
\index{affine approximation, $\ell_f(z;\bar z)$}
\index{$\ell_f(z;\bar z)$, affine approximation}

For $t>0$ and $b>1$, define
\begin{equation}\label{logarithmicNotation}
    \log_b^+(t):=\max\{\log_b(t),0\},
    \qquad
    \log_b^{++}(t):=\max\{\log_b(t),1\}.
\end{equation}
When no ambiguity arises, $\log(t)$ denotes the natural logarithm.

We let the Lagrangian and dual functions of \Cref{LinearConstraints} be
defined respectively as 
\begin{equation}\label{Lagrangian Function}
  \textdef{$\cL(z,p):=\psi(z)+\left\langle p,\cA z-b\right\rangle$},
\quad \textdef{$D(p):=\min_z \cL(z,p)$}.
\end{equation}
Then the dual problem and weak duality of \Cref{LinearConstraints} are
\begin{equation}
\label{dual}
\pLC \geq d_{LC}^{*}:=\max_p D(p).
\end{equation}

Let $\mathcal{P}_\star$ denote the set of optimal dual solutions of
problem~\cref{LinearConstraints}. For $p\in\mathbb{R}^m$, define the squared
distance to dual optimality by
\begin{equation}\label{DualDomainNew}
    \Pi(p):=\dist^2(p,\mathcal{P}_\star)
    =\inf_{p_\star\in\mathcal{P}_\star}\|p-p_\star\|^2.
\end{equation}
\index{distance to dual optimality, $\Pi(p)$}
\index{$\Pi(p)$, distance to dual optimality}
\index{smallest positive singular value, $\sigma^+_Q$}
\index{$\sigma^+_Q$, smallest positive singular value}
\index{boundary, $\partial Z$}
\index{$\partial Z$, boundary}

\section{Linearly-Constrained Convex Optimization}
\index{inexact augmented Lagrangian methods}
\phantomsection

In this section, we develop inexact augmented Lagrangian (AL) methods
for solving the linearly constrained convex optimization problem
\cref{LinearConstraints} under the assumption that $\psi$ is convex.
More specifically, we design inexact AL methods that compute approximate
solutions satisfying suitable primal-dual optimality conditions for the
pair \cref{LinearConstraints,dual}, while attaining optimal complexity
guarantees.

To establish convergence results, we impose the following standard assumption.

\begin{assump}
\label{ass:A1} 

\index{$(z_*,p_*)$, optimal primal-dual pair}
There exists an \textdef{optimal primal-dual pair, $(z_*,p_*)$},
satisfying the first-order primal-dual optimality conditions
\begin{align}\label{Optimal Primal-Dual}
  0\in \nabla \psi_s(z_*) + \partial \psi_n(z_{*})+\mathcal{A}^*p_*
\text{  (d.f.)}, \qquad  \mathcal{A}z_*-b = 0\text{  (p.f.)}.
\end{align}
Here, the multiplier $p_{*}$ attains the optimal dual value $d^{*}_{LC}$
in \Cref{dual}, and $\cA^*$ denotes the adjoint of the linear operator
$\mathcal A$.
\end{assump}

\index{tolerances}
\index{tolerances!\Cref{LinearConstraints} pair $(\hat \epsilon,\hat \rho)$}
\index{$(\hat \epsilon,\hat \rho)\in \mathbb R_{++}^2$, tolerance pair}
\index{$(\hat \epsilon, \hat \rho)$-approximate optimal solution}

We next define the notion of approximate solution used by our inexact AL
methods throughout the paper.

\begin{definition}\label{ApproxOptSolution}
Given problem \Cref{LinearConstraints} under \Cref{ass:A1}, and a
\textdef{tolerance pair, $(\hat \epsilon,\hat \rho)\in \mathbb R_{++}^2$},
we say that a triple $(z,p,r)$ is a
\textbf{$(\hat \epsilon, \hat \rho)$-approximate primal-dual optimal solution}
if it approximately satisfies the primal-dual optimality conditions in
\cref{Optimal Primal-Dual}, i.e., the triple satisfies
\begin{gather}\label{ALsolutiontype}
    r \in \nabla \psi_s(z)+\partial \psi_n(z)+\cA^{*}p, \quad \|r\|\leq
\hat \rho, \quad \|\cA z-b\|\leq \hat \epsilon.
\end{gather}
\end{definition}

Condition \cref{ALsolutiontype} encodes approximate dual stationarity and
primal feasibility.


To approximately minimize \Cref{lagrangian2}, our inexact AL methods employ PF-AR, the optimal parameter-free method developed in \Cref{PF-AR}. PF-AR nontrivially extends the parameter-free method of Lan et al.~\cite{AL27Lan} from unconstrained minimization to convex composite optimization by using a new line-search mechanism and a different termination criterion.

The remainder of this section is organized as follows. \Cref{PF-AR} presents 
the parameter-free AL subproblem solver, PF-AR,
together with its optimal complexity guarantees for convex composite
minimization.
\Cref{OIALMethod} presents O-IAL, an optimal dual-type inexact AL method that maintains approximate dual stationarity at each iteration, and proves its optimal best-iterate guarantees for \Cref{LinearConstraints}. O-IAL is not parameter-free as it requires knowledge of the diameter $D$ of $\Dn$.
\Cref{OPF-IAL} then presents OPF-IAL, an essentially restarted version of O-IAL that is parameter-free and requires no knowledge of $D$ or any other problem parameter while preserving optimal best-iterate guarantees. Finally, \Cref{APFIAL} presents APF-IAL, a more practical adaptive parameter-free method that chooses the penalty parameter and subproblem tolerances adaptively and achieves optimal last-iterate guarantees under assumptions only mildly stronger than those required for O-IAL and OPF-IAL.

\subsection{PF-AR Method: Parameter-Free Convex Subproblem Solver}\label{PF-AR}
\phantomsection
This section presents a parameter-free accumulative regularization (AR)
method, called PF-AR, which will be used by O-IAL, OPF-IAL, and APF-IAL, to approximately solve their AL subproblems. More generally, PF-AR solves the unconstrained variant of the composite optimization problem in
\Cref{LinearConstraints}, i.e.,
\begin{equation}\label{ARoptProblem}
\min \left\{\psi(x):=\psi_s(x)+\psi_n(x)\right\},
\end{equation}
where $\psi_s$ is convex and $\bar L$-smooth, and $\psi_n$ is a closed proper convex function with compact domain $\Dn$ of diameter $D$. We also assume that the proximal mapping of $\psi_n$ can be computed efficiently. More specifically, for any $\sigma,\eta\geq 0$ and any $x,\bar x\in \E$, we assume that the optimization problem
\[
x^{++}(\eta,\sigma,\bar x;x):=\argmin_{u}
\inner{\nabla \psi_s(x)}{u}+\psi_n(u)
+\frac{\sigma}{2}\|u-\bar x\|^2
+\frac{\eta}{2}\|u-x\|^2
\]
can be solved efficiently,
where recall
that $\E$ is a finite-dimensional inner product space. 

The proposed PF-AR algorithm can be viewed as an extension of the
parameter-free accelerated regularization (AR) method of \cite{AL27Lan},
originally developed for unconstrained problems, to the composite setting.
This extension, however, is nontrivial. In particular, PF-AR employs a new
backtracking line-search mechanism for estimating the Lipschitz constant,
which differs fundamentally from the parameter-free strategy in
\cite{AL27Lan}. This modification is essential for obtaining optimal
convergence guarantees in terms of the norm of a subgradient of $\psi$,
which is the stationarity criterion required in our augmented Lagrangian
analysis.

This stopping criterion also distinguishes our approach from the
parameter-dependent composite AR method in \cite{AL27Lan}, which terminates
upon finding $(x,\eta)$ such that $\|G_\eta(x)\|$ is small, where
\begin{equation}\label{Gradient Mapping}
G_{\eta}(x):=\eta(x-x^{+}(\eta;x)), 
\quad   
x^{+}(\eta;x):=\argmin_{u}
\inner{\nabla \psi_s(x)}{u}+\psi_n(u)
+\frac{\eta}{2}\|u-x\|^2.
\end{equation}
In contrast,
PF-AR terminates when $\|v\|$ is small for some
$
v \in \nabla \psi_s(x)+\partial \psi_n(x).
$
To obtain this guarantee, our method enforces two key inequalities through a
backtracking line-search procedure at every iteration. Indeed, the PF-AR method is a multi-loop algorithm that is fully parameter-free: it
requires no prior knowledge of either the diameter $D$ or the Lipschitz
constant $\bar L$. To remove the need to know $D$, PF-AR repeatedly calls
a second algorithm, AR-L, with different guesses for the diameter. The AR-L
method itself requires no knowledge of $\bar L$; instead, each of its
iterations uses the new line-search mechanism described above. Since this
line-search is central to both PF-AR and AR-L, we present it first.

\begin{algorithm*}[h]
\caption{Backtracking($\psi_s$, $\psi_n$, $\sigma$, $x$, $\bar x$, $M_0$). \label{alg:_searchL}}
	\begin{algorithmic}[1]
		\FOR{$j=0,1,\ldots,$}\label{Back1}
        \STATE Set $x^{++}:=x^{++}(2M_{j},\sigma,\bar x;x):=\argmin_{u} \inner{\nabla \psi_s(x)}{u}+\psi_n(u)+\frac{\sigma}{2}\|u-\bar x\|^2+M_{j}\|u-x\|^2$ \label{Back2}
        \vspace{0.3em}
		\STATE Set
		$ x^{+}:=x^{+}(2M_j;x):=\argmin_{u} \inner{\nabla \psi_s(x)}{u}+\psi_n(u)+M_{j}\|u-x\|^2
		$\label{Back3}
        \vspace{0.3em}
		\STATE If \label{Back4}
		\begin{align}
			&\psi_s(x^{++})-\psi_s(x) -\inner{\nabla \psi_s(x)}{x^{++}-x}\leq \frac{M_j}{2}\|x^{++}-x\|^2, \label{linesearch1}\\
            &\|\nabla \psi_s(x^{+})-\nabla \psi_s(x)\|\leq M_j\|x^{+}-x\|,\label{linesearch2}
		\end{align}
		then {\bf terminate} with $x^{++}$, $x^{+}$, and $M=M_j$. Otherwise, set $M_{j+1} = 2M_{j}$. 
		\ENDFOR \label{Back5}
        \OUTPUT triple $(x^{++},x^{+},M)$.\label{BackOutput}
	\end{algorithmic}
\end{algorithm*}

As discussed above, AR-L uses \Cref{alg:_searchL} as a subroutine and does
not require prior knowledge of $\bar L$. It is an inexact proximal-point
scheme that approximately minimizes a sequence of proximal subproblems. More
specifically, as in \cite{AL27Lan}, AR-L computes a sequence of outer
iterates $x_s$ satisfying, approximately,
\begin{equation}\label{PhiMin}
x_s \approx
\argmin_u
\left\{
\phi_s(u):=\psi_s(u)+\psi_n(u)+\frac{\sigma_s}{2}\|u-\bar x_s\|^2
\right\},
\end{equation}
by using $x_{s-1}$ as an initial point and $N_s$ inner iterations.

Following \cite{AL27Lan}, we assume the availability of an accelerated
algorithm $\mathcal A$ for approximately solving these proximal subproblems.
Many accelerated methods in the literature satisfy the required property such as the R-FISTA method presented in \Cref{Restarted FISTA} and the method in \cite{AL58Nesterov}. We state it formally in the next assumption.

\begin{assump}\label{Assumption1}
    Algorithm $\mathcal A$ used to compute $x_{s}=\mathcal A(\psi_s,\psi_n,\sigma_{s},\bar x_{s},x_{s-1})$ (see \Cref{eq:subproblem_unconstrained_searchL} in algorithm below) to approximately minimize $\phi_s$
has the following performance guarantees after the $k$-th evaluation of $\nabla \psi_s$:
\begin{equation}\label{assumptionA}
\phi_s(x_s^{k})-\phi_s(x_s^{*})\leq \frac{L_s^{k}}{k^2}\|x_{s-1}-x_{s}^{*}\|^2.
\end{equation}
Here $x_s^{k}$ is the computed approximate solution and $L_s^{k}$
is an estimate of Lipschitz constant $\bar L$ such that 
$L_s^{k}\leq c_{\mathcal A}\bar L$
where $c_{\mathcal A}$
is a universal constant that depends on algorithm $\mathcal A$.
We further assume that
subroutine $\mathcal A$ is terminated whenever
\begin{equation}\label{ATermin}
    k\geq 8\sqrt{\frac{2 L_s^{k}}{\sigma_{s}}}.
\end{equation}
When this occurs, we output
the approximate solution $x_s=x_{s}^{k}$.
\end{assump}

We now formally present the AR-L method for solving \Cref{ARoptProblem} which requires no knowledge of $\bar L$.

\begin{algorithm}[!h]
	\caption{\label{alg:ARL} AR-L$(\psi_s,\psi_n,x_0,\sigma_1, M_0)$.
}
	\begin{algorithmic}[1]
        \REQUIRE function pair $(\psi_s,\psi_n)$, initial point $x_0\in \Dn$, initial regularization parameter $\sigma_1>0$, and 
        initial Lipschitz estimate $M_0>0$.\label{ARLInput}
        \STATE $s\gets 0$ \label{ARL1}
        \STATE $\text{stopcrit } \gets \infty$ \label{ARL2}
		\WHILE{$\text{stopcrit } >0$}\label{ARL3}
        \STATE set $s\gets s+1$; \label{ARL4}
        \IF{$s=1$} \label{ARL5}
        \STATE{$\bar x_s=x_0$} \label{ARL6}
        \ELSE \label{ARL7}
        \STATE{set $\sigma_s=4\sigma_{s-1}$ and $\xb{s} =  (1-\gamma_s)\xb{s-1} + \gamma_s x_{s-1}$ with $\gamma_s=1 - \sigma_{s-1}/\sigma_s$}; \label{ARL8}
        \ENDIF \label{ARL9}
		
		\STATE compute an approximate solution $x_s\in \Dn$ of
\begin{align}\label{eq:subproblem_unconstrained_searchL}
			x_s^*:=\argmin_{x\in \Dn} 
            \left\{\phi_s(x):=\psi_s(x)+\psi_n(x) + \dfrac{\sigma_s}{2}\|x - \bar x_{s}\|^2\right\}
		\end{align}
		by running subroutine $x_s = \cA(\psi_s, \psi_n, \sigma_s, \bar x_{s}, x_{s-1})$ with initial point $x_{s-1}$;\label{ARL10}	
		\STATE set $(\xpp{s}, \xp{s}, M_s)=$ Backtracking ($\psi_s$, $\psi_n$, $\sigma_s$, $x_s$, $\bar x_s$, $M_{s-1}$); \label{ARL11}
        \STATE set $v_{s}:=2M_s\left(x_s-\xp{s}\right)-\nabla \psi_s(x_s)+\nabla \psi_s(\xp{s})$; \label{ARL12}
		\STATE set $\text{stopcrit}=M_s-\sigma_s$ \label{ARL13}
		\ENDWHILE
        \OUTPUT triple $(\hat x,\hat v,\hat M):=(\xp{s},v_s,M_s)$\label{ARLOut}.
	\end{algorithmic}
\end{algorithm}

In the analysis of AR-L, we set $\sigma_0=0$ and $\bar x_0=x_0$. The main complexity result for the AR-L method is stated in \Cref{AR-LComplex}. 
Due to the length of
its proof, the proof of the theorem is given in \Cref{ProofAR-L}.

\begin{theorem}\label{AR-LComplex}
The AR-L method outputs a triple $(\hat x,\hat v,\hat M)$ satisfying
\begin{align}
    &\hat v\in \nabla \psi_s(\hat x)+\partial \psi_n(\hat x),
    \qquad \|\hat v\|\leq 7\sigma_1D, \label{Output1}\\
    &M_0\leq \hat M\leq \max\left\{M_0,\,2\bar L\right\}. \label{Output3}
\end{align}
Moreover, if $M_0\leq  2 \bar L$, then such a triple is obtained within at most
\begin{equation}\label{CompResultARL}
    6
    +C_1\sqrt{\frac{\bar L}{\sigma_1}}
    +\log_{2}\left(\frac{2\bar L}{M_0}\right)
\end{equation}
gradient evaluations, where
\begin{equation}\label{C1Constant}
    C_1:=2\sqrt{2}+16\sqrt{2c_{\mathcal A}},
\end{equation}
and $c_{\mathcal A}$ is the universal constant in \Cref{Assumption1}.

Finally, if the input is chosen as $\sigma_1=\epsilon/(7\overline D)$ for some scalar
$\overline D>0$, then AR-L outputs $(\hat x,\hat v,\hat M)$ satisfying
\Cref{Output1} and \Cref{Output3}, with
$
    \|\hat v\|\leq (\epsilon D)/\overline D
$
within at most
$
    6
    +\sqrt{7}C_1\sqrt{\frac{\bar L\overline D}{\epsilon}}
    +\log_{2}\left(\frac{2\bar L}{M_0}\right)
$
gradient evaluations.
\end{theorem}

\begin{remark}
The assumption that $M_0\leq  2 \bar L$ is not strictly necessary to establish a complexity result similar to \Cref{CompResultARL}. It is imposed only to simplify the analysis and the resulting complexity bound of the AR-L method. It is easy to construct an initial estimate $M_0$ satisfying this assumption. Indeed, if $\psi_s$ is not affine and we choose $w_0\neq x_0$ with $\nabla \psi_s(x_0)\neq \nabla \psi_s(w_0)$, where $x_0\in \Dn$ is the initial point, and define
$
M_0:=\|\nabla \psi_s(x_0)-\nabla \psi_s(w_0)\|/(4\|x_0-w_0\|),
$
then it follows from the definition of $\bar L$ that $M_0\leq  2 \bar L$. In the affine case, since any positive scalar $\bar L$ is a valid
smoothness constant for $\psi_s$, we may choose any $M_0>0$ and then
take $\bar L$ large enough so that $M_0\leq 2\bar L$.
\end{remark}

As follows from \Cref{AR-LComplex}, the AR-L method requires the knowledge of $D$ in order to compute a point whose subgradient norm is at most $\epsilon$. 

We next present the PF-AR method, which requires no prior knowledge of either $\bar L$ or $D$ to compute a point whose subgradient norm is at most $\epsilon$. That is, given functions $\psi$, $\psi_s$, and $\psi_n$ as in \Cref{ARoptProblem}, and a tolerance parameter $\epsilon > 0$, the goal of
PF-AR is to find a pair $(\hat x, \hat v) \in \Dn \times \E$ such that
\begin{gather}\label{PFARProb}
\|\hat v\| \le \epsilon, \quad \hat v \in \nabla \psi_s(\hat x) + \partial \psi_n(\hat x).
\end{gather}
The PF-AR method repeatedly invokes AR-L with different guesses for $D$ and is guaranteed to compute a pair $(\hat x,\hat v)$ in at most $\mathcal O(1/\epsilon)$ gradient evaluations. The method is now presented in \Cref{alg:PF-AR} below.

\begin{algorithm}[!h]
	\caption{\label{alg:PF-AR} PF-AR Method}
	\begin{algorithmic}[1]
\REQUIRE function pair $(\psi_s,\psi_n)$, initial point $x_0\in \Dn$,
initial smoothness estimate $M_0>0$, initial diameter $D_0>0$ estimate, and tolerance $\epsilon>0$.\label{PFARInput}
\STATE $t\gets 0$\label{PFAR1}
\STATE $\text{stopcrit } \gets \infty$\label{PFAR2}
\WHILE{$\text{stopcrit }>\epsilon$}\label{PFAR3}
\STATE set $t\gets t+1$;\label{PFAR4}
		\STATE set $D_t=4D_{t-1}$\label{PFDiamstep};
		\STATE set $(v_t,x_t, M_t)=\ $AR-L$(\psi_s,\psi_n, x_{t-1}, \epsilon/(7D_t), M_{t-1})$ (see \Cref{alg:ARL}) \label{ARLcallstep};
		\STATE set $\text{stopcrit } \gets \|v_t\|$\label{PFARtermstep};
\ENDWHILE\label{PFAR8}
\OUTPUT pair $(\hat x,\hat v)=(x_t,v_t)$ satisfying \Cref{PFARProb}.\label{PFAROutput}
	\end{algorithmic}
\end{algorithm}

The main complexity result of the PF-AR method is stated in the following theorem.
\begin{theorem}\label{PFARComplexity}
Suppose that inputs $M_0$ and $D_0$ satisfy
$M_0\leq 2\bar L$ and $D_0\leq D$. The PF-AR method then computes a pair $(\hat x,\hat v)$ such that 
    \begin{align}
    &\hat v\in \nabla \psi_s(\hat x)+\partial \psi_n(\hat x), \quad \|\hat v\|\leq \epsilon\label{OutputPF-AR}
   \end{align}
within at most
\begin{equation}\label{ComplexityPFAR}
 \left[6+\log_2(2\bar L/M_0)\right]\left[1+\left \lceil \log_{4}\left(\frac{D}{D_0} \right)\right \rceil\right]+8\sqrt{7}C_1\sqrt{\frac{\bar LD}{\epsilon}}
\end{equation}
gradient evaluations where $C_1$ is a universal constant as in \Cref{C1Constant}.
Hence, the PF-AR method achieves the optimal gradient evaluation complexity bound of 
$\mathcal O\left(\sqrt{\bar LD/\epsilon}\right)$
for finding a pair $(\hat x,\hat v)$ satisfying \Cref{OutputPF-AR}.
\end{theorem}
\begin{proof}
It follows from step~\ref{ARLcallstep} of the PF-AR Method that
during the $t$-th iteration, PF-AR calls the AR-L method 
with inputs 
$(\psi_s,\psi_n,x_{t-1},\epsilon/(7D_t),M_{t-1})$
and outputs a triple
$(v_t,x_t, M_t)$. It then follows from 
\Cref{Output1} and \Cref{Output3}
that the triple $(v_t,x_t,M_t)$
satisfies
\begin{align}
    &v_t\in \nabla \psi_s(x_t)+\partial \psi_n(x_t),
    \qquad \|v_t\|\leq \frac{\epsilon D}{D_t}, \label{OutputARL1}\\
    &M_{t-1}\leq M_t\leq \max\left\{M_{t-1},\,2\bar L\right\}. \label{OutputARL2}
\end{align}
Let 
\begin{equation}\label{OuterComplexityPFAR}
T:=1+\left \lceil \log_{4}\left(\frac{D}{D_0} \right)\right \rceil.
\end{equation}
It then 
follows from 
the second bound in \Cref{OutputARL1}, the bound $D_0\leq D$, the fact that step~\ref{PFDiamstep} implies that $D_t=4^{t}D_0$, the way stopcrit is set in step~\ref{PFARtermstep}, and the condition of the while loop that PF-AR performs at most $T$ outer iterations where $T$ is as in \Cref{OuterComplexityPFAR}. 

To show the total gradient evaluation complexity result,
we first show that for all $t\geq 0$
\begin{equation}\label{MBound}
M_{0}\leq M_t\leq 2\bar L.
\end{equation}
We proceed by induction. It follows from the assumption
$M_0\leq \bar 2L$ of \Cref{PFARComplexity} that both inequalities
in \Cref{MBound}
must hold for $t=0$.
Let $t \geq 1$
and assume now that
$M_0\leq M_{t-1}\leq 2\bar L$.
It then follows from the second inequality in
\Cref{OutputARL2} and the induction hypothesis
that 
\[M_{t}\overset{\Cref{OutputARL2}}{\leq} \max\left\{M_{t-1}, 2\bar L\right\}\leq 2\bar L.\]
Moreover, it follows from the first inequality in 
\Cref{OutputARL2} and the induction hypothesis
that 
\[M_0\leq M_{t-1}\leq M_{t}.\]
Hence, it follows by induction that \Cref{MBound}
holds for all $t \geq 0$.

It then follows from \Cref{MBound},
the fact that PF-AR calls the AR-L method 
with inputs 
as in step~\ref{ARLcallstep}, and
the complexity bound in \Cref{CompResultARL}
that the call made to the AR-L 
method in the $t$-th
iteration 
performs 
at most 
\begin{equation}
 6+C_1\sqrt{\frac{7\bar LD_t}{\epsilon}}
    +\log_{2}\left(\frac{2\bar L}{M_{t-1}}\right)\overset{\Cref{MBound}}{\leq} 6+C_1\sqrt{\frac{7\bar L D_t}{\epsilon}}
    +\log_{2}\left(\frac{2\bar L}{M_{0}}\right)
\end{equation}
gradient evaluations.
The total number of gradient evaluations is then bounded by
\begin{align*}
&\left[6+\log_2(2\bar L/M_0)\right]T+C_1\sum_{t=1}^{T}\sqrt{\frac{7\bar L D_t}{\epsilon}}= \left[6+\log_2(2\bar L/M_0)\right]T
+C_1\sqrt{\frac{7\bar L}{\epsilon}}\sum_{t=1}^{T}\sqrt{4^{t}D_0}\\
&\leq 
\left[6+\log_2(2\bar L/M_0)\right]T+
C_1\sqrt{\frac{7\bar L}{\epsilon}}\sqrt{D_0}2^{T+1}\overset{\Cref{OuterComplexityPFAR}}{\leq} 
\left[6+\log_2(2\bar L/M_0)\right]T+
8C_1\sqrt{7}\sqrt{\frac{\bar LD}{\epsilon}}
\end{align*}
which implies the complexity result in 
\Cref{ComplexityPFAR} in view of the definition of $T$ in \Cref{OuterComplexityPFAR}.
The last conclusion of \Cref{PFARComplexity} is immediate from \Cref{ComplexityPFAR}.
\end{proof}

\begin{remark}\label{remarkPFAR}
The assumptions in \Cref{PFARComplexity} that the inputs $M_0$ and $D_0$ satisfy
$M_0\leq 2\bar L$ and $D_0\leq D$ are not strictly necessary for establishing
PF-AR's complexity result. However, they simplify the analysis and the resulting
complexity bound. Moreover, both conditions can be easily satisfied. As mentioned
in the remarks following the AR-L method, it is easy to construct an $M_0$ such
that $M_0\leq 2\bar L$. To construct a $D_0$ satisfying $D_0\leq D$, select
$w_0\neq x_0$ with $w_0\in\Dn$ and define
$
D_0:=\|x_0-w_0\|.
$
It then follows from the definition of the diameter $D$ of $\Dn$ that
$D_0\leq D$.
\end{remark}

\subsection{Best Iterate Convergence of an Optimal Augmented
Lagrangian Method}\label{OIALMethod}

In this subsection, we present an  \textbf{inexact augmented Lagrangian
method}, referred to as the \textbf{optimal inexact augmented Lagrangian
(O‑IAL) method}, for solving the linearly constrained convex problem
presented in \Cref{LinearConstraints}.
The objective of O-IAL is to compute a primal-dual pair that 
satisfies the approximate 
primal-dual optimality conditions
in \Cref{ALsolutiontype}, while achieving optimal first-order complexity bounds.

O-IAL can be viewed as an inexact AL method that maintains
approximate dual stationarity at each iteration.
Given a primal-dual pair ($z_{k-1},p_{k-1})$, an
O-IAL iteration uses the PF-AR method 
described in 
\Cref{PF-AR}
to compute its next primal iterate, $z_k$, by approximately minimizing
the convex AL subproblem 
\begin{equation}\label{eq:approx_primal_update}
\min_{z} {\cal L}_{c}(z;p_{k-1}),
\end{equation}
with accuracy less than or equal to $\hat \rho>0$.
O-IAL thus maintains approximate dual stationarity, as specified by the dual residual condition 
$\|r_k\|\leq \hat \rho$ in \cref{ALsolutiontype}.
After approximately minimizing \Cref{eq:approx_primal_update},
O-IAL then performs the dual update
according to \Cref{eq:dual_update}.

\Cref{alg:O-IAL} now formally presents the \textdef{O-IAL method}. 
\begin{algorithm}[H]
\caption{O-IAL Method}
\label{alg:O-IAL}
\begin{algorithmic}[1]
\REQUIRE 
function pair $(\psi_s,\psi_n)$, initial points $z_0 \in \Dn$ and
$p_0\in \R^m$, primal-dual tolerance pair
$(\hat \epsilon,\hat \rho) \in \R_{++}^2$, penalty parameter $c>0$,
PF-AR tolerance
$
    \epsilon=\min\left\{\frac{c\hat\epsilon^2}{4D},\hat\rho\right\},
$
and smoothness and diameter estimates
$M_0:=\|\nabla \psi_s(z_0)-\nabla \psi_s(w_0)\|/(4\|z_0-w_0\|)$
and $D_0:=\|z_0-w_0\|$, where $w_0\in \Dn$ is chosen as in
\Cref{remarkPFAR}.
\label{InputOIAL}
\STATE $k \gets 0$\label{OIAL1}
\STATE call the PF-AR method (\Cref{alg:PF-AR}) with function pair
$(\cL^s_c(\cdot,p_k),\psi_n)$, initial point $z_k$,\\
\qquad estimates $M_0$ and $D_0$, and tolerance $\epsilon$; and let
$(z_{k+1},r_{k+1})$ be its output;\label{OIAL2}
\STATE set:  $Res_{k+1} \gets \mathcal Az_{_{k+1}}-b$;\label{OIAL3}
\STATE set: $p_{{k+1}}=p_{k}+c Res_{k+1}$;\label{OIAL4}
\STATE update: $k \gets k+1$;\label{OIAL5}
\WHILE{$\|\Resk\| >  \hat \epsilon$}\label{OIAL6}
\STATE call the PF-AR method (\Cref{alg:PF-AR}) with function pair
$(\cL^s_c(\cdot,p_k),\psi_n)$, initial point $z_k$,\\
\qquad estimates $M_0$ and $D_0$, and tolerance $\epsilon$; and let
$(z_{k+1},r_{k+1})$ be its output; \label{OIAL7}
\STATE set:  $Res_{k+1} \gets \mathcal Az_{_{k+1}}-b$;\label{OIAL8}
\STATE set:  $p_{{k+1}}=p_{k}+c Res_{k+1}$;\label{OIAL9}
\STATE update: $k \gets k+1$;\label{OIAL10}
\ENDWHILE\label{OIAL11}
\OUTPUT triple $(z,p,r):=(z_k,p_{k},r_k)$ satisfying
\cref{ALsolutiontype}.\label{OIALOutput}
\end{algorithmic}
\end{algorithm}
\index{tolerances!O-IAL $\epsilon=\min\left\{\frac{c\hat
\epsilon^2}{4D},\hat\rho\right\}$}

It will be shown in \Cref{ARtranslation}(b) of the next subsection
that every call to the PF-AR method produces a pair $(z_k,r_k)$ satisfying
the approximate dual stationarity conditions in
\cref{ALsolutiontype}. Consequently,
when the while loop terminates, the output triple
$(z_k,p_k,r_k)$ satisfies all conditions in \cref{ALsolutiontype}.

\Cref{Main Complexity O-IAL} provides the main complexity result for O-IAL. The
proof of the theorem is given in a collection of propositions and
lemmas in \Cref{Proof of Theorem O-IAL}.

\begin{theorem}\label{Main Complexity O-IAL}
Suppose that \Cref{ass:A1} holds. Recall that $\Pi(\cdot)$ is as in \Cref{DualDomainNew}, $\bar L$ is the smoothness parameter of $\psi_s$, and $D$ is the diameter of $\Dn$. 
The following statements then hold.
    \begin{enumerate}[label=(\alph*)]
        \item The O-IAL method requires at most  \begin{equation}\label{Complexity a}
 \mathcal O\left( \frac{\Pi(p_0)}{\hat
\epsilon^2c^2}   \sqrt{\left(\bar L+c\|\cA\|^2\right) \max\left\{\frac{D^2}{c\hat \epsilon^2}, \frac{D}{\hat \rho} \right\}} \right)
        \end{equation}

resolvent/gradient evaluations to find a $(\hat \epsilon,\hat \rho)$-approximate
solution of \cref{LinearConstraints}, i.e., a triple $(z,p,r)$ satisfying
\cref{ALsolutiontype}.
\item By choosing the penalty parameter $c$ as $c=\mathcal O(1/\hat \epsilon)$, the above complexity bound is reduced to the optimal gradient evaluation complexity bound
\begin{equation}\label{Complexity b}
\mathcal O\left(\Pi(p_0)  \sqrt{\left(\bar L+ \frac{\|\cA\|^2}{\hat \epsilon} \right) \max\left\{\frac{D^2}{\hat \epsilon}, \frac{D}{\hat \rho} \right\}} \right).
 \end{equation}
Hence, if $\hat\epsilon=\hat\rho$, it follows that the O-IAL method requires at most $\mathcal
O\left(1/\hat \epsilon\right)$ total gradient evaluations to find an $(\hat \epsilon,\hat \epsilon)$-approximate solution of \cref{LinearConstraints}.
    \end{enumerate}
\end{theorem}

\subsubsection{Proof of \texorpdfstring{\Cref{Main Complexity O-IAL}}{Lg}}\label{Proof of Theorem O-IAL}
This subsection proves \Cref{Main Complexity O-IAL}.
We begin with the following proposition, which characterizes the
smoothness properties of the augmented Lagrangian subproblem and the
behavior of the PF-AR subroutine
used in step~\ref{OIAL7} of O-IAL.

\begin{prop}\label{ARtranslation}
   The following statements about the $(k-1)$-st iteration of O-IAL hold:
    \begin{enumerate}[label=(\alph*)]
       \item 
\label{item:propFISTA}
the function $\mathcal L^s_{c}(\cdot,p_{k-1})$ is
convex and $(\bar L+c\|\cA\|^2)$-smooth;
        \item the call made to the PF-AR method in step~\ref{OIAL7} outputs a pair $(z_k,r_k)$ that satisfies the following relations
        \begin{equation}\label{inclusion Lagrangian}
	 r_{k}\in \nabla \psi_s(z_k)+\partial\psi_n(z_k)+\cA^{*}p_{k}, \quad \|r_k\|\leq \min\left\{\frac{c\hat \epsilon^2}{4D},\hat \rho\right\},
        \end{equation}
        within at most 
        \begin{equation}\label{OIALInComplexity}
        \mathcal O\left(\sqrt{\left(\bar L+c\|\cA\|^2\right) \max\left\{\frac{D^2}{c\hat \epsilon^2}, \frac{D}{\hat \rho} \right\}}\right)
        \end{equation}
        gradient evaluations.
\end{enumerate}

\end{prop}

\begin{proof}
We first show part a). The fact that $\mathcal L^s_{c}(\cdot;p_{k-1})$ is convex
is immediate from its definition in \cref{smooth part} and the fact that
$\psi_s$ is convex. To show the second result, let $z$ and $z'$ be in
$\mathbb E$. The definition of $\mathcal L_{c}^{s}(\cdot;p_{k-1})$ in
\cref{smooth part}, the fact that $\psi_s$ is $\bar L$-smooth, and the
triangle inequality, then imply that
 \begin{align*}
     \|\nabla\mathcal L^s_{c}(z;p_{k-1})-\nabla\mathcal
L^s_{c}(z';p_{k-1})\|
     &\leq \|\nabla \psi_s(z)-\nabla
\psi_s(z')\|+ 
	\|c\cA^{*}[(\cA z-b)-(\cA z'-b)]\|\\
     &\leq \bar L\|z-z'\|+\|c\cA^{*}\cA(z-z')\|\\
     &\leq (\bar L+c\|\cA\|^2)\|z-z'\|.
 \end{align*}

(b) 
It is easy to see from the definition of $\mathcal L_c^{s}(z_k;p_{k-1})$ in \cref{smooth part} and the update formula for $p_k$ in step~\ref{OIAL9} that
\begin{align}\label{gradient formula}
\nabla \mathcal L_c^{s}(z_k;p_{k-1})=\nabla
\psi_s(z_k)+\cA^{*}(p_{k-1}+c(\cA z_k-b))=\nabla \psi_s(z_k)+\mathcal A^{*}p_k.
\end{align}
It then follows from this observation,
the fact that O-IAL calls PF-AR
in its step \ref{OIAL7}
with function pair $(\mathcal L_{c}^{s}(\cdot,p_k),\psi_n)$ and tolerance $\epsilon=\min\left\{(c\hat \epsilon^2)/4D,\hat \rho\right\}$,
and the inclusion in \Cref{OutputPF-AR}
that PF-AR outputs a pair $(z_k,r_k)$
satisfying
\[
 r_{k}\in \nabla \psi_s(z_k)+\partial\psi_n(z_k)+\cA^{*}p_{k}, \quad \|r_k\|\leq \min\left\{\frac{c\hat \epsilon^2}{4D},\hat \rho\right\},
\]
which implies 
that \Cref{inclusion Lagrangian} holds.

The gradient evaluation complexity bound of the PF-AR method call in \Cref{OIALInComplexity}
is immediate from fact that O-IAL calls PF-AR
in its step \ref{OIAL7}
with function pair $(\mathcal L_{c}^{s}(\cdot,p_k),\psi_n)$ and tolerance $\epsilon=\min\left\{(c\hat \epsilon^2)/4D,\hat \rho\right\}$, part(a), and the last conclusion of 
\Cref{PFARComplexity}.
\end{proof}

The next lemma derives a lower bound on an inner product involving
the primal feasibility residual $\mathcal Az_k-b$. This estimate
will play a key role in the complexity analysis.

\begin{lemma}\label{lem:lowerbndA}
For any iteration index $k \geq 1$ generated by the O-IAL method, it holds that
\begin{equation}\label{inner product bound}
  \inner{\mathcal{A}z_k-b}{p_{*}-p_k}\geq \inner{r_k}{z_{*}-z_k},
\end{equation}
where the pair $(z_{*},p_{*})$ satisfies \cref{Optimal Primal-Dual}.
\end{lemma}
\begin{proof}
It is easy to see that the inclusion in \cref{inclusion Lagrangian} and the fact that $\psi(z):=\psi_s(z)+\psi_n(z)$ is a convex function implies that
$r_k \in \partial \psi(z_k)+\mathcal A^{*}p_k$. This inclusion and the subdifferential inequality at $z_{*}$ imply that
\begin{align}\label{rel 1}
\psi(z_*) - \psi(z_k) \geq \inner{r_k -\mathcal{A}^*p_k}{z_* - z_k}.
\end{align}
Now, suppose that $(z_{*},p_{*})$
is an optimal primal-dual pair satisfying \Cref{ass:A1}. It then follows that $(z_{*},p_{*})$ satisfies the inclusion in \cref{Optimal Primal-Dual}. It then follows from this inclusion and the subdifferential inequality at $z_k$ that
\begin{align}\label{rel 2}
  \psi(z_k) - \psi(z_*) \geq \inner{-\mathcal{A}^*p_*}{z_k - z_*}.
\end{align}
It then follows by adding \cref{rel 1} and \cref{rel 2} that
\begin{equation}\label{monotonicity}
  \inner{\mathcal{A}^{*}(p_{*}-p_k)}{z_k-z_{*}}\geq \inner{r_k}{z_{*}-z_k}.
\end{equation}
The fact that $z_{*}$ satisfies the second relation in \cref{Optimal Primal-Dual} together with
\cref{monotonicity} then imply that
\[\inner{\mathcal{A}z_k-b}{p_{*}-p_k}\overset{\cref{Optimal Primal-Dual}}{=}\inner{\mathcal{A}(z_k-z_{*})}{p_{*}-p_k}=\inner{z_k-z_{*}}{\mathcal{A}^{*}(p_{*}-p_k)}\overset{\cref{monotonicity}}{\geq} \inner{r_k}{z_{*}-z_k},\]
which immediately implies the result.
\end{proof}

The following lemma combines the previous inner-product estimate
with a telescoping argument to derive a bound on the
primal feasibility residuals.

\begin{lemma}\label{lem:ALtelescopic}
For any iteration index $k \geq 1$ generated by O-IAL, it holds that
\begin{equation}
\label{final inequality}
 \sum_{\ell=1}^{k}\|\mathcal{A}z_{\ell}-b\|^2\leq \frac{\|p_{*}-p_0\|^2}{c^2}+\frac{k\hat\epsilon^2}{2}
\end{equation}
where $p_{*}$ is as in
\cref{Optimal Primal-Dual} and $p_0$, $c$, and $\hat \epsilon$ are inputs to O-IAL.
\end{lemma}
\begin{proof}
Suppose that $k\geq 1$ is an index generated by O-IAL and let $\ell\geq 1$ be
an index such that $\ell \leq k$. 
Relation \cref{inner product bound}, Cauchy-Schwarz inequality,
and the fact that second relation in \cref{inclusion
Lagrangian} holds for $k=\ell$ imply that 
\begin{align}\label{final bound on inner product}
\inner{\mathcal{A}z_{\ell}-b}{p_{*}-p_{\ell}}
      \overset{\cref{inner product bound}} \geq
      \inner{r_{\ell}}{z_{*}-z_{\ell}}
      \overset{\cref{diameter}}{\geq}  -D \|r_{\ell}\|\overset{\cref{inclusion Lagrangian}}{\geq} -\frac{c\hat \epsilon^2}{4}
\end{align}
holds for all indices $\ell \leq k$.
Using the above relation, the update for $p_k$ in step~\ref{OIAL9}, and completing the square, we then have that
\begin{align}\label{bound on feasibility}
\|p_*-p_{\ell-1}\|^2-\|p_*-p_{\ell}\|^2
  &=\|p_{\ell-1}-p_{\ell}\|^2 + 2 \inner{p_{\ell}-p_{\ell-1}}{p_*-p_{\ell}} \nonumber\\
&=c^2 \|\mathcal{A}z_{\ell}-b\|^2 + 
2 c \inner{\mathcal{A}z_{\ell}-b}{p_*-p_{\ell}}.
\end{align}

It then follows by summing relation \cref{bound on feasibility} from
$\ell=1$ to $k$ and re-arranging terms that
\[c^2\sum_{\ell=1}^k\|\mathcal Az_{\ell}-b\|^2\leq \|p_*-p_0\|^2+\frac{kc^2\hat\epsilon^2}{2}.\]
The result then immediately follows from dividing both sides of the above relation by $c^2$.
\end{proof}
The following proposition characterizes the number of
outer iterations that O-IAL performs to find a triple $(z,p,r)$
satisfying the approximate primal-dual optimality conditions in
\cref{ALsolutiontype}.

\begin{prop}\label{OutItCompl}
    The O-IAL method performs at most
    \begin{equation}\label{Outer Complexity Result}
        \left\lceil \frac{2\Pi(p_0)}{\hat \epsilon^2c^2} \right\rceil
    \end{equation}
outer iterations to find a triple $(z,p,r)$ satisfying
\cref{ALsolutiontype}, where $\Pi(\cdot)$ is as in \cref{DualDomainNew}, $\hat \epsilon$ is an input tolerance, and $c$ is the input penalty parameter.
\end{prop}
\begin{proof}
It follows from \Cref{ARtranslation}(b) that, for each outer index $k\geq 1$
generated by O-IAL, the call made to the PF-AR method in step~\ref{OIAL7}
outputs a pair $(z_k,r_k)$ satisfying
\[
    r_k\in \nabla \psi_s(z_k)+\partial\psi_n(z_k)+\mathcal A^{*}p_k,
    \qquad
    \|r_k\|\leq \hat \rho .
\]
Hence, by the condition in the while loop of O-IAL, the definition of $Res_k$ in
step~\ref{OIAL8}, and the fact that the above inclusion holds for every
iteration index $k\geq 1$, if O-IAL terminates, then its output
$(z,p,r)=(z_k,p_k,r_k)$ satisfies \cref{ALsolutiontype}.

It remains to show that O-IAL exits the while loop, and therefore terminates,
after at most
\Cref{Outer Complexity Result}
outer iterations. To this end, let $p_*$ be an arbitrary optimal solution of the Lagrangian dual
of \cref{LinearConstraints}. Relation \cref{final inequality} implies that
\[
    k\min_{1\leq \ell \leq k}\|\mathcal A z_\ell-b\|^2
    \leq
    \sum_{\ell=1}^{k}\|\mathcal A z_\ell-b\|^2
    \leq
    \frac{\|p_*-p_0\|^2}{c^2}
    +\frac{k\hat\epsilon^2}{2}.
\]
Dividing both sides by $k$ then gives
\begin{align}\label{BoundFeas}
    \min_{1\leq l \leq k}\|\mathcal Az_l-b\|^2\leq \frac{\|p_{*}-p_0\|^2}{c^2k}+\frac{\hat\epsilon^2}{2}.
\end{align}
The outer iteration bound \Cref{Outer Complexity Result} then follows from \Cref{BoundFeas}, the fact that $p_{*}$ is an arbitrary optimal multiplier, and a simple proof by contradiction.
The result of the proposition then follows from the outer iteration complexity bound and the above observations. 
\end{proof}

We are now ready to prove \Cref{Main Complexity O-IAL}. 
\begin{proof}[Proof of \Cref{Main Complexity O-IAL}]
\begin{enumerate}[label=(\alph*)]
\item
It follows from \Cref{OutItCompl} that the O-IAL method performs at most  
\begin{equation}\label{Outer Complexity}
        \left\lceil \frac{2\Pi(p_0)}{\hat \epsilon^2c^2} \right\rceil
    \end{equation} 
outer iterations to find a triple $(z,p,r)$ satisfying
\cref{ALsolutiontype}.
	Moreover, it follows from \Cref{ARtranslation}(b) that the call to the PF-AR method performs at most \begin{align*}
 \mathcal O\left(\sqrt{\left(\bar L+c\|\cA\|^2\right) \max\left\{\frac{D^2}{c\hat \epsilon^2}, \frac{D}{\hat \rho} \right\}}\right) 
\end{align*}
gradient/resolvent evaluations during each outer iteration of O-IAL. The conclusion of part (a) then follows from multiplying these two complexity bounds.

\item
It follows immediately from part (a) and the assumption that $c=\mathcal O(1/\hat \epsilon)$ that the O-IAL method performs at most
        \begin{equation}
          \mathcal O\left(2\Pi(p_0)\sqrt{\left(\bar L+\frac{\|\cA\|^2}{\hat \epsilon}\right) \max\left\{\frac{D^2}{\hat \epsilon}, \frac{D}{\hat \rho} \right\}}\right)
        \end{equation}
total gradient evaluations to find an $(\hat \epsilon, \hat \rho)$-approximate solution. The first conclusion of part (b) follows immediately
from this observation. The second
conclusion of part (b) follows immediately from the first conclusion and considering $\hat \epsilon=\hat \rho$.
\end{enumerate}
\end{proof}

\subsection{Best-Iterate Convergence of a Restarted Parameter-Free AL Method}\label{OPF-IAL}
The O-IAL method in \Cref{alg:O-IAL} is not parameter-free, since it requires knowledge of the domain diameter $D$. In this subsection, we develop OPF-IAL, a restarted parameter-free variant of O-IAL that uses a guess-and-check procedure for $D$. Each cycle runs O-IAL for a prescribed number of iterations using the current diameter guess in \Cref{diameter}; if no triple $(z_k,p_k,r_k)$ satisfying \Cref{ALsolutiontype} is found, OPF-IAL restarts with a larger diameter guess and a larger iteration budget.

As in O-IAL, the goal is to compute a $(\hat \epsilon,\hat \rho)$-approximate primal-dual solution satisfying \Cref{ALsolutiontype}. When $\hat \epsilon=\hat \rho$, OPF-IAL retains the optimal $\mathcal O(1/\hat \epsilon)$ complexity bound of O-IAL under \Cref{ass:A1}. We now formally describe the OPF-IAL method.



\begin{algorithm}[H]
\caption{OPF-IAL Method}
\label{alg:OPF-IAL}
\begin{algorithmic}[1]
\REQUIRE
function pair $(\psi_s,\psi_n)$, initial points $z_0 \in {\Dn}$ and $p_0\in \R^{m}$, primal-dual tolerance pair $(\hat \epsilon,\hat \rho) \in \mathbb R^{2}_{++}$, penalty parameter $c>0$, initial guess $I_{0}>0$, and smoothness and diameter estimates
$M_0:=\|\nabla \psi_s(z_0)-\nabla \psi_s(w_0)\|/(4\|z_0-w_0\|)$
and $D_0:=\|z_0-w_0\|$, where $w_0\in \Dn$ is chosen as in
\Cref{remarkPFAR}. \label{InputOPF}
\STATE      $\ell \gets 1$ \label{OPF1}
\STATE $\text{stopcrit } \gets \infty$ \label{OPF2}
\ENSURE set: $k \gets 0$, $z_0 \gets z_{\ell-1}$, and $p_0 \gets p_0$.\label{OPFInit}
\WHILE{$\text{stopcrit } >  \hat \epsilon$ (main outer loop)}\label{OPF3}
\STATE set $k\gets k+1$\label{OPF4}
\STATE call PF-AR method  (\Cref{alg:PF-AR})  with 
function pair  $(\cL^s_{c}(\cdot,p_{k-1}), \psi_n(\cdot))$, initial point $z_{k-1}$,\\
\qquad estimates $M_0$ and $D_0$, and
tolerance 
$\epsilon=\min\left\{\frac{\hat \epsilon}{4
D_{\ell-1}}, \hat \rho\right\}$, and let $(z_k,r_k)$ be its output; \label{ARorFISTACall}
\STATE $p_{k} \gets p_{k-1}+c(\mathcal Az_k-b)$;\label{OPF6}
\IF{$k>\left\lceil 2 I_{\ell-1} \right \rceil$}\label{OPF7}
\STATE 
\textbf{restart}, \label{OPF8}
\STATE i.e., set $z_{\ell}=z_k$, $D_{\ell}=2 D_{\ell-1}$, $I{_\ell}=2 I_{\ell-1}$, and  $\ell \leftarrow
\ell+1$, \label{OPF9}
\STATE and execute \textbf{Initialization} above;\label{OPF10}
\ELSE \label{OPF11}
\STATE set $\text{stopcrit } \gets \|\mathcal Az_k-b\|$.\label{OPF12}
\ENDIF \label{OPF13}
\ENDWHILE \label{OPF14}
\OUTPUT triple $(z,p,r):=(z_k,p_{k},r_k)$ satisfying
\cref{ALsolutiontype} \label{OPFOutput}
\end{algorithmic}
\end{algorithm}

We now make several remarks about the structure of OPF-IAL. The algorithm involves three nested types of iterations. Iterations indexed by $\ell$ are referred to as \emph{cycles}, iterations indexed by $k$ are referred to as \emph{outer iterations}, and the iterations performed internally by the PF-AR method in step~\ref{ARorFISTACall} are referred to as \emph{inner iterations}.

During each outer iteration of the $\ell$-th cycle, step~\ref{ARorFISTACall} calls the PF-AR method to approximately minimize the augmented Lagrangian function
$
\mathcal L_c(z,p):=\cL^s_c(z,p)+\psi_n(z)
$
with tolerance depending on the current estimate $ D_{\ell-1}$ of the diameter $D$. Moreover, step~\ref{OPF7} checks whether the current outer iteration count $k$ exceeds the threshold $\lceil 2 I_{\ell-1}\rceil$. If this condition is satisfied, then OPF-IAL restarts and a new cycle begins. At the restart, the estimates $D_\ell$ and $ I_{\ell}$ are updated according to
$
D_\ell=2 D_{\ell-1}$ and $
 I_{\ell}=2I_{\ell-1},
$
and the last primal iterate generated during the previous cycle is used to initialize the next cycle.

Finally, like O-IAL, the while loop in steps \ref{OPF3}--\ref{OPF14} of OPF-IAL terminates once
$
\|\mathcal A z_k-b\|\leq \hat \epsilon,
$
since OPF-IAL maintains approximate dual stationarity at every outer iteration.

\Cref{Main Complexity OPF-IAL} provides the main complexity result for OPF-IAL. The
proof of \Cref{Main Complexity OPF-IAL} is given in a collection of propositions and
lemmas in \Cref{Proof OPF-IAL}.

\begin{theorem}\label{Main Complexity OPF-IAL}
If OPF-IAL in \Cref{alg:OPF-IAL} uses penalty parameter input $c=1/\hat \epsilon$, it performs at most
\begin{equation*}
\begin{aligned}
\mathcal O \Bigg(
&\left(1+\left \lceil \max \left\{\log_{2}\left(\frac{D}{D_0} \right),
\log_{2}\left(\frac{\Pi(p_0)}{I_{0}} \right) \right\}\right \rceil\right)
\left[1+\left\lceil  \max\left\{\Pi(p_0),\frac{DI_{0}}{D_0} \right\}\right\rceil\right]  \\
&\times
\sqrt{\left(\bar L+\frac{\|\cA\|^2}{\hat \epsilon}\right)
\max\left\{\frac{D^2}{\hat \epsilon},
\frac{\Pi(p_0)D_0D}{I_{0}\hat \epsilon},
\frac{D}{\hat \rho} \right\}}
\Bigg)
\end{aligned}
\end{equation*}
resolvent/gradient evaluations to find an $(\hat \epsilon,\hat \rho)$-approximate solution of \Cref{LinearConstraints}
according to the criterion in \cref{ALsolutiontype}.
Hence, if $\hat\rho=\hat \epsilon$, it follows that OPF-IAL performs at most $\mathcal
O\left(1/\hat \epsilon\right)$ total gradient evaluations to find an $(\hat \epsilon,\hat \epsilon)$-approximate solution
of \Cref{LinearConstraints}.
\end{theorem}

\subsubsection{Proof of \Cref{Main Complexity OPF-IAL}}\label{Proof OPF-IAL}

This subsection is devoted to the proof of \Cref{Main Complexity
OPF-IAL}. We begin with a key proposition that provides a bound on the number of outer iterations that a cycle of OPF-IAL performs if it is made with a correct guess for the diameter.

\begin{prop}\label{Upper Bound Cycles}
    Suppose that the $\ell$-th cycle of OPF-IAL is performed with a guess $D_{\ell-1} \geq D$, where $D$ is as in \Cref{diameter}. It then holds that the $\ell$-th cycle of OPF-IAL performs at most
    \begin{equation}\label{Outer Complexity OPF-IAL}
        \left\lceil 2\Pi(p_0)\right\rceil
    \end{equation}
outer iterations, i.e, those indexed by $k$, to find a triple $(z,p,r):=(z_k,p_k,r_k)$  satisfying \Cref{ALsolutiontype}.
\end{prop}

\begin{proof}
Suppose that the $\ell$-th cycle of OPF-IAL is performed with a guess $D_{\ell-1} \geq D$ and let $k\geq 1$ be an iteration
index generated during the $\ell$-th cycle.
It follows from relations \Cref{diameter,inner product bound,inclusion Lagrangian} that 
\begin{align}\label{Bound Inner Product}
\inner{\mathcal{A}z_k-b}{p_{*}-p_k}
      \overset{\cref{inner product bound}} \geq
      \inner{r_k}{z_{*}-z_k}
      \overset{\cref{diameter}}{\geq}  -D \|r_k\|\overset{\cref{inclusion Lagrangian}}{\geq} -\frac{D\hat \epsilon}{4 D_{\ell-1}}\geq -\frac{D\hat \epsilon}{4D}=-\frac{\hat \epsilon}{4}.
\end{align}
It follows from a similar argument as in the proof of \Cref{bound on feasibility}, the update
rule for $p_k$ in step~\ref{OPF6} of OPF-IAL,
and the fact that $c=1/\hat \epsilon$ that 
\begin{align}\label{Recursion}
\|p_*-p_{k-1}\|^2-\|p_*-p_k\|^2&=\frac{1}{\hat \epsilon^2} \|\mathcal{A}z_k-b\|^2 + 
\frac{2}{\hat \epsilon} \inner{\mathcal{A}z_k-b}{p_*-p_k} \nonumber\\
& \overset{\cref{Bound Inner Product}}{\geq} \frac{1}{\hat \epsilon^2}\|\mathcal{A}z_k-b\|^2 - 
\frac{1}{2}.
\end{align}   
It follows from summing the above inequalities from $k=1$ to $m$, passing to the minimum, and multiplying both 
sides by $\hat \epsilon^2$ that
\[\min_{1 \leq k\leq m}\|\mathcal A z_k-b\|^2  \leq \frac{\hat \epsilon^2 \|p_{*}-p_0\|^2}{m}+\frac{\hat \epsilon^2}{2}.\]
It then follows from this bound and a simple proof by contradiction that the $\ell$-th cycle performs at most 
$\lceil 2 \Pi(p_0)\rceil$ outer iterations.
Moreover, for every outer iteration generated
during the cycle, a similar argument as in \Cref{ARtranslation}
implies that
\[
r_k\in \nabla \psi_s(z_k)+\partial\psi_n(z_k)+\cA^{*}p_k,
\qquad
\|r_k\|\leq \hat\rho
\]
for all outer indices $k$ generated.
Therefore, the returned triple $(z,p,r)$ satisfies \Cref{ALsolutiontype}.
\end{proof}

\begin{prop}\label{upper bound cycles Complexity}
   OPF-IAL performs at most
   \[1+\left \lceil \max \left\{\log_{2}\left(\frac{D}{D_0} \right), \log_{2}\left(\frac{\Pi(p_0)}{I_0} \right) \right\}\right \rceil\]
cycles to find a triple $(z,p,r)$ satisfying \Cref{ALsolutiontype}.
\end{prop}

\begin{proof}
It follows from \Cref{Upper Bound Cycles} that if the $\ell$-th cycle of 
OPF-IAL is performed with $D_{\ell-1}\geq D$ and $I_{{\ell-1}}\geq \Pi(p_0)$, then 
OPF-IAL finds a triple $(z_k,p_k,r_k)$ satisfying \Cref{ALsolutiontype}. Now, using the definitions of $D_{\ell}$ and $I_{\ell}$, it follows that $D_{\ell-1} \geq D$ and
$I_{\ell-1}\geq \Pi(p_0)$ if 
\[D_{\ell-1}=2^{\ell-1}D_0 \geq D, \text{ and } I_{\ell-1}=2^{\ell-1}I_{0}\geq \Pi(p_0).\]
Hence, by re-arranging the above inequalities and using the fact that $D_0\leq D$, it follows that if the number of cycles $\ell$ satisfies
\[\ell\geq 1+\max \left\{\log_{2}\left(\frac{D}{D_0} \right), \log_{2}\left(\frac{ \Pi(p_0)}{I_{0}} \right) \right\},\]
then OPF-IAL must find a triple $(z,p,r)$ satisfying \Cref{ALsolutiontype}. The result of the lemma then immediately 
follows from this conclusion.
\end{proof}

The following lemma provides upper and lower bounds on the quantities $D_{\ell}$
and $I_{\ell}$.

\begin{lemma}
    The following inequalities hold for any cycle index $\ell\geq 1$ generated by OPF-IAL:
    \begin{align}
    D_{0} &\leq D_{\ell} \leq \max\left\{4D,\frac{4\Pi(p_0)D_0}{I_{0}} \right\} \label{DL bounds}\\
    I_{0} &\leq I_{\ell} \leq  \max\left\{4\Pi(p_0),\frac{4DI_{0}}{D_0} \right\} \label{dpl bounds}.
    \end{align}
\end{lemma}

\begin{proof}
The facts that $D_0\leq D_{\ell}$ and $I_{0} \leq I_{\ell}$ follow immediately
from the update rules of $D_{\ell}$ and $I_{\ell}$ in step~\ref{OPF9} of OPF-IAL.
To show the second inequality in \Cref{DL bounds}, first observe that it follows from \Cref{upper bound cycles Complexity} that the number of cycles $\ell$ that OPF-IAL performs satisfies
\begin{align*}
\ell&\leq 1+  \left \lceil \max \left\{\log_{2}\left(\frac{D}{D_0} \right), \log_{2}\left(\frac{\Pi(p_0)}{I_{0}} \right) \right\}\right \rceil\leq 2+ \max \left\{\log_{2}\left(\frac{D}{D_0} \right), \log_{2}\left(\frac{\Pi(p_0)}{I_{0}} \right) \right\}.
\end{align*}
Hence, it follows from the update rule for $ D_{\ell}$ in step~\ref{OPF9} of \Cref{alg:OPF-IAL} that the following inequalities hold for any cycle index
$\ell\geq 1$ generated by OPF-IAL:
\begin{align}\label{Upper Dl}
D_{\ell}=2^{\ell}D_0&\leq 2^{2+ \max \left\{\log_{2}\left(\frac{D}{D_0} \right), \log_{2}\left(\frac{\Pi(p_0)}{I_{0}} \right) \right\}}D_0 =2^2\cdot \max\left\{\frac{D}{ D_0},\frac{\Pi(p_0)}{I_{0}} \right\}D_0.
\end{align}
The second inequality in \Cref{DL bounds} then immediately follows from \Cref{Upper Dl}.
Likewise, the second inequality in \Cref{dpl bounds} follows from a similar argument as the above argument and 
the fact that $I_{\ell}=2^{\ell}I_{0}$.
\end{proof}

The following proposition upper bounds the total number of outer iterations, i.e. those indexed by $k$, that OPF-IAL performs during each of its cycles.
The proposition also characterizes the number of inner iterations that the call to the PF-AR method performs during each of OPF-IAL's outer iterations.

\begin{proposition}\label{Prop Outer and Inner Complexities}
Let $\ell\geq 1$ be a cycle index generated by OPF-IAL. The following statements hold.
\begin{itemize}
\item[(a)] The $\ell$-th cycle of OPF-IAL performs at most 
\begin{equation}\label{Outer per Cycle}
        1+\left\lceil  \max\left\{8\Pi(p_0),\frac{8DI_{0}}{D_0} \right\}\right\rceil
\end{equation}
outer iterations, i.e., iterations indexed by k;
\item[(b)] During step~\ref{ARorFISTACall} of each of OPF-IAL's outer iterations, the PF-AR method performs at most
\begin{equation}
 \mathcal O\left(\sqrt{\left(\bar L+\frac{\|\cA\|^2}{\hat \epsilon}\right) \max\left\{\frac{D^2}{\hat \epsilon}, \frac{\Pi(p_0)D_0D}{I_{0}\hat \epsilon}, \frac{D}{\hat \rho} \right\}}\right) 
\end{equation}
gradient evaluations.
\end{itemize}
\end{proposition}
\begin{proof}
(a) Let $\ell\geq 1$ be a cycle index generated by OPF-IAL. It follows from the inequality check that OPF-IAL performs in its step~\ref{OPF7} that 
its $\ell$-th cycle performs at most $1+\lceil 2 I_{\ell-1}\rceil $ outer iterations, i.e., iterations indexed by $k$. The
result then follows from this bound and the upper bound for $I_{\ell-1}$ in \Cref{dpl bounds}.

(b) It follows from the last conclusion of \Cref{PFARComplexity}, a similar argument as in \Cref{ARtranslation}(b), and the fact that the OPF-IAL method calls the 
PF-AR method in its step~\ref{ARorFISTACall} with tolerance pair $\epsilon= \min\left\{\frac{\hat \epsilon}{4
D_{\ell-1}}, \hat \rho\right\}$ that the PF-AR method performs at most
\begin{align}\label{intermediary bound}
\mathcal O\left(\sqrt{\left(\bar L+c\|\cA\|^2\right) \max\left\{\frac{D_{\ell-1}D}{c\hat \epsilon^2}, \frac{D}{\hat \rho} \right\}}\right)
\end{align}
gradient evaluations during each of its calls.
The conclusion of the proposition then follows from \Cref{intermediary bound}, the fact that $c=1/\hat \epsilon$, and
the upper bound for $D_{\ell-1}$ in \Cref{DL bounds}.
\end{proof}

We are now ready to prove \Cref{Main Complexity OPF-IAL}.

\begin{proof}[Proof of \Cref{Main Complexity OPF-IAL}]
It follows from \Cref{upper bound cycles Complexity} that OPF-IAL performs at most
\[1+\left \lceil \max \left\{\log_{2}\left(\frac{D}{D_0} \right), \log_{2}\left(\frac{\Pi(p_0)}{I_{0}} \right) \right\}\right \rceil\]
cycles to find a triple $(z,p,r)$ satisfying \Cref{ALsolutiontype}. Moreover, it follows from 
\Cref{Prop Outer and Inner Complexities}(a) that each cycle of OPF-IAL performs at most
\[1+\left\lceil  \max\left\{8\Pi(p_0),\frac{8DI_{0}}{D_0} \right\}\right\rceil\]
outer iterations.
Finally, \Cref{Prop Outer and Inner Complexities}(b) implies that the PF-AR method performs
at most 
\[\mathcal O\left(\sqrt{\left(\bar L+\frac{\|\cA\|^2}{\hat \epsilon}\right) \max\left\{\frac{D^2}{\hat \epsilon}, \frac{\Pi(p_0)D_0D}{I_{0}\hat \epsilon}, \frac{D}{\hat \rho} \right\}}\right) \]
gradient evaluations during each of OPF-IAL's outer iterations. The total complexity result of
OPF-IAL presented in \Cref{Main Complexity OPF-IAL} then follows from multiplying the
above three complexity bounds and absorbing the constants in the big O.
\end{proof}

\subsection{Last-Iterate Convergence of an Adaptive Parameter-Free AL Method}\label{APFIAL}
This subsection presents an adaptive parameter-free inexact augmented Lagrangian method, 
namely, APF-IAL, with last-iterate convergence guarantees for approximately solving the primal-dual pair \cref{LinearConstraints}
and \cref{dual}. Throughout this subsection, we consider problem \cref{LinearConstraints} 
under the following two assumptions.
\begin{assump}
\label{ass:A6} 
    The function $\psi_n$ is $M_n$-Lipschitz continuous on $\Dn$;
\end{assump}
\begin{assump}
\label{ass:A7} 
    There exists $\bar z \in \inte({\Dn})$ such that $\mathcal A\bar z=b$.
\end{assump}

Unlike O-IAL and OPF-IAL, APF-IAL achieves stronger last-iterate convergence rate
guarantees while the former two methods achieve best-iterate convergence rate guarantees. Like OPF-IAL,
APF-IAL is parameter-free and requires no knowledge of any parameters underlying \Cref{LinearConstraints}
and \Cref{dual}. However, unlike OPF-IAL and O-IAL, APF-IAL has the advantage that it is provably able to adaptively update the penalty parameter $c$
and the tolerance that it uses to approximately minimize the AL function at every iteration.
Finally given a tolerance pair $(\hat \epsilon,\hat \rho) \in \mathbb R_{++}$ and assuming that $\hat \epsilon=\hat \rho$, 
APF-IAL, like O-IAL and OPF-IAL, achieves an optimal complexity bound $\mathcal O(1/\hat
 \epsilon)$ for finding an approximate primal-dual pair of \Cref{LinearConstraints} and \Cref{dual} according to the
 criterion in \cref{ALsolutiontype}.

The APF-IAL method is formally presented in \Cref{alg:APF-IAL}
below.

\begin{algorithm}[H]
\caption{APF-IAL Method}
\label{alg:APF-IAL}
\begin{algorithmic}[1]
\REQUIRE
function pair $(\psi_s,\psi_n)$, initial points $z_0 \in {\Dn}$ and $p_0\in \R^{m}$, primal-dual tolerance pair $(\hat \epsilon,\hat \rho) \in \mathbb R^{2}_{++}$, an initial penalty parameter $c_1>0$, scalars $\alpha>1$ and $\omega>1$, a scalar $\tilde \epsilon_1\geq \hat \rho$, and smoothness and diameter estimates
$M_0:=\|\nabla \psi_s(z_0)-\nabla \psi_s(w_0)\|/(4\|z_0-w_0\|)$
and $D_0:=\|z_0-w_0\|$, where $w_0\in \Dn$ is chosen as in
\Cref{remarkPFAR}.\label{APFInput}
\STATE     set $k \gets 0$;\label{APF1}
\STATE set $\Resp \gets \mathcal{A}$$z_{k}-b$ and $\Resd\gets \infty$;\label{APF2}
\WHILE{$\|\Resp\| >  \hat \epsilon$ or $\|\Resd\|>\hat \rho$ (main outer loop)}\label{APF3}
\STATE set $k \gets k+1$;\label{APF4}
\IF{$\|\Resp\|>\hat \epsilon$}\label{APF5}
\STATE $\epsilon_k=\max\left\{\tilde \epsilon_k,\hat\rho\right\}$\label{APF6}
\ELSE \label{APF7}
\STATE $\epsilon_k=\hat \rho$\label{APF8}
\ENDIF\label{APF9}
\STATE call PF-AR method  (\Cref{alg:PF-AR})  with 
function pair  $(\cL^s_{c}(\cdot,p_{k-1}), \psi_n(\cdot))$, initial point $z_{k-1}$,\\
\qquad estimates $M_0$ and $D_0$, and
tolerance 
$\epsilon=\epsilon_k$, and let $(z_k,r_k)$ be its output\label{APF10}; 
\STATE set $\Resp\gets \mathcal Az_k-b$ and $\Resd\gets r_k$;\label{APF11}
\STATE update \label{APF12}
$p_k=p_{k-1}+c_k\Resp$;
\STATE set $c_{k+1}=\alpha c_k$ and $\tilde \epsilon_{k+1}=\tilde \epsilon_{k}/\omega$;\label{APF13}
\ENDWHILE\label{APF14}
\OUTPUT triple $(z,p,r):=(z_k,p_{k},r_k)$ satisfying
\cref{ALsolutiontype}.\label{APFOutput}
\end{algorithmic}
\end{algorithm}
We make a few remarks on APF-IAL. First, the while loop in steps~\ref{APF3}--\ref{APF14} terminates only when both primal feasibility and dual stationarity are small, namely, $\|\Resp\|\leq \hat \epsilon$ and $\|\Resd\|\leq \hat \rho$. Unlike O-IAL and OPF-IAL, APF-IAL does not maintain approximate dual stationarity at every outer iteration. Second, $\epsilon_k$ is the PF-AR tolerance used in step~\ref{APF10}: it is set to $\hat \rho$ once primal feasibility is sufficiently small, and otherwise to the decreasing adaptive tolerance $\tilde \epsilon_k$. Third, step~\ref{APF12} performs the standard multiplier update, as in O-IAL and OPF-IAL. Finally, step~\ref{APF13} updates the penalty and tolerance according to $c_{k+1}=\alpha c_k$ and $\tilde \epsilon_{k+1}=\tilde \epsilon_k/\omega$, where $\alpha,\omega>1$. Gradually increasing the penalty parameter is typically more practical than using a large fixed penalty, while still yielding an optimal convergence rate.


The following result describes the main complexity
result of APF-IAL. Its proof is given in \Cref{sect:proofalgD}.

\begin{theorem}\label{Complexity APF-IAL}
Let \begin{equation}\label{phi* and lambda bound and bar d}
 \bar{d}:=\dist(\bar z, {\bd(\Dn)}), \quad \nabla_{s}:=\sup_{z\in \Dn} \|\nabla \psi_s(z)\|
\end{equation}
where here $\bd(\Dn)$ denotes the boundary of $\dom \psi_n$.
APF-IAL performs at most
\begin{equation}\label{APFComplexityBound}
\mathcal O\left(\frac{\alpha^{3/2}\|\mathcal A\|D\sqrt{M_n+\nabla_s+\tilde \epsilon_1}}{\sqrt{\bar d\sigma^{+}_{\mathcal A}}(\sqrt{\alpha}-1)\sqrt{\hat \epsilon\hat \rho}}\right)
\end{equation}
total gradient evaluations to find an $(\hat \epsilon,\hat \rho)$-approximate
	solution according to the criterion in \cref{ALsolutiontype},
where
$(\hat \epsilon, \hat \rho)$, $\alpha$, and $c_1$ are inputs, $D$ and $M_n$ are as in 
\Cref{diameter,ass:A6},
respectively, and $\sigma_{\mathcal A}^{+}$ is the smallest positive
singular value of the linear transformation $\mathcal A$.
Hence, it follows that APF-IAL performs at most 
$\mathcal O(1/\hat \epsilon)$
total gradient evaluations to find an $(\hat \epsilon,\hat \epsilon)$-approximate solution of \Cref{LinearConstraints}. 
\end{theorem}

\subsubsection{Proof of \texorpdfstring{\Cref{Complexity APF-IAL}}{Lg}}
\label{sect:proofalgD}
This subsection is dedicated to proving \Cref{Complexity APF-IAL}.
The following key result from \cite{AL38Sujanani} is presented with minor changes in notation. It will be instrumental in showing that the sequence of Lagrange multipliers $\{p_k\}$ is bounded.

\begin{lemma}\label{lem:qbounds-2}
Assume that $\psi_n$ satisfies
\Cref{ass:A6}, $\mathcal A:\E \to \Rm$ is a linear transformation
satisfying \Cref{ass:A7}, and the triple $(z,p,s) \in \E \times \R^{m} \times \E$ satisfies
$s \in \partial \psi_n(z)+\mathcal A^{*}p$.
Then:
\begin{itemize}
\item[(a)] there holds
\begin{equation}\label{ineq:aux9001-2}
 \bar{d}\sigma_{\mathcal A}^{+}\|p\|
\leq  2 D\left(M_n + \|s\| \right)  - \inner{p}{\mathcal Az-b};
\end{equation}
\item[(b)] if, in addition, 
\begin{equation}\label{q form}
p=p^-+\chi(\mathcal Az-b)   
\end{equation}
for some $p^-\in \Rm$ and $\chi>0$, we have 
\begin{equation}\label{q bound-2}
 \|p\|\leq \max\left\{\|p^-\|,\frac{2D(M_n+\|s\|)}{\bar d \sigma^{+}_{\mathcal A}} \right\}.
\end{equation}
\end{itemize}
\end{lemma}

The following proposition characterizes the properties and the output of
the call made to the PF-AR method during the $k$-th iteration of APF-IAL.
We omit the proof of the proposition since it is nearly identical to the
proof of \Cref{ARtranslation} with the only differences being that the static penalty parameter $c$ is replaced with the adaptive penalty parameter $c_k$ and that the quantity $\min\left\{(c\hat \epsilon^2)/4D,\hat \rho\right\}$ is now replaced by $\epsilon_k$ since the PF-AR method is called in APF-IAL with the latter tolerance.

\begin{prop}\label{AR adaptive}
    The following statements about the $k$-th iteration of APF-IAL hold:
    \begin{itemize}
        \item[(a)] the call made to the PF-AR method in step~\ref{APF10} outputs a pair $(z_k,r_k)$  satisfying
        \begin{equation}\label{rk AR adaptive}
          r_{k}\in \nabla \psi_s(z_k)+\partial\psi_n(z_k)+\mathcal A^{*}p_k, \quad \|r_k\|\leq \epsilon_k;
        \end{equation}
        \item[(b)] the call made to the PF-AR method performs at most 
\begin{align*}
 \mathcal O\left(\sqrt{\frac{\left(\bar L+c_k\|\mathcal A\|^2\right)D}{\epsilon_k}}\right) 
\end{align*}
gradient evaluations to find a pair $(z_k,r_k)$ satisfying \cref{rk AR adaptive}.
\end{itemize}
\end{prop}

The following result provides upper and lower bounds on the sequence of tolerances $\{\epsilon_k\}$ generated in steps \ref{APF5} to \ref{APF9} of APF-IAL.

\begin{lemma}\label{Bound on EpsilonK}
For any iteration index $k \geq 1$ generated by APF-IAL, we have
\begin{equation}\label{Epsilon Sequence Decay}
       \hat \rho \leq \epsilon_k \leq \tilde \epsilon_1.
\end{equation}
\end{lemma}

\begin{proof}
It follows from the way $\epsilon_k$ is set in
steps \ref{APF5} to \ref{APF9} of APF-IAL
that $\epsilon_{k}=\max\{\tilde \epsilon_k,\hat \rho\}$ or $\epsilon_{k}=\hat \rho$, which immediately implies the first bound in \Cref{Epsilon Sequence Decay}.
The second bound in \Cref{Epsilon Sequence Decay} follows immediately from this observation, the fact that $\tilde \epsilon_{k}$ is a decreasing sequence, and the fact that $\tilde \epsilon_{1}\geq \hat \rho$.
\end{proof}

\begin{prop}\label{PK Bounded} 
The sequence  $\{p_k\}$ generated by APF-IAL satisfies 
\begin{equation} \label{P Bound}
\|p_k\|\leq \max\left\{\frac{2D(M_n+\nabla_s+\tilde\epsilon_1)}{\bar d \sigma^{+}_{\mathcal A}}, \|p_0\|\right\}\quad \forall k\geq 0.
\end{equation}
\end{prop}
\begin{proof}
It follows from the definition of $\nabla_s$ in \cref{phi* and lambda bound and bar d}, the inequality in \cref{rk AR adaptive}, the second relation in \cref{Epsilon Sequence Decay}, and triangle inequality that the following inequality holds
\begin{equation}\label{Bound on Residual}
\|r_k-\nabla \psi_s(z_k)\|\leq \|r_k\|+\|\nabla\psi_s(z_k)\|\overset{\cref{rk AR adaptive}}{\leq} \epsilon_k+\|\nabla \psi_s(z_k)\|\overset{\cref{Epsilon Sequence Decay}}{\leq} \tilde \epsilon_1+\|\nabla \psi_s(z_k)\|\overset{\cref{phi* and lambda bound and bar d}}{\leq} \tilde \epsilon_1+\nabla_s
\end{equation}
for all $k \geq 1$.
Now, it follows from the above inequality, the fact that the inclusion
in \cref{rk AR adaptive} implies that $r_k-\nabla \psi_s(z_k)\in
\partial \psi_n(z_k)+\mathcal A^{*}p_k$, and \Cref{lem:qbounds-2}(b) with $(z,p,s)=(z_k,p_k,r_k-\nabla \psi_s(z_k))$
that the following holds
\begin{align*}
\|p_k\|\overset{\cref{q bound-2}}{\leq} \max\left\{\|p_{k-1}\|,\frac{2D(M_n+\|r_k-\nabla \psi_s(z_k)\|)}{\bar d \sigma^{+}_{\mathcal A}} \right\}\overset{\cref{Bound on Residual}}{\leq}\max\left\{\|p_{k-1}\|, \frac{2D(M_n+ \nabla_s+\tilde \epsilon_1)}{\bar d \sigma^{+}_{\mathcal A}}\right\}
\end{align*}
for all $k \geq 1$.
The conclusion of the proposition then follows immediately from the
above relation, the fact that inequality \cref{P Bound} is satisfied with $k=0$, and an induction argument.
\end{proof}

The following proposition bounds the primal feasibility of the iterates generated by APF-IAL.

\begin{prop}\label{Feasibility Proposition}
    The following relation holds for any iteration index $k\geq 1$ generated by APF-IAL:
    \begin{equation}\label{Feasibility Adaptive}
        \|\mathcal A z_k-b\|\leq \max\left\{\frac{4D(M_n+\nabla_s+\tilde \epsilon_1)}{c_k\bar d \sigma^{+}_{\mathcal A}}, \frac{2\|p_0\|}{c_k}\right\}.
    \end{equation}
\end{prop}

\begin{proof}
    The update rule for the dual variable $p_k$ in step~\ref{APF12}, triangle inequality, and relation \Cref{P Bound} imply that the following relation holds for any iteration index $k \geq 1$:
\begin{align*}
    \|\mathcal Az_k-b\|&=\frac{1}{c_k}\|p_k-p_{k-1}\|\leq \frac{1}{c_k}\left(\|p_k\|+\|p_{k-1}\| \right) \overset{\cref{P Bound}}{\leq} \max\left\{\frac{4D(M_n+\nabla_s+\tilde \epsilon_1)}{c_k\bar d \sigma^{+}_{\mathcal A}}, \frac{2\|p_0\|}{c_k}\right\},
\end{align*}
from which the result of the proposition immediately follows.
\end{proof}
The following proposition provides an upper bound on the sequence of penalty parameters, $c_k$, generated by APF-IAL.

\begin{prop}\label{Penalty Bound Prop}
    For any iteration index $k\geq 1$ generated by APF-IAL, we have
\begin{equation}\label{Bound Penalty}
    c_k \leq \max\left\{\alpha c_1, \frac{4\alpha^2 D(M_n+\nabla_s+\tilde \epsilon_1)}{\hat \epsilon\bar d\sigma^{+}_{\mathcal A}}, \frac{2\alpha^2\|p_0\|}{\hat \epsilon}\right\}
\end{equation}
where $\hat \epsilon>0$ is the input tolerance to APF-IAL and $c_1$ is the initial penalty parameter.
\end{prop}

\begin{proof}
First, observe that relation \cref{Bound Penalty} holds for $k=1$ and $k=2$ since $\alpha>1$ and $c_2=\alpha c_1$. 
Suppose by contradiction that APF-IAL generates an iteration index $k\geq 3$ such that 
\begin{equation}\label{Contradiction}
c_k>\max\left\{\alpha c_1, \frac{4\alpha^2 D(M_n+\nabla_s+\tilde\epsilon_1)}{\hat \epsilon\bar d\sigma^{+}_{\mathcal A}}, \frac{2\alpha^2\|p_0\|}{\hat \epsilon}\right\}.
\end{equation}
It then follows from this relation and the way that $c_k$ is updated in step~\ref{APF13} of APF-IAL that
\begin{align}\label{Penalty contradiction}
c_{k-2}=\frac{c_k}{\alpha^2}>\max\left\{\frac{4D(M_n+\nabla_s+\tilde \epsilon_1)}{\hat \epsilon\bar d\sigma^{+}_{\mathcal A}}, \frac{2\|p_0\|}{\hat \epsilon}\right\}.
\end{align}
This relation together with relation \cref{Feasibility Adaptive} then implies that 
\begin{align}\label{Feas 2}
\|\mathcal Az_{k-2}-b\|\overset{\cref{Feasibility Adaptive}}{\leq} \frac{4D(M_n+\nabla_s+\tilde \epsilon_1)}{c_{k-2}\bar d \sigma^{+}_{\mathcal A}} < \hat \epsilon.
\end{align}
If $\|r_{k-2}\|\leq \hat \rho$, it follows from the above relation
that APF-IAL must terminate in its $(k-2)$-th iteration, which
contradicts the fact that APF-IAL generated the iteration index $k-1$.
For the other situation, suppose now that $\|r_{k-2}\|>\hat \rho$.
It then follows from relation \cref{Feas 2} and the way $\epsilon_{k-1}$
is set during steps \ref{APF5} to \ref{APF9} of APF-IAL that $\epsilon_{k-1}=\hat \rho$.
From \cref{rk AR adaptive}, we get that $\|r_{k-1}\|\leq \hat \rho$. Moreover, 
it follows from relation \cref{Penalty contradiction}, the update rule for $c_{k}$ in step~\ref{APF13} of APF-IAL, and the fact that $\alpha>1$ that $c_{k-1}=\alpha c_{k-2}$ that 
\[c_{k-1}>\max\left\{\frac{4D(M_n+\nabla_s+\tilde \epsilon_1)}{\hat \epsilon\bar d\sigma^{+}_{\mathcal A}}, \frac{2\|p_0\|}{\hat \epsilon}\right\}.\] 
Hence, this relation together with relation \cref{Feasibility Adaptive} implies that 
\begin{align}\label{Feas Again}
    \|\mathcal A z_{k-1}-b\| \overset{\cref{Feasibility Adaptive}}{\leq} \max\left\{\frac{4D(M_n+\nabla_s+\tilde \epsilon_1)}{c_{k-1}\bar d \sigma^{+}_{\mathcal A}}, \frac{2\|p_0\|}{c_{k-1}}\right\} <  \hat \epsilon.
\end{align}
It follows from relation \cref{Feas Again} and the fact 
that $\|r_{k-1}\|\leq \hat \rho$ that APF-IAL must terminate in its $(k-1)$-th iteration, 
which contradicts the fact that APF-IAL generated 
the iteration index $k$ satisfying relation \cref{Contradiction}. Hence, relation \cref{Bound Penalty} must hold for all iteration indices $k$ generated by APF-IAL.
\end{proof}

\begin{prop}
    The number of outer iterations performed by APF-IAL is at most 
 \begin{equation}\label{Outer Iterations}
    K:=2+\left\lceil\log_{\alpha}^{+} \left(\max\left\{\frac{4\alpha^2 D(M_n+\nabla_s+\tilde\epsilon_1)}{\hat \epsilon (c_1\bar d\sigma^{+}_{\mathcal A})}, \frac{2\alpha^2\|p_0\|}{\hat \epsilon c_1}\right\}\right)\right\rceil
\end{equation}
where $\hat \epsilon>0$ is the input tolerance to APF-IAL and $c_1$ is the initial penalty parameter.
\end{prop}

\begin{proof}
    Suppose by contradiction that APF-IAL generates an iteration index $\hat K>K$. It follows from the definition of $K$ above that
    \begin{align*}
    \hat K-1&> K-1=1+\left\lceil \log^{+}_{\alpha} \left(\max\left\{\frac{4\alpha^2 D(M_n+\nabla_s+\tilde\epsilon_1)}{\hat \epsilon (c_1\bar d\sigma^{+}_{\mathcal A})}, \frac{2\alpha^2\|p_0\|}{\hat \epsilon c_1}\right\}\right)\right\rceil\\
    &>\log^{+}_{\alpha} \left(\max\left\{\frac{4\alpha^2 D(M_n+\nabla_s+\tilde\epsilon_1)}{\hat \epsilon (c_1\bar d\sigma^{+}_{\mathcal A})}, \frac{2\alpha^2\|p_0\|}{\hat \epsilon c_1}\right\}\right).
    \end{align*}
It is then easy to see that the above relation implies that
\[\alpha^{\hat K-1}>\max\left\{\frac{4\alpha^2D(M_n+\nabla_s+\tilde\epsilon_1)}{\hat \epsilon (c_1\bar d\sigma^{+}_{\mathcal A})}, \frac{2\alpha^2\|p_0\|}{\hat \epsilon c_1}\right\}.\]
Now, it is easy to see from the way that $c_k$ is updated in step~\ref{APF13} of APF-IAL that $c_{\hat K}=c_1 \alpha^{\hat K-1}$. This fact together with the above bound then imply that
\begin{equation}\label{Bound}
c_{\hat K}>\max\left\{\alpha c_1, \frac{4\alpha^2 D(M_n+\nabla_s+\tilde\epsilon_1)}{\hat \epsilon (\bar d\sigma^{+}_{\mathcal A})}, \frac{2\alpha^2\|p_0\|}{\hat \epsilon}\right\},
\end{equation}
which contradicts the bound on $c_k$ in \Cref{Bound Penalty}.
\end{proof}

We are now ready to prove \Cref{Complexity APF-IAL}.

\begin{proof}[Proof of \Cref{Complexity APF-IAL}]
Observe first that it follows from \Cref{AR adaptive}(b),
\Cref{Bound on EpsilonK},
and the fact that $c_{k}=c_1\alpha^{k-1}$ that the PF-AR method performs at most
\begin{equation}\label{Inner Iterations}
T_k:=\sqrt{\frac{\bar LD+c_1\alpha^{k-1}D\|\mathcal A\|^2}{\hat \rho}}
\end{equation}
gradient evaluations during the $k$-th outer iteration of APF-IAL.
It then follows from this observation, the bound \cref{Outer
	Iterations} on the number of outer iterations performed by
	APF-IAL, \Cref{AR adaptive}(a), the fact that $\sqrt{a+b}\leq
	\sqrt{a}+\sqrt{b}$, and the bound on $c_k$ in \Cref{Penalty Bound Prop} that APF-IAL performs at most the following number of gradient evaluations:
\begin{align*}
    &\sum_{k=1}^{K}T_k\overset{\cref{Inner Iterations}}{=}\sum_{k=1}^{K}\sqrt{\frac{\bar LD+c_1\alpha^{k-1}D\|\mathcal A\|^2}{\hat \rho}}
    \leq K\sqrt{\frac{\bar LD}{\hat \rho}}+\frac{\sqrt{c_1D}\|\mathcal A\|}{\sqrt{\hat \rho}}\sum_{k=1}^{K}\sqrt{\alpha^{k-1}}\\
&= K \sqrt{\frac{\bar LD}{\hat \rho}}+\frac{\sqrt{c_1D}\|\mathcal A\|}{\sqrt{\hat\rho}}\left[\frac{\sqrt{\alpha^{K}}-1}{\sqrt{\alpha}-1}\right] \leq K \sqrt{\frac{\bar LD}{\hat \rho}}+\frac{\|\mathcal A\|\sqrt{D}\sqrt{c_1\alpha^{K}}}{\sqrt{\hat\rho}(\sqrt{\alpha}-1)}
=K \sqrt{\frac{\bar LD}{\hat \rho}}+\frac{\sqrt{\alpha}}{\sqrt{\alpha}-1}\frac{\|\mathcal A\|\sqrt{D}\sqrt{c_{K}}}{\sqrt{\hat\rho}}\\
&\overset{\cref{Bound Penalty}}{\leq} K \sqrt{\frac{\bar LD}{\hat \rho}} +\max\left\{\left[\frac{\alpha\|\mathcal A\|\sqrt{D}\sqrt{c_1}}{(\sqrt{\alpha}-1)}\right]\frac{1}{\sqrt{\hat \rho}}, \frac{2\alpha^{3/2}\|\mathcal A\|D\sqrt{M_n+\nabla_s+\tilde \epsilon_1}}{\sqrt{\bar d\sigma^{+}_{\mathcal A}}(\sqrt{\alpha}-1)}\left[\frac{1}{\sqrt{\hat \epsilon}\sqrt{\hat \rho}}\right]\right\}\\
&\overset{\cref{Outer Iterations}}{=}\left[2+\left\lceil\log_{\alpha} \left(\max\left\{\frac{4\alpha^2 D(M_n+\nabla_s+\tilde\epsilon_1)}{\hat \epsilon (c_1\bar d\sigma^{+}_{\mathcal A})}, \frac{2\alpha^2\|p_0\|}{\hat \epsilon c_1}\right\}\right)\right\rceil\right]\sqrt{\frac{\bar LD}{\hat \rho}}\\
&+\max\left\{\left[\frac{\alpha\|\mathcal A\|\sqrt{D}\sqrt{c_1}}{(\sqrt{\alpha}-1)}\right]\frac{1}{\sqrt{\hat \rho}}, \frac{2\alpha^{3/2}\|\mathcal A\|D\sqrt{M_n+\nabla_s+\tilde \epsilon_1}}{(\sqrt{\alpha}-1)\sqrt{\bar d\sigma^{+}_{\mathcal A}\hat \epsilon\hat \rho}}\right\}
\end{align*}
to find an $(\hat \epsilon,\hat \rho)$-approximate solution of \Cref{LinearConstraints} according to \Cref{ALsolutiontype}.
Hence, it follows from the above relation that APF-IAL performs at most
\[\mathcal O\left(\frac{\alpha^{3/2}\|\mathcal A\|D\sqrt{M_n+\nabla_s+\tilde \epsilon_1}}{\sqrt{\bar d\sigma^{+}_{\mathcal A}}(\sqrt{\alpha}-1)\sqrt{\hat \epsilon\hat \rho}}\right)\]
total gradient evaluations to find an $(\hat \epsilon,\hat \rho)$-approximate solution,
which implies the complexity bound in \Cref{APFComplexityBound}.

The last conclusion of the theorem immediately follows
from \Cref{APFComplexityBound} and taking $\hat \epsilon=\hat \rho$.
\end{proof}

\section{Linearly-Constrained Strongly Convex Optimization}
\phantomsection 
In this section, we discuss the complexities of the
inexact AL methods, O-IAL (\Cref{alg:O-IAL}), OPF-IAL (\Cref{alg:OPF-IAL}) and APF-IAL (\Cref{alg:APF-IAL}), for solving
\Cref{LinearConstraints}, where $\psi$ is now assumed to be $\bar \mu$-strongly convex with $\bar \mu>0$.

The rest of this section is organized as follows. \Cref{Restarted FISTA}
presents the parameter-free R-FISTA method and its complexity guarantees for
unconstrained strongly convex composite minimization. The R-FISTA method
will be used as a subroutine inside O-IAL, OPF-IAL, and APF-IAL to approximately
minimize their strongly convex AL subproblems. \Cref{StrongConvComplexity} presents the near-optimal complexity bounds of O-IAL, OPF-IAL, and APF-IAL for the linearly-constrained strongly convex problem \Cref{LinearConstraints}

\subsection{R-FISTA: Parameter-Free Strongly Convex Optimization Subproblem Solver}\label{Restarted FISTA}
\phantomsection 
This section describes a restarted FISTA method, namely R-FISTA, which will 
be used by OPF-IAL and APF-IAL to solve their strongly convex composite AL subproblems. More generally, R-FISTA solves the unconstrained composite optimization problem version of
\Cref{LinearConstraints}, i.e.,
\begin{equation}\label{MainProblem}
\min \left\{\psi(z):=\psi_s(z)+\psi_n(z)\right\},
\end{equation}
where $\psi$ is assumed to be $\bar \mu$-strongly convex with $\bar \mu>0$, $\psi_s$ is convex and $\bar L$-smooth, and $\psi_n$ is a closed proper convex function on its compact domain $\Dn$, which has diameter $D$.

Given $\epsilon>0$, R-FISTA aims to find $(y,v)\in \Dn\times \E$ such that
\begin{gather}\label{acg problem}
\|v\| \le \epsilon, \quad v \in \nabla \psi_s(y) + \partial \psi_n(y).
\end{gather}
This is the unconstrained form of the approximate optimality condition in
\Cref{ApproxOptSolution}, obtained from \Cref{ALsolutiontype} by setting
$\hat\epsilon=0$, $\hat\rho=\epsilon$, and $\mathcal A=0$. Any pair satisfying
\cref{acg problem} is therefore called an \textbf{$\epsilon$-optimal solution}
of \cref{MainProblem}.

The residual criterion \cref{acg problem} also yields a function-gap bound.
Recall that a proper convex function $F$ satisfies the quadratic growth
condition $\mathcal G_\mu^2$ if
$
\frac{\mu}{2}d(x,X^{*})^2\leq F(x)-F^{*}
$
holds for all $x\in \mathbb E$.
Every $\mu$-strongly convex function satisfies this condition. The following
result from \cite{AL40Aujol} relates function gap criterion to the distance from
stationarity.

\begin{lemma}\label{KL}
    Let $F: \R^{n}\rightarrow \R \cup \{+\infty\}$ be a proper lower
    semicontinuous convex function with nonempty minimizer set $X^{*}$, and let
    $F^{*}=\inf F$. If $F$ satisfies $\mathcal G_\mu^2$ for some $\mu>0$, then
    $
        2\mu(F(x)-F^{*})\leq d\left(0,\partial F(x)\right)^{2}
    $
holds for all $x\in \mathbb R^{n}$.
\end{lemma}

Applying \Cref{KL} to $F=\psi$ with $\mu=\bar\mu$, and using
$v\in\partial \psi(y)$, gives
$
    \psi(y)-\psi(x^{*})\leq \frac{\|v\|^2}{2\bar\mu}.
$
Hence, any pair satisfying $\|v\|^2\leq 2\bar\mu\epsilon$ also satisfies
$\psi(y)-\psi(x^{*})\leq \epsilon$.

R-FISTA is a restarted, parameter-free accelerated gradient method. It begins with an aggressive estimate of the strong convexity parameter and checks a restart condition at each iteration. Whenever this condition fails, the method starts a new cycle with a smaller strong convexity estimate.

R-FISTA can be viewed as a compact-domain specialization of RPF-SFISTA~\cite{AL39Sujanani}. Compactness simplifies both the analysis and the restart mechanism: the restart condition depends only on the final iterate of the current cycle, so there is no need to track the best function-value iterate or restart from it. Since each cycle's complexity is independent of its initial function value, any point in $\Dn$ may be used to start the next cycle. This yields a simpler and more flexible method, which we now present.



\begin{algorithm}[H]
\small 
\caption{R-FISTA Method}
\label{alg:RFISTA}
\begin{algorithmic}[1]
\REQUIRE
 scalars
$\chi \in (0,1)$, $\left(\mu_0,\bar M_0\right)\in
\R_{++}^{2}$, tolerance $\epsilon>0$, function pair $(\psi_s,\psi_n)$, and initial point $\bar x_0\in \Dn$.\label{FISTAInput}
\STATE   set $\ell \gets 1$ \label{FISTA1}
\STATE set $\text{stopcrit } \gets \infty$\label{FISTA2}
\ENSURE initialize $j=0$,\label{FISTAInit}
$\underbar M_{\ell}\in [\max\{0.25\bar M_{\ell-1},\bar M_0\},\bar M_{\ell-1}]$,
$\tilde \mu=\mu_{\ell-1}$,
$(A_0,\tau_0,L_0)=(0,1,\underbar M_{\ell})$,
and $x_0\in\Dn$ with $x_0=\bar x_0$ if $\ell=1$; set $y_0=x_0$.
\WHILE{$\text{stopcrit } > \epsilon$ (main outer loop)}\label{FISTA3}
\STATE set $j\gets j+1$;\label{FISTA4}
\STATE set $L_{j}=L_{j-1}$ and evaluate\label{InnerWhileFISTA}:
\begin{equation}\label{def:ak-sfista1}
        a_{j-1}=\frac{\tau_{j-1}+\sqrt{\tau_{j-1}^2+4\tau_{j-1} A_{j-1}L_{j}}}{2L_{j}}, \quad \tx_{j-1}=\frac{A_{j-1}y_{j-1}+a_{j-1} x_{j-1}}{A_{j-1}+a_{j-1}},
        \end{equation}
        \begin{equation}
        y_{j}:=prox_{(1/L_j)\psi_n}\left(\tilde x_{j-1}-\frac{1}{L_j}\nabla \psi_s(\tilde x_{j-1})\right).
        \label{eq:ynext-sfista1}
        \end{equation}
\begin{enumerate}
\item[i.]
\noindent \textbf{While}: 
        $\quad
        \ell_{\psi_s}(y_{j};\tilde x_{j-1})+\frac{(1-\chi) L_{j}}{4}\|y_{j}-\tilde x_{j-1}\|^2< \psi_s(y_{j})
        $
         \[
\qquad \text{set } L_{j} \leftarrow 2 L_{j}
\text{  and update the equations  
        \cref{def:ak-sfista1}, \cref{eq:ynext-sfista1};}
\]

\noindent \textbf{EndWhile}: 
\end{enumerate}
\STATE evaluate\label{FISTA6}
\begin{align}
A_{j}&=A_{j-1}+a_{j-1}, \quad \tau_{j}= \tau_{j-1} + \frac{a_{j-1}\tilde\mu}{2},  \label{eq:taunext-sfista1} \\
s_{j}&=L_{j}(\tilde x_{j-1}-y_{j}),\label{eq:sk}\\
\quad x_{j}&= \frac{1}{\tau_{j}} \left[\frac{\tilde\mu a_{j-1} y_{j}}{2} + \tau_{j-1} x_{j-1}-a_{j-1}s_{j} \right] , \label{eq:xnext-sfista1}\\
v_{j}&=\nabla \psi_s(y_{j})-\nabla \psi_s(\tilde x_{j-1})+s_{j};\label{def:uk}
\end{align}
\IF{
\begin{equation}\label{restart condition}
\|y_j-x_{0}\|^{2} \geq \chi A_{j}L_{j} \|y_{j}-\tilde x_{j-1}\|^2
\end{equation}\textbf{does not hold}} 
\label{RestartCheckStep}
\STATE 
\textbf{restart}, i.e., set $\bar
M_{{\ell}}=L_{j}$, $\mu_{\ell}=0.5\tilde \mu$,  $\ell \gets
\ell+1$, $\text{stopcrit}\gets \infty$\label{FISTA8},
\STATE and execute \textbf{Initialization} above\label{FISTA9};
\ELSE\label{FISTA10}
\STATE set stopcrit $= \|v_j\|$ \label{stopcritFISTA}.\label{FISTA11}
\ENDIF\label{FISTA12}

\ENDWHILE\label{FISTA13}
\OUTPUT pair $(y,v):=(y_{j},v_{j})$ that satisfies \Cref{acg problem}.\label{FISTAOut}
\end{algorithmic}
\end{algorithm}
Several remarks about R-FISTA are now given. R-FISTA performs two types of iterations, inner ACG iterations indexed by $j$ and cycles indexed by $\ell$. At each of its inner ACG iterations, R-FISTA checks in its step~\ref{RestartCheckStep} a key inequality \Cref{restart condition} to determine when to restart and reduce its current estimate $\mu_{\ell}$ of the strong convexity parameter. If the condition fails to hold during R-FISTA's $\ell$-th cycle, then R-FISTA sets $\mu_{\ell}=0.5\mu_{\ell-1}$ and starts its $(\ell+1)$-st cycle with this strong convexity estimate. Finally, note that each of R-FISTA's inner iterations may perform several gradient/proximal operator evaluations due to its inner while loop in step~\ref{InnerWhileFISTA}(i). We express our total complexity of R-FISTA in terms of the number of inner ACG iterations that it performs and the total number of gradient evaluations that it performs. These two quantities are on the same order.

Before stating the main results of the R-FISTA method, the following quantities are introduced
\begin{equation}\label{D0 def}
 \theta=4/(1-\chi), \quad \zeta_{\ell}:=\bar L+\max\{\underbar
M_{\ell}, \theta\bar L\}, \quad Q_{\ell}:=  2\sqrt{2} \sqrt{\frac{\max\{\underbar
M_{\ell},\theta{\bar L}\}}{\mu_{\ell-1}}}.
\end{equation}

The following proposition and theorem state the main complexity results
of the R-FISTA method and key properties of its output. The proofs of \Cref{prop:nest_complex1} and \Cref{Total Complexity} are given in Appendix~\ref{SC-FISTA Appendix Proof}.

\begin{prop}\label{prop:nest_complex1} 
The following statements about the $\ell$-th cycle of R-FISTA hold:
\begin{itemize}
\item[(a)] 
its main while loop either stops successfully with stopcrit $\leq \epsilon$ or the cycle stops because the inequality \Cref{restart condition} in step~\ref{RestartCheckStep} does not hold (in one of its inner iterations) in at most
\begin{equation}\label{eq:eq1}
\left\lceil\left(1+Q_{\ell}\right)\log^{++}\left(\frac{D^2\zeta_{\ell}^2}{\chi \epsilon^2}\right)+1\right\rceil+\left\lceil  \log_2^+\left(\frac{2\bar L}{(1-\chi)\underbar M_{\ell}}\right) \right\rceil
\end{equation}
ACG iterations/gradient evaluations;
\item[(b)] if the main while loop of the cycle terminates successfully with stopcrit $\leq \epsilon$, then the pair $(y,v)$ that R-FISTA outputs satisfies
\begin{equation}\label{OutputRFISTA}
    \|v\|\leq \epsilon, \quad v\in \nabla \psi_s(y)+\partial \psi_n(y),
\end{equation}
i.e., $(y,v)$ is an $\epsilon$-optimal solution of \cref{MainProblem};
\item[(c)] if $\mu_{\ell-1} \in (0,\bar \mu]$, then a restart is never executed during the cycle and the main while loop of the cycle always stops successfully with stopcrit $\leq \epsilon$ 
and a pair $(y,v)$
satisfying \Cref{OutputRFISTA}
in at most \cref{eq:eq1}
ACG iterations/gradient evaluations. 
\end{itemize}
\end{prop}

\begin{theorem}\label{Total Complexity}
Suppose that inputs $\mu_0$ and $\bar M_0$ satisfy $\mu_0\geq \bar \mu$ and 
$\bar M_0\leq 2 \bar L$. Then, R-FISTA terminates with a pair $(y,v)$,
that is an $\epsilon$-optimal solution of \cref{MainProblem},
in at most
\begin{equation}\label{eq:eq1-3}
\mathcal O_1\left(\left\lceil\log^{++}_2\left(\mu_0/\bar \mu\right)\right\rceil\left[\sqrt{\frac{\bar L}{\bar \mu}}\log^{++} \left(\frac{D\bar L}{\epsilon}\right)+\log_2^{+}(\bar L/\bar M_0)\right] \right)
\end{equation}
ACG iterations/gradient evaluations. Hence R-FISTA performs $\mathcal O\left(\sqrt{\bar L/\bar \mu}\log\left(D\bar L/\epsilon\right)\right)$
ACG iterations/gradient evaluations to find an $\epsilon$-optimal solution of \Cref{MainProblem}.
\end{theorem}

\begin{rem}\label{RFISTAAssumption}
The assumptions in 
\Cref{Total Complexity} that the inputs $(\mu_0,\bar M_0)$ satisfy $\mu_0\geq \bar \mu$ and $\bar M_0\leq 2\bar L$ are not necessary to establish convergence results, but simplify our analysis and R-FISTA's complexity bound. 

For more details on how to construct such a $\bar M_0$, see the discussion after the AR-L method. It is also easy to construct $\mu_0\geq \bar \mu$. For example, select $w_0\neq x_0$ and compute
$w_0^{+}=\textrm{prox}_{(1/\bar M_0)\psi_n}(w_0-\frac{1}{\bar M_0} \nabla \psi_s(w_0))$
and $x_0^{+}=\textrm{prox}_{(1/\bar M_0)\psi_n}(x_0-\frac{1}{\bar M_0} \nabla \psi_s(x_0))$.
If $x_0^{+}\neq w_0^{+}$,
then set 
$g_{x_0}=\bar M_0(x_0-x_0^{+})+\nabla \psi_s(x_0^{+})-\nabla \psi_s(x_0)$ and $g_{w_0}=\bar M_0(w_0-w_0^{+})+\nabla \psi_s(w_0^{+})-\nabla \psi_s(w_0)$.
By the optimality conditions of the proximal-gradient steps,
$g_{x_0}\in \partial \psi(x_0^{+})$ and $g_{w_0}\in \partial \psi(w_0^{+})$.
Hence, by the $\bar\mu$-strong convexity of $\psi$, we have
\[
\frac{\inner{g_{x_0}-g_{w_0}}{x_0^{+}-w_0^{+}}}{\|x_0^+-w_0^{+}\|^2}\geq \bar\mu.
\]
Thus, if $x_0^{+}\neq w_0^{+}$, then set
\[
\mu_0=\frac{\inner{g_{x_0}-g_{w_0}}{x_0^{+}-w_0^{+}}}{\|x_0^+-w_0^{+}\|^2}.
\]
If $x_0^{+}= w_0^{+}$, then recompute both points (possibly with different stepsizes) until the resulting proximal-gradient points are distinct.
\end{rem}

\subsection{Complexity of O-IAL, OPF-IAL, and APF-IAL for Linearly-Constrained \texorpdfstring{\\}{Lg} Strongly Convex Optimization}\label{StrongConvComplexity}
This subsection considers the complexity of the O-IAL, OPF-IAL, and APF-IAL methods presented
in \Cref{OIALMethod}, \Cref{OPF-IAL}, and \Cref{APFIAL}, respectively, for solving \cref{LinearConstraints} under 
the assumption that $\psi$ is a $\bar \mu$-strongly convex function, and the assumption that R-FISTA is used (instead of the PF-AR method) to approximately minimize the augmented Lagrangian function. 
All three achieve a near-optimal complexity of $\mathcal O(\hat \epsilon^{-1/2}\log(1/\hat \epsilon))$
for
finding an $(\hat \epsilon,\hat \epsilon)$
approximate solution of \Cref{LinearConstraints}.

The following theorem states the main complexity result of O-IAL under the strong convexity assumption.
\begin{theorem}\label{thm:RFISTAOIALCompl}
Suppose that \Cref{ass:A1} holds and that $\psi$ is $\bar \mu$-strongly convex for some $\bar \mu>0$. Assume that O-IAL is initialized with the function pair $(\psi_s,\psi_n)$, a point $z_0 \in {\Dn}$, a multiplier $p_0\in \R^{m}$, a primal-dual tolerance pair $(\hat \epsilon,\hat \rho) \in \mathbb R^{2}_{++}$, the penalty parameter $c=1/\hat \epsilon$, and the constants $M_0$ and $D_0$ defined in \Cref{remarkPFAR}. In addition, let $\chi\in (0,1)$ and let the initial strong convexity estimate $\mu_0>0$ satisfy the conditions stated in \Cref{RFISTAAssumption}.

Suppose further that, during its $k$-th iteration, O-IAL calls R-FISTA, rather than PF-AR, in step~\ref{OIAL7} with inputs $\chi\in(0,1)$, $\mu_0>0$, initial Lipschitz estimate $\bar M_0=M_0$, tolerance $\epsilon= \min\left\{\frac{\hat \epsilon}{4
D}, \hat \rho\right\}$, function pair $(\cL^s_{c}(\cdot,p_{k-1}),\psi_n(\cdot))$, and initial point $z_{k-1}\in \Dn$. Then O-IAL performs at most
\begin{equation}\label{OIALFISTAcomplexity}
\mathcal O_1\left(\Pi(p_0)\left\lceil\log^{++}_2\left(\mu_0/\bar \mu\right)\right\rceil\left[\sqrt{\frac{\bar L+\|\mathcal A\|^2/\hat \epsilon}{\bar \mu}}\log^{++} \left(\frac{D(\bar L+\|\mathcal A\|^2/\hat \epsilon)}{\min\left\{\frac{\hat \epsilon}{4
D}, \hat \rho\right\}}\right)+\log_2^{+}\left(\frac{\bar L+\|\mathcal A\|^2/\hat \epsilon}{\bar M_0}\right)\right]\right)
\end{equation}
gradient evaluations to obtain an $(\hat \epsilon, \hat \rho)$-approximate solution of \Cref{LinearConstraints}. Consequently, O-IAL performs at most
$\mathcal O(\hat \epsilon^{-1/2}\log(1/\hat \epsilon))$
total gradient evaluations to obtain an $(\hat \epsilon,\hat \epsilon)$-approximate solution of \Cref{LinearConstraints}.
\end{theorem}

\begin{proof}
Suppose that the assumptions of the theorem hold, i.e., O-IAL uses penalty $c=1/\hat \epsilon$ and calls R-FISTA in step~\ref{OIAL7} of its $k$-th iteration with inputs 
$\chi\in (0,1)$, $\mu_0>0$ as in \Cref{RFISTAAssumption}, $\bar M_0=M_0$, tolerance $\epsilon= \min\left\{\frac{\hat \epsilon}{4
D}, \hat \rho\right\}$, function pair $(\cL^s_{c}(\cdot,p_{k-1}),\psi_n(\cdot))$, and initial point $z_{k-1}\in \Dn$. It then follows from these assumptions, the fact that $\mathcal L_c^{s}(\cdot,p_{k-1})$ is 
$\bar L+\|\mathcal A\|^2/\hat \epsilon$-smooth, and \Cref{Total Complexity} that the call to R-FISTA outputs a pair 
$(z_k,r_k)$ that satisfies
 \begin{equation}\label{inclusionRFISTA}
	 r_{k}\in \nabla \psi_s(z_k)+\partial\psi_n(z_k)+\cA^{*}p_{k}, \quad \|r_k\|\leq \min\left\{\frac{\hat \epsilon}{4D},\hat \rho\right\},
        \end{equation}
within at most 
\begin{equation}\label{RFISTAOIALInnerOIAL}
\mathcal O_1\left(\left\lceil\log^{++}_2\left(\mu_0/\bar \mu\right)\right\rceil\left[\sqrt{\frac{\bar L+\|\mathcal A\|^2/\hat \epsilon}{\bar \mu}}\log^{++} \left(\frac{D(\bar L+\|\mathcal A\|^2/\hat \epsilon)}{\min\left\{\frac{\hat \epsilon}{4
D}, \hat \rho\right\}}\right)+\log_2^{+}\left(\frac{\bar L+\|\mathcal A\|^2/\hat \epsilon}{\bar M_0}\right)\right] \right)
\end{equation}
gradient evaluations.
It then follows from the fact that O-IAL uses penalty parameter $c=1/\hat \epsilon$, and  similar arguments as in \Cref{Proof of Theorem O-IAL} and \Cref{OutItCompl} that O-IAL performs at most $\lceil 2\Pi(p_0)\rceil $ outer iterations, i.e., those indexed by $k$. It then follows from this bound, the bound in \Cref{RFISTAOIALInnerOIAL}, and both relations in \Cref{inclusionRFISTA} that the number of gradient evaluations O-IAL performs to find an $(\hat \epsilon,\hat \rho)$-approximate solution of \Cref{LinearConstraints} is on the order of the quantity in \Cref{OIALFISTAcomplexity}. The final conclusion of \Cref{thm:RFISTAOIALCompl} is then immediate from the complexity bound in \Cref{OIALFISTAcomplexity} and setting $\hat \rho=\hat \epsilon$.
\end{proof}

\begin{remark}
It can also be shown under similar assumptions as in \Cref{thm:RFISTAOIALCompl} that if the parameter-free AL method, OPF-IAL (\Cref{alg:OPF-IAL}), calls R-FISTA (instead of PF-AR) during its step~\ref{ARorFISTACall}, it also achieves a near-optimal primal-dual complexity bound of $\mathcal O(\hat \epsilon^{-1/2}\log(\hat \epsilon^{-1}))$ to find an $(\hat \epsilon,\hat \epsilon)$-approximate solution of \Cref{LinearConstraints}.
We omit the proof since it follows from nearly identical arguments as the proof of \Cref{thm:RFISTAOIALCompl}.
\end{remark}
The theorem below now states the last-iterate complexity  of the adaptive APF-IAL method (\Cref{alg:APF-IAL}) under the strong convexity assumption of $\psi$.

\begin{theorem}\label{thm:RFISTAPFLCompl}
Suppose that \Cref{ass:A6,ass:A7} hold, $\psi$ is $\bar \mu$-strongly convex for some $\bar \mu>0$, and $K$ is as in \Cref{Outer Iterations}. Assume that APF-IAL is initialized with the function pair $(\psi_s,\psi_n)$, a point $z_0 \in {\Dn}$, a multiplier $p_0\in \R^{m}$, a primal-dual tolerance pair $(\hat \epsilon,\hat \rho) \in \mathbb R^{2}_{++}$, parameters $c_1>0$, $\alpha>1$, $\omega>1$, and $\tilde \epsilon_1\geq \hat \rho$, and constants $M_0$ and $D_0$ defined in \Cref{remarkPFAR}. In addition, let $\chi\in (0,1)$ and let the initial strong convexity estimate $\mu_0>0$ satisfy the conditions stated in \Cref{RFISTAAssumption}.

Suppose further that, during its $k$-th iteration, APF-IAL calls R-FISTA, rather than PF-AR, in step~\ref{APF10} with inputs $\chi\in(0,1)$, $\mu_0>0$, initial Lipschitz estimate $\bar M_0=M_0$, tolerance $\epsilon= \epsilon_k$, function pair $(\cL^s_{c}(\cdot,p_{k-1}),\psi_n(\cdot))$, and initial point $z_{k-1}\in \Dn$. Then APF-IAL performs at most
\begin{equation}\label{APFIALFISTAcomplexity}
\mathcal O_1\left(\left\lceil\log^{++}_2\left(\mu_0/\bar \mu\right)\right\rceil\left[K\sqrt{\frac{\bar L}{\bar \mu}}+\sqrt{\frac{\|\mathcal A\|^2}{\hat \epsilon \bar \mu}}\right]\log^{++} \left(\frac{D(\bar L+\|\mathcal A\|^2/\hat \epsilon)}{\hat \rho}\right)\right)
\end{equation}
gradient evaluations to obtain an $(\hat \epsilon, \hat \rho)$-approximate solution of \Cref{LinearConstraints}. Consequently, APF-IAL performs at most
$\mathcal O(\hat \epsilon^{-1/2}\log(1/\hat \epsilon))$
total gradient evaluations to obtain an $(\hat \epsilon,\hat \epsilon)$-approximate solution of \Cref{LinearConstraints}.
\end{theorem}
\begin{proof}
It follows from the assumptions of the theorem, \Cref{Total Complexity}, and the fact that
\Cref{Bound on EpsilonK} implies that $\epsilon_k\geq \hat \rho$,
that the call to R-FISTA during the $k$-th iteration of APF-IAL outputs a pair 
$(z_k,r_k)$ that satisfies
 \begin{equation}\label{inclusionRFISTA2}
	 r_{k}\in \nabla \psi_s(z_k)+\partial\psi_n(z_k)+\cA^{*}p_{k}, \quad \|r_k\|\leq \epsilon_k
        \end{equation}
within at most 
\begin{equation}\label{RFISTAAPFInner}
\mathcal O_1\left(\left\lceil\log^{++}_2\left(\mu_0/\bar \mu\right)\right\rceil\left[\sqrt{\frac{\bar L+c_k\|\mathcal A\|^2}{\bar \mu}}\log^{++} \left(\frac{D(\bar L+c_k\|\mathcal A\|^2)}{\hat \rho}\right)\right]\right)
\end{equation}
gradient evaluations.
It then follows from the bound in \Cref{Outer Iterations}
that
APF-IAL performs at most 
\[K:=2+\left\lceil\log_{\alpha}^{+} \left(\max\left\{\frac{4\alpha^2 D(M_n+\nabla_s+\tilde\epsilon_1)}{\hat \epsilon (c_1\bar d\sigma^{+}_{\mathcal A})}, \frac{2\alpha^2\|p_0\|}{\hat \epsilon c_1}\right\}\right)\right\rceil\]
outer iterations
to find an $(\hat \epsilon,\hat \rho)$-approximate solution of \Cref{LinearConstraints}.
Moreover, the bound on $c_k$ in \Cref{Bound Penalty} implies that
$c_k=\mathcal O(1/\hat \epsilon)$. It then follows from summing the complexity bound in \Cref{RFISTAAPFInner} from $k=1$ to $K$ (using a similar argument as in the proof of \Cref{Complexity APF-IAL})
that the number of gradient evaluations that APF-IAL performs to find a $(\hat \epsilon,\hat \rho)$ solution is on the order of the quantity in \Cref{APFIALFISTAcomplexity}. The last conclusion of \Cref{thm:RFISTAPFLCompl} then follows from this bound and setting $\hat \rho=\hat \epsilon$.
\end{proof}

\section{Numerical Experiments}\label{NumericalExperiments}
This section compares the numerical performance of our inexact methods O-IAL (\Cref{alg:O-IAL}) and APF-IAL (\Cref{alg:APF-IAL}) with that of a representative PAL method, which we refer to as ProxALM (Algorithm~5 in \cite{LuM}), on six classes of linearly constrained convex optimization problems, including both convex and strongly convex instances. To make the comparison as fair as possible, we use exactly the same termination criterion for all three methods. Specifically, each method terminates when it produces a triple $(z,p,r)$ satisfying
\begin{equation}\label{CompCriterion}
 r \in \nabla \psi_s(z)+\partial \psi_n(z)+\cA^{*}p, \quad \|r\|\leq
\hat \rho, \quad \|\cA z-b\|\leq \hat \epsilon.
\end{equation}
In all our experiments, we set $\hat \rho=\hat \epsilon=10^{-5}$. We vary the strong convexity modulus $\bar \mu$, where $\bar \mu=0$ corresponds to the convex setting.

We next describe the implementation details of the three methods considered in our numerical experiments. ProxALM takes as input $\epsilon=10^{-5}$, a randomly generated initial primal-dual point $(x_0,\lambda_0)\in \Dn\times \mathbb R^{m}$, $M=10$, $\delta=0.8$, the initial penalty parameter $\rho_0=\max\left\{10,\bar \mu+\sqrt{\bar \mu^2+4}\right\}$, $\alpha_0=1$, the initial tolerance $\eta_0=10$ for the first proximal AL subproblem, the penalty-parameter increase factor $\zeta=1.1$, and the proximal AL subproblem tolerance decrease factor $\sigma=1/(\zeta+0.4)=2/3$. At every iteration $k\geq 0$, the penalty parameter and proximal AL subproblem tolerance are updated according to $\rho_k=\rho_0\zeta^{k}$ and $\eta_k=\eta_0\sigma^{k}$, respectively.

For many of our test problems $\bar \mu<5$, and hence $\rho_0=10$. Some problems have $\bar \mu\geq 5$, in which case $\rho_0=\bar \mu+\sqrt{\bar \mu^2+4}$. Thus, in all cases, the required condition $\rho_0>\left(\bar \mu+\sqrt{\bar \mu^2+4}\right)/2$ is satisfied. As prescribed in step~2 of ProxALM, we use the strongly convex version of FISTA presented in Algorithm~2 of \cite{LuM} to approximately minimize each proximal augmented Lagrangian subproblem. Algorithm~2 is implemented exactly as specified in \cite{LuM}. Finally, we modify the termination criterion in step~4 of ProxALM so that it uses the common criterion in \Cref{CompCriterion}.

Since ProxALM uses a strongly convex APG/FISTA method to minimize its PAL subproblems, O-IAL and APF-IAL use R-FISTA (\Cref{alg:RFISTA}) to minimize their AL subproblems. Both methods choose the R-FISTA-specific input parameter
$\chi=0.001$ and use $L_j=\beta L_j$, where $\beta=1.25$ instead of $2$. In particular, $1/\beta=0.8$, which matches the value $\delta=0.8$ used by ProxALM. They also use the initial strong convexity estimate $\mu_0=2(\psi_s(w_0)-\psi_s(z_0)-\inner{\nabla \psi_s(z_0)}{w_0-z_0})/(\|w_0-z_0\|^2)$, where $w_0,z_0\in \Dn$ and $z_0$ is the initial point. When $\psi_n$ is the indicator function of a closed convex set, this choice satisfies $\mu_0\geq \bar \mu$.

O-IAL uses the same randomly generated initial primal and multiplier points as ProxALM, together with the fixed penalty parameter $c=10$, the tolerances $\hat \epsilon=10^{-5}$ and $\hat \rho=10^{-5}$, and a domain-diameter parameter $D>0$. When the diameter of $\Dn$ can be estimated from the problem data, we use this estimate for $D$. Otherwise, if the diameter cannot be readily estimated or $\Dn$ is not necessarily compact, we set $D=1$. Hence, some experiments intentionally test the algorithms beyond the theoretical setting considered in this paper, which assumes that $\Dn$ is compact.

APF-IAL uses the same initial primal-dual point as O-IAL and ProxALM, the initial penalty parameter $c_1=10$, the penalty-parameter increase factor $\alpha=1.1$, $\omega=1.5$, and the initial tolerance $\tilde \epsilon_1=10$. Here, $1/\omega=2/3$ matches the value $\sigma=2/3$ used by ProxALM, while $\tilde \epsilon_1=10$ matches its initial tolerance $\eta_0=10$.

All experiments are performed in MATLAB R2025b on a 2023 MacBook Pro equipped with an 8-core CPU. The latest codes to reproduce the experiments are available at this
\href{https://github.com/asujanani6/Parameter_Free_ALProj_Codes}
{clickable-link} or with URL 
\url{https://github.com/asujanani6/Parameter_Free_ALProj_Codes}.
 Each of the following six subsections reports numerical results for a different class of linearly constrained problems.

\subsection{Quadratic Objective over the Simplex}
\label{QuadraticExp}

Given dimensions $(m,n)\in\mathbb N^2$, scalars
$(\xi,\tau)\in\mathbb R_{++}^2$, matrices
$A,C\in\mathbb R^{m\times n}$ and $B\in\mathbb R^{n\times n}$,
a positive diagonal matrix $D\in\mathbb R^{n\times n}$, and vectors
$d,b\in\mathbb R^m$, we consider
\[
\begin{aligned}
\min_x\quad
&
\psi_s(x)
:=
\frac{\xi}{2}\|DBx\|^2
+
\frac{\tau}{2}\|Cx-d\|^2
\\
\text{s.t.}\quad
&
x\in\Delta^n,
\qquad
Ax=b,
\end{aligned}
\]
where
\[
\Delta^n
:=
\left\{
x\in\mathbb R_+^n:
\boldsymbol{1}_n^\top x=1
\right\}
\]
is the probability simplex (and hence where $\boldsymbol{1}_n$ is the vector of all ones).
We vary the dimensions over
$10\leq m\leq700$ and $50\leq n\leq1100$, and use target
curvatures
$\bar L\in\{200,500,1000\}$ and
$\bar\mu\in\{5,10,100\}$.
The entries of $A$, $C$, $d$, and a seed matrix $\widetilde B$
are independently sampled from $\mathcal U[0,1]$, while the
diagonal entries of $D$ are sampled uniformly from
$\{1,\ldots,\alpha\}$. The matrix $B$ and the coefficients
$(\xi,\tau)$ are then spectrally calibrated so that
$\lambda_{\max}(\nabla^2\psi_s)=\bar L$ and
$\lambda_{\min}(\nabla^2\psi_s)=\bar\mu$.
All tested instances in this subsection are strongly convex.
We set
$
b=A(\boldsymbol{1}_n/n),
$
and generate
$x_0=\widehat x/(\boldsymbol{1}_n^\top\widehat x)$, where the
entries of $\widehat x$ are independently sampled from
$\mathcal U[0,1]$. 
The numerical comparisons are
reported in \Cref{quadraticLC}.

As seen from \Cref{quadraticLC}, O-IAL and APF-IAL were, on average, $7.24$ and $52.37$ times faster than ProxALM, respectively. O-IAL was faster than ProxALM on $13$ instances, slower on one, and tied on one. It also consistently achieved dual-stationarity residuals on the order of $10^{-10}$ as it maintains high dual stationarity at each of its iterations. APF-IAL was faster than ProxALM on all 15 instances.

\begin{table}[H]
\centering
\caption{Quadratic Objective Over Simplex Comparison}
\resizebox{\columnwidth}{!}
{
\input{table_QuadraticLC.tex}
}
\label{quadraticLC}
\end{table}

\subsection{Logistic Regression over an $\ell_1$ Ball}
\label{LogRegExp}

Given a number of observations $N\in\mathbb N$, dimensions
$(m,n)\in\mathbb N^2$, a data matrix $X\in\mathbb R^{N\times n}$,
labels $y\in\{-1,1\}^N$, a constraint matrix
$A\in\mathbb R^{m\times n}$, and $b\in\mathbb R^m$, we consider
\[
\begin{aligned}
\min_x\quad
&
\psi_s(x)
:=
\frac{1}{N}
\sum_{i=1}^N
\log\left(1+\exp(-y_iX_i^\top x)\right)
+
\frac{\bar\mu}{2}\|x\|^2
\\
\text{s.t.}\quad
&
x\in\mathbb B_1^n,
\qquad
Ax=b,
\end{aligned}
\]
where $X_i^\top$ denotes the $i$-th row of $X$ and
\[
\mathbb B_1^n
:=
\left\{
x\in\mathbb R^n:\|x\|_1\leq1
\right\}
\]
is the $\ell_1$ ball.
In our experiments, we vary parameters $50\leq N\leq650$, $50\leq m\leq1000$, and
$100\leq n\leq2500$, and we vary the target curvature pair from
$30\leq\bar L\leq1000$ and $0\leq\bar\mu\leq5$. Hence, the instances with
$\bar\mu=0$ are convex, while those with $\bar\mu>0$ are strongly
convex.
We generate the matrix $X$ according to
\[
X
=
\frac{\sqrt{4N(\bar L-\bar\mu)}}{\|\widetilde X\|_2}
\widetilde X
\]
where $\widetilde X$ is a fully dense Gaussian matrix.
The entries of $y$ are independently and uniformly sampled from
$\{-1,1\}$.

We set our initial point $x_0=0$. The entries of the second to $m$-th row of $A$ are independently sampled from the standard normal
distribution, while its first row is generated as
$
A_{1,:}=2\boldsymbol{1}_n^\top.
$
Finally, we set
$
b
=
A\left(\boldsymbol{1}_n/2n\right).
$
The numerical comparisons are reported in \Cref{LogisticReg}.

\begin{table}[H]
\centering
\caption{Logistic Regression Comparison}
\resizebox{\columnwidth}{!}
{
\input{table_LogReg.tex}
}
\label{LogisticReg}
\end{table}

O-IAL and APF-IAL were, on average, $8.30$ and $34.99$ times faster than ProxALM, respectively, with O-IAL and APF-IAL outperforming ProxALM on all $15$ instances. ProxALM often reduced the primal-feasibility residual far below the prescribed tolerance of $10^{-5}$, whereas APF-IAL provided a better balance between primal feasibility, dual stationarity, and computational cost.

\subsection{Elastic-Net Least-Squares Regression}
\label{ElasticNetExp}

Given a number of observations $N\in\mathbb N$, dimensions
$(m,n)\in\mathbb N^2$, a data matrix $X\in\mathbb R^{N\times n}$,
a response vector $y\in\mathbb R^N$, a constraint matrix
$A\in\mathbb R^{m\times n}$, and $b\in\mathbb R^m$, we consider
\[
\begin{aligned}
\min_x\quad
&
\psi_s(x)+\psi_n(x)
:=
\frac{1}{2N}\|Xx-y\|^2
+
\frac{\bar\mu}{2}\|x\|^2
+
\frac{1}{\sqrt{n}}\|x\|_1
\\
\text{s.t.}\quad
&
Ax=b,
\end{aligned}
\]
We vary $300\leq N\leq1560$, $50\leq m\leq1000$,
$100\leq n\leq1500$, and the target curvature pair over
$15\leq\bar L\leq480$ and $0\leq\bar\mu\leq5$.
The matrix $X$ is generated as
\[
X
=
\frac{\sqrt{N(\bar L-\bar\mu)}}{\|\widetilde X\|_2}
\widetilde X
\]
where $\widetilde X \in\mathbb R^{N\times n}$ is a fully dense but rank-deficient Gaussian design matrix.

A sparse reference vector $x_{\mathrm{ref}}$ is generated with
approximately $10\%$ nonzero entries and normalized to have unit
Euclidean norm. We generate
$y=Xx_{\mathrm{ref}}+\sigma\varepsilon$, where $\varepsilon$ is a
standard Gaussian vector and
$\sigma=0.05\max\{\|Xx_{\mathrm{ref}}\|/\sqrt N,1\}$.
The initial point is
$
x_0
=
x_{\mathrm{ref}}
+
(\boldsymbol{1}_n/\sqrt n).
$
The matrix $A$ is initially Gaussian, after which its first row is
shifted along $x_0-x_{\mathrm{ref}}$ so that
$A_{1,:}(x_0-x_{\mathrm{ref}})=1$. Finally,
$b=Ax_{\mathrm{ref}}$. The results are reported in
\Cref{ElasticNetTable}.

\begin{table}[H]
\centering
\caption{Elastic-Net Least Squares Regression Comparison}
\resizebox{\columnwidth}{!}
{
\input{table_ML1_ElasticNet.tex}
}
\label{ElasticNetTable}
\end{table}
O-IAL and APF-IAL were, on average, $4.31$ and $55.14$ times faster than ProxALM, respectively, with O-IAL faster on all $15$ instances. This problem class produced the largest average speedup for APF-IAL among the six classes.

\subsection{Group-Sparse Huberized-Hinge SVM}
\label{HuberizedSVMExp}

Given a number of observations $N\in\mathbb N$, dimensions
$(m,n)\in\mathbb N^2$, a data matrix $X\in\mathbb R^{N\times n}$,
labels $y\in\{-1,1\}^N$, a constraint matrix
$A\in\mathbb R^{m\times n}$, and $b\in\mathbb R^m$, we consider
\[
\begin{aligned}
\min_x\quad
& \psi_s(x)+\psi_n(x):=
\frac{1}{N}
\sum_{i=1}^N
h_\delta\left(1-y_iX_i^\top x\right)
+
\frac{\bar\mu}{2}\|x\|^2
+
\lambda\sum_{G\in\mathcal G}\|x_G\|_2
\\
\text{s.t.}\quad
&
Ax=b,
\end{aligned}
\]
where $\delta=2$, $\lambda=1/\sqrt n$, and $\mathcal G$ partitions
the variables into contiguous groups of at most ten coordinates.
The Huberized-hinge loss is defined by
\[
h_\delta(t)
=
\begin{cases}
0, & t\leq0,\\[1mm]
t^2/(2\delta), & 0<t<\delta,\\[1mm]
t-\delta/2, & t\geq\delta.
\end{cases}
\]
We vary parameters $260\leq N\leq7000$, $100\leq m\leq2000$,
$200\leq n\leq4000$, and the target curvature pair over
$20\leq\bar L\leq800$ and $0\leq\bar\mu\leq5$. The matrix $X$ is generated in a similar way as the previous two problem classes.

A group-sparse reference classifier $x_{\mathrm{ref}}$ is generated
by activating approximately $20\%$ of the groups and normalizing the
result to have unit Euclidean norm. Binary labels $y$ are obtained by
adding Gaussian noise with standard deviation
$0.1\max\{\|Xx_{\mathrm{ref}}\|/\sqrt N,1\}$ to
$Xx_{\mathrm{ref}}$ and taking componentwise signs. We set initial point
$x_0=0$.

Each row of constraint matrix $A$ compares the average linear scores of two synthetic
cohorts. The first row compares the highest- and lowest-scoring
cohorts under $x_{\mathrm{ref}}$ and is scaled so that
$A_{1,:}x_{\mathrm{ref}}=1$. The remaining rows use disjoint,
randomly paired cohorts and are normalized to have unit Euclidean
norm. Finally, $b=Ax_{\mathrm{ref}}$. The numerical comparisons are reported in
\Cref{HuberizedSVMTable}.

\begin{table}[H]
\centering
\caption{Group-Sparse Huberized-Hinge SVM Comparison}
\resizebox{\columnwidth}{!}
{
\input{table_ML2_HuberizedSVM.tex}
}
\label{HuberizedSVMTable}
\end{table}

APF-IAL was, on average, $15.42$ times faster than ProxALM, whereas the average speedup factor of O-IAL relative to ProxALM was $0.94$, indicating that ProxALM was marginally faster on average than O-IAL. However, O-IAL consistently achieved much smaller dual-stationarity residuals of approximately $2.5\times 10^{-10}$ than ProxALM. APF-IAL was faster than ProxALM on all 15 instances considered.

\subsection{Bayesian $D$-Optimal Experimental Design}
\label{BayesianDOptimalExp}
Given an information-matrix dimension $N\in\mathbb N$, dimensions
$(m,n)\in\mathbb N^2$, candidate experiment vectors
$v_1,\ldots,v_n\in\mathbb R^N$, a constraint matrix
$A\in\mathbb R^{m\times n}$, and $b\in\mathbb R^m$, we consider
\[
\begin{aligned}
\min_x\quad
&
\psi_s(x)
:=
-\gamma\log\det\left(
I_N+\sum_{i=1}^n x_i v_i v_i^\top
\right)
+
\frac{\bar\mu}{2}\|x\|^2
\\
\text{s.t.}\quad
&
x\in\mathcal B^n,
\qquad
Ax=b,
\end{aligned}
\]
where
\[
\mathcal B^n
:=
\left[0,\frac{2}{\sqrt n}\right]^n.
\]

We vary $3\leq N\leq17$, $5\leq m\leq200$,
$25\leq n\leq300$, and the target curvature pair over
$5\leq\bar L\leq100$ and
$0\leq\bar\mu\leq5\times10^{-2}$. The dimensions satisfy
$n>N(N+1)/2$.

The candidate vectors are independently sampled from the standard
normal distribution and normalized so that $\|v_i\|=1/2$. Letting
$V=[v_1\ \cdots\ v_n]\in\mathbb R^{N\times n}$ and
$Q_0=V^\top V$, we choose
\[
\gamma
=
\frac{\bar L-\bar\mu}
{\lambda_{\max}(Q_0\circ Q_0)},
\]
where $\circ$ denotes the Hadamard product. 
We set
initial point $x_0=0$.
The first row of $A$ imposes the total-budget equality
$\boldsymbol{1}_n^\top x=1$. The remaining rows impose category
budgets on a random partition of all but one of the candidate
experiments into disjoint experiment families. Finally, we set
$b=A(\boldsymbol{1}_n/n)$. The results are reported in
\Cref{DOptimalTable}.

\begin{table}[H]
\centering
\caption{Bayesian Optimal Design Comparison}
\resizebox{\columnwidth}{!}
{
\input{table_ML3_BayesianDOptimal.tex}
}
\label{DOptimalTable}
\end{table}

O-IAL and APF-IAL were, on average, $1.78$ and $16.79$ times faster than ProxALM, respectively. O-IAL was faster than ProxALM on $10$ of the $15$ instances, while APF-IAL was faster than ProxALM on all 15 instances.
\subsection{Huberized Quantum State Tomography Semidefinite Program (SDP)}
\label{QuantumTomographyExp}

Let $\mathbb H^d$ denote the space of $d\times d$ complex Hermitian
matrices, equipped with the inner product
$\langle U,V\rangle_F:=\operatorname{Re}\operatorname{tr}(U^*V)$,
and set $n=d^2$. 
Given measured observables $O_1,\ldots,O_N\in\mathbb H^d$,
measurements $y\in\mathbb R^N$, a constraint matrix
$A\in\mathbb R^{m\times d^2}$, and $b\in\mathbb R^m$, we consider
\[
\begin{aligned}
\min_{\rho\in\mathbb H^d}\quad
&
\psi_s(\rho)+\psi_n(\rho):=\frac{\gamma}{N}\sum_{i=1}^N
h_\delta\left(\langle O_i,\rho\rangle_F-y_i\right)
+
\frac{\bar\mu}{2}\|\rho\|_F^2
+
\underbrace{\iota_{\mathcal D_d}(\rho)}_{\psi_n(\rho)}
\\
\text{s.t.}\quad
&
A\operatorname{svec}_{\mathbb H}(\rho)=b,
\end{aligned}
\]
where $\iota_{\mathcal D_d}$ is the indicator function of the
density-matrix spectrahedron and $h_{\delta}(t)$ is defined as follows:
\[
\mathcal D_d
:=
\left\{
\rho\in\mathbb H^d:
\rho\succeq0,\ 
\operatorname{tr}(\rho)=1
\right\}, \quad h_\delta(t)
:=
\begin{cases}
t^2/(2\delta), & |t|\leq\delta,\\
|t|-\delta/2, & |t|>\delta.
\end{cases}
\]
The operator $\operatorname{svec}_{\mathbb H}(H)$ is a linear isometry that takes in a matrix $H= (H_{ij})\in\mathbb H^d $ as input
and outputs a vector $\mathbb R^{d^2}$:
\[
\operatorname{svec}_{\mathbb H}(H)
:=
\left[
H_{11},\ldots,H_{dd},\,
\big(\sqrt{2}\operatorname{Re}(H_{ij})\big)_{i<j},\,
\big(\sqrt{2}\operatorname{Im}(H_{ij})\big)_{i<j}
\right]^\top.
\]

We set $\delta=0.5$ and vary $60\leq N\leq970$, $8\leq m\leq86$, and
$12\leq d\leq64$. Hence, $n$ satisfies $144\leq n\leq4096$.
The target curvatures satisfy
$14\leq\bar L\leq630$ and $0\leq\bar\mu\leq5$.

Each $O_i$ is generated as a dense, traceless Hermitian matrix and
normalized so that $\|O_i\|_F=1$. We set
$\gamma
=
[N\delta(\bar L-\bar\mu)]/\|B\|_2^2,
$
where
$B\in\mathbb R^{N\times d^2}$ has rows
$
B_i
=
\operatorname{svec}_{\mathbb H}(O_i)^\top.
$
A normalized complex Gaussian vector
$g\in\mathbb C^d$ is generated, and the reference density matrix and
measurements are defined by $\rho_{\mathrm{ref}}
=
0.35\,gg^*
+
0.65\,\frac{I_d}{d}$ and $y_i
=
\langle O_i,\rho_{\mathrm{ref}}\rangle_F+\varepsilon_i$, where $\varepsilon_i\sim\mathcal U[-0.02,0.02].$
The initial point is the maximally mixed state $\rho_0=I_d/d$.
The matrix $A$ is initially Gaussian, after which its first row is
shifted along
$\operatorname{svec}_{\mathbb H}(\rho_0-\rho_{\mathrm{ref}})$, and
we set
$
b
=
A\operatorname{svec}_{\mathbb H}(\rho_{\mathrm{ref}}).
$
The results are reported in \Cref{QuantumTomographyTable}.

\begin{table}[H]
\centering
\caption{Huberized Quantum State Tomography Comparison}
\resizebox{\columnwidth}{!}
{
\input{table_quantSDP4_QuantumTomography.tex}
}
\label{QuantumTomographyTable}
\end{table}
O-IAL and APF-IAL were, on average, $18.30$ and $35.24$ times faster than ProxALM, respectively. O-IAL consistently attained dual-stationarity residuals near $1.7\times10^{-10}$, while APF-IAL achieved the smallest runtimes and kept both residuals within the requested tolerances.

\subsection{Summary of Numerical Results}
In conclusion, APF-IAL was the fastest method in these experiments and struck the best balance between achieving small primal-feasibility and dual-stationarity residuals. Overall, O-IAL was the second-fastest method and frequently achieved exceptionally small dual-stationarity residuals. This behavior is consistent with its design: O-IAL explicitly maintains approximate stationarity throughout its iterations and consequently produces dual residuals near $10^{-10}$, despite the requested tolerance being only $10^{-5}$.

ProxALM was the least computationally efficient of the three methods. Although it often reduced the primal-feasibility residual well below the prescribed threshold, its main bottleneck was efficiently reducing the dual-stationarity residual. This behavior can be explained by the distinction between AL and PAL methods. After the multiplier update in an AL method, approximate stationarity of the AL subproblem directly controls the dual stationarity residual of the original linearly constrained problem at the updated multiplier. By contrast, the stationarity condition for a PAL subproblem contains an additional proximal-displacement term. Consequently, accurate minimization of the PAL subproblem does not translate as directly into a small stationarity residual for the original problem unless the proximal displacement is also sufficiently small. This helps explain why ProxALM had greater difficulty efficiently reducing the dual-stationarity residual.

\section{Conclusion}
In this paper, we developed three inexact augmented Lagrangian methods, O-IAL, OPF-IAL, and APF-IAL, that attain an optimal complexity bound of $\mathcal O(\epsilon^{-1})$ in the convex setting and a near-optimal bound of $\mathcal O(\epsilon^{-1/2}\log(\epsilon^{-1}))$ in the strongly convex setting. Among these methods, OPF-IAL and APF-IAL are parameter-free. To solve the inexact AL subproblems, our methods employ either PF-AR or R-FISTA, both of which are accelerated parameter-free methods developed in this paper. Finally, numerical experiments on six important classes of problems demonstrate that O-IAL and APF-IAL are often $5$--$50$ times faster than a representative proximal augmented Lagrangian method.

\newpage
\appendix

\section{Proof of \Cref{AR-LComplex}}\label{ProofAR-L}

The following lemma characterizes properties of the output $(\xpp{s},\xp{s},M_s)$ of the Backtracking function.
\begin{lemma}
   For every $s\geq 1$, the following relations hold:
   \begin{align}
       &\xpp{s}:=x^{++}(2M_s,\sigma_s,\bar x_s;x_s)\label{eq1}\\
       &\xp{s}:=x^{+}(2M_s;x_s):=\argmin_{u} \inner{\nabla \psi_s(x_s)}{u}+\psi_n(u)+M_s\|u-x_s\|^2\label{eq2}\\
       & \psi_s(x_s^{++})-\psi_s(x_s)-\inner{\nabla \psi_s(x_s)}{x_s^{++}-x_s}\leq \frac{M_s}{2}\|\xpp{s}-x_s\|^2\label{eq3}\\
       & \|\nabla \psi_s(\xp{s})-\nabla \psi_s(x_s)\|\leq 2M_{s}\|x_s^{+}-x_s\|\label{eq4}\\
       & M_0\leq M_{s}\leq \max\left\{M_0, 2\bar L\right\}\label{eq5}.
   \end{align}
\end{lemma}
\begin{proof}
    Relations \Cref{eq1,eq2}
    follow immediately from the way $x^{++}$ and $x^{+}$ are computed in steps \ref{Back2} and \ref{Back3} of Backtracking and the fact that the $s$-th iteration of \Cref{alg:ARL} calls 
    Backtracking with inputs ($\psi_s$, $\psi_n$, $\sigma_s$, $x_s$, $M_{s-1}$) and outputs $(x_{s}^{++},x_s^{+},M_s):=(x^{++},x^{+},M)$.
    Likewise, relations \Cref{eq3,eq4}
    follow immediately from the checks that Backtracking performs in \Cref{linesearch1,linesearch2}.

    To see \Cref{eq5}, observe that \Cref{eq3,eq4}
    in Backtracking always hold if $M_{j}\geq \bar L$
    for some $j$ generated during Backtracking
    because $\psi_s$ is assumed to be $\bar L$-smooth. 
    It follows from this fact, the fact that $M_{j+1}$ is set inside of Backtracking as $M_{j+1}=2M_{j}$ if \Cref{eq3,eq4} do not hold, and the fact that \Cref{alg:ARL} calls Backtracking with $M_0=M_{s-1}$ that 
    $M_{s-1}\leq M_{s}\leq \max\left\{M_{s-1}, 2\bar L\right\}$.
    Relation \Cref{eq5} then follows from this fact and a simple induction argument.
\end{proof}
The following lemma provides key bounds which involve the gradient mapping. We omit the proof of the lemma since the proof is identical to the proofs of (3.5) and (3.6) in \cite{AL27Lan} with $(\eta,\sigma,\bar x,x)$
replaced by $(2M_{s},\sigma_{s},\bar x_{s},x_{s})$.
\begin{lemma}
For every $s\geq 1$, we have that
\begin{align}
    &\|x^{+}(2M_s+\sigma_s;x_s)-x^{++}(2M_s,\sigma_s,\bar x_s;x_s)\|\leq \frac{\sigma_s}{2M_s+\sigma_s}\|x_s-\bar x_s\|;\label{GM1}\\
    &\|G_{2M_s}(x_s)\|\leq \|G_{\sigma_s+2M_s}(x_s)\|\label{GM2}.
\end{align}
\end{lemma}

The next lemma characterizes key descent properties of $x_{s}^{++}$.

\begin{lemma}\label{KeyDescent}
For every $s\geq 1$, we have that the following relation holds:
\begin{align}
&\left[\psi_s(x_s^{++})+\psi_n(x_s^{++})+\frac{\sigma_s}{2}\|x_s^{++}-\bar x_s\|^2\right]-\left[\psi_s(u)+\psi_n(u)+\frac{\sigma_s}{2}\|u-\bar x_s\|^2 \right]\nonumber\\
&\leq M_s\|u-x_s\|^2-\frac{\sigma_s+2M_s}{2}\|u-x_s^{++}\|^2-\frac{M_s}{2}\|x_s^{++}-x_s\|^2 \quad \forall u\in \Dn.\label{DescentIneq}
\end{align}
\end{lemma}
\begin{proof}
    It follows from \Cref{eq1} that
    \[\psi_s(x_s^{++})-\psi_s(x_s)-\inner{\nabla \psi_s(x_s)}{x_s^{++}-x_s}\leq\frac{M_s}{2}\|x_s^{++}-x_s\|^2\]
where $x_s^{++}:=x^{++}(2M_s,\sigma_s,\bar x_s;x_s)$. The result is then immediate from 
(3.8) in \cite{AL27Lan} with $(x^{++},\eta,\bar x,x,M)$ in the statement replaced with $(x_s^{++},2M_s,\sigma_s,\bar x_s,x_s,M_s)$.
\end{proof}

The following proposition gives a key bound on the norm of the gradient mapping in terms of the distances between $x_s^{*}$ and $x_s$ and $x_s^{*}$ and $\bar x_s$.

\begin{proposition}
    For any $s\geq 1$, we have that the following relation holds:
    \begin{equation}\label{GradientMapBound}
    \|G_{2M_s}(x_s)\|\leq (3\sigma_s+4M_s)\|x_s^{*}-x_s\|+\sigma_s\|x_s^{*}-\bar x_s\|.
    \end{equation}
\end{proposition}
where $x_s^{*}$ is as in \Cref{eq:subproblem_unconstrained_searchL} and $G_{2M_s}(x_s):=2M_s(x_s-x_s^{+})$.

\begin{proof}
It follows from the definition
of $x_s^{*}$ in \Cref{eq:subproblem_unconstrained_searchL}
and \Cref{KeyDescent} with 
$u=x_s^{*}$ that the following holds:
\begin{align*}
0&\overset{\Cref{eq:subproblem_unconstrained_searchL}}{\leq} \left[\psi_s(x_s^{++})+\psi_n(x_s^{++})+\frac{\sigma_s}{2}\|x_s^{++}-\bar x_s\|^2\right]-\left[\psi_s(x_s^{*})+\psi_n(x_s^{*})+\frac{\sigma_s}{2}\|x_s^{*}-\bar x_s\|^2 \right]\\
&\overset{\Cref{DescentIneq}}{\leq} M_s\|x_s^{*}-x_s\|^2-\frac{\sigma_s+2M_s}{2}\|x_s^{*}-x_s^{++}\|^2-\frac{M_s}{2}\|x_s^{++}-x_s\|^2.
\end{align*}
Thus, rearranging the above inequality,
we have that
$0.5M_s\|x_s^{++}-x_s\|^2\leq M_s\|x_s^{*}-x_s\|^2$ and hence $\|x_s^{++}-x_s\|\leq \sqrt{2}\|x_s^{*}-x_s\|$.
It then follows from this relation, \Cref{GM1},
\Cref{GM2} that
\begin{align*}
    \|G_{2M_s}(x_s)\|&\leq \|G_{\sigma_s+2M_s}(x_s)\|=(\sigma_s+2M_s)\|x_s-x^{+}(2M_s+\sigma_s;x_s)\|\\
&\leq (\sigma_s+2M_s)\|x_s-x_s^{++}\|+(\sigma_s+2M_s)\|x_s^{++}-x^{+}(2M_s+\sigma_s;x_s)\|\\
&\overset{\Cref{GM1}}{\leq} (\sigma_s+2M_s)\|x_s-x_s^{++}\|+\sigma_s\|x_s-\bar x_s\|\\
& \leq (2\sigma_s+4M_s)\|x_s^{*}-x_s\|+\sigma_s\|x_s-x_s^{*}\|+\sigma_s\|x_s^{*}-\bar x_s\|\\
&=(3\sigma_s+4M_s)\|x_s^{*}-x_s\|+\sigma_s\|x_s^{*}-\bar x_s\|
\end{align*}
which immediately implies the result holds.
\end{proof}

The following proposition provides key bounds on the distances of the iterates generated by AR-L. We omit the proof since the proof follows similarly to the proof of Propositions 2.1 and 3.2 in \cite{AL27Lan}.

\begin{proposition}
    For all $s\geq 1$, we have that the following relations must hold:
\begin{align}
    &\|x_{s-1}-x_s^{*}\|\leq \|x_{s-1}-x_{s-1}^{*}\|\label{DistBound}\\
    & \sigma_s\|\bar x_s-x_s^{*}\|\leq \sum_{i=1}^{s}(\sigma_{i-1}+\sigma_i)\|x_{i-1}^{*}-x_{i-1}\|\label{SigmaDistBound}.
\end{align}
\end{proposition}

The following proposition shows that the iterate $v_s$ computed in step \ref{ARL12} lies in the subdifferential of $\phi$ and also shows a key bound on its norm.

\begin{proposition}
For every $s\geq 1$, the following relations must hold:
\begin{align}
    &v_s\in \nabla \psi_s(x_{s}^{+})+\partial \psi_n(x_s^{+})\label{Vinclusion}\\
    &\|v_s\|\leq (6\sigma_s+8M_s)\|x_s^{*}-x_s\|+2\sigma_s\|x_s^{*}-\bar x_s\|\label{Vbound}.
\end{align}
where recall $v_{s}$ is computed as in step \ref{ARL12} of \Cref{alg:ARL}.
\end{proposition}
\begin{proof}
It follows from the definition of $x_s^{+}$ in \Cref{eq2}
and the optimality of the subproblem in the definition that $0\in \nabla \psi_s(x_s)+\partial \psi_n(x_s^{+})+2M_s(x_s^{+}-x_s)$. Hence, it follows that
$2M_s(x_s-x_s^{+})+\nabla \psi_s(x_s^{+})-\nabla \psi_s(x_s)\in \nabla \psi_s(x_s^{+})+\partial \psi_n(x_s^{+})$ which immediately implies the inclusion in \Cref{Vinclusion}
in view of the definition of $v_s$ in step \ref{ARL12} of \Cref{alg:ARL}.

Now observe that it follows from the definition of $x_s^{+}$ in \Cref{eq2} and the definition of $G_{\eta}(x)$ in \Cref{Gradient Mapping} that $G_{2M_s}(x_s)=2M_s(x_s-x_s^{+})$. It follows from this observation, the definition of $v_s$ in step \ref{ARL12} of \Cref{alg:ARL}, \Cref{eq4}, and \Cref{GradientMapBound} that
\begin{align*}
    \|v_s\|&=\|2M_s(x_s-x_s^{+})+\nabla \psi_s(x_s^{+})-\nabla \psi_s(x_s)\|\le \|2M_s(x_s-x_s^{+})\|+\|\nabla \psi_s(x_s^{+})-\nabla \psi_s(x_s)\|\\
    & \overset{\Cref{eq4}}{\leq}2M_s\|(x_s-x_s^{+})\|+2M_s\|x_s^{+}-x_s\|=2\|G_{2M_s}(x_s)\|\overset{\Cref{GradientMapBound}}{\leq} (6\sigma_s+8M_s)\|x_s^{*}-x_s\|+2\sigma_s\|x_s^{*}-\bar x_s\|
\end{align*}
which shows \Cref{Vbound}.
\end{proof}

We are now ready to prove the main complexity theorem of AR-L.

\begin{proof}[Proof of \Cref{AR-LComplex}]
Observe that it follows  from 
\Cref{SigmaDistBound} and \Cref{Vbound}
that for every $s\geq 1$, it holds that
\begin{align}
    \|v_s\|&\leq (6\sigma_s+8M_s)\|x_s^{*}-x_s\|+2\sigma_s\|x_s^{*}-\bar x_s\| \nonumber \\
    &\overset{\Cref{SigmaDistBound}}{\leq}
    (6\sigma_s+8M_s)\|x_s^{*}-x_s\|+2\sum_{i=1}^{s}(\sigma_{i-1}+\sigma_i)\|x_{i-1}^{*}-x_{i-1}\|\label{VsS1Bound}\\
    &=(6\sigma_s+8M_s)\|x_s^{*}-x_s\|+2\sigma_1\|x_0-x^{*}\|+2\sum_{i=2}^{s}(\sigma_{i-1}+\sigma_i)\|x_{i-1}^{*}-x_{i-1}\|\label{BoundVs}.
\end{align}
Now observe that 
since $\phi_s$ in \Cref{eq:subproblem_unconstrained_searchL} is $\sigma_s$-strongly
convex, it follows that
\begin{equation}\label{StrongConv}
    \|x-x_s^{*}\|\leq \sqrt{\frac{2}{\sigma_s}}\sqrt{\phi_s(x)-\phi_s(x_s^{*})} \quad \forall x\in \Dn.
\end{equation}
 It then follows from \Cref{StrongConv}, \Cref{assumptionA}, and the bound $L_s^{k}\leq c_{\mathcal A}\bar L$ that, after the $k$-th evaluation of $\nabla \psi_s$, subroutine $\mathcal A$ computes an approximate solution $x_s^{k}$ such that 
\begin{align}
    \|x_s^{k}-x_s^{*}\|&\overset{\Cref{StrongConv}}{\leq} \sqrt{\frac{2}{\sigma_s}}\sqrt{\phi_s(x_s^{k})-\phi_s(x_s^{*})} \nonumber \\
    &\overset{\Cref{assumptionA}}{\leq} \sqrt{\frac{2}{\sigma_s}}\sqrt{\frac{L_s^{k}\|x_{s-1}-x_s^{*}\|^2}{k^2}} \leq \frac{1}{k}\sqrt{\frac{2c_{\mathcal A}\bar L}{\sigma_s}}\|x_{s-1}-x_s^{*}\|\label{BoundXs}.
\end{align}
Moreover, it follows from the assumption on $\mathcal A$ in \Cref{ATermin} and the bound $L_s^{k}\leq c_{\mathcal A}\bar L$
that the number of gradient evaluations that $\mathcal A$ performs is bounded by
\begin{equation}\label{BoundNs}
    N_s:=\left\lceil 8 \sqrt{\frac{2L_s^k}{\sigma_s}}\right \rceil\leq \left \lceil 8\sqrt{\frac{2c_{\mathcal A}\bar L}{\sigma_s}} \right\rceil.
\end{equation}
It then follows from \Cref{DistBound}, \Cref{BoundXs}, and \Cref{BoundNs}
that
\begin{align}
    \|x_s-x_s^{*}\|&\overset{\Cref{BoundXs}}{\leq} \frac{1}{N_s}\sqrt{\frac{2 c_{\mathcal A}\bar L}{\sigma_s}}\|x_{s-1}-x_{s}^{*}\|\nonumber\\
   & \overset{\Cref{BoundNs}}{\leq}\frac{1}{8}\|x_{s-1}-x_{s}^{*}\|\overset{\Cref{DistBound}}{\leq} \frac{1}{8}\|x_{s-1}-x_{s-1}^{*}\|\label{ContractionBound}.
\end{align}
It then follows from the bound on $v_s$ in
\Cref{BoundVs}, \Cref{ContractionBound}, the fact that $8^{-s}\leq (0.5)4^{-s}$ for $s\geq 2$, and the fact that $\sigma_{s}=4\sigma_{s-1}$ that the following relations hold for every $s\geq 2$:
\begin{align}
\|v_s\|&\overset{\Cref{BoundVs}}{\leq} (6\sigma_s+8M_s)\|x_s^{*}-x_s\|+2\sigma_1\|x_0-x^{*}\|+2\sum_{i=2}^{s}(\sigma_{i-1}+\sigma_i)\|x_{i-1}^{*}-x_{i-1}\|\nonumber\\
&\overset{\Cref{ContractionBound}}{\leq} (6\sigma_s+8M_s)8^{-s}\|x^{*}-x_0\|+2\sigma_{1}\|x_0-x^{*}\|+10\sigma_{1}\sum_{i=2}^{s} 4^{i-2}8^{-i+1}\|x^{*}-x_0\|\nonumber\\
&\leq (3\sigma_s+4M_s)4^{-s}\|x^{*}-x_0\|+2\sigma_{1}\|x_0-x^{*}\|+\frac{5}{2}\sigma_{1}\|x_0-x^{*}\|\nonumber\\
&=(3\sigma_s+4M_s)4^{-s}\|x^{*}-x_0\|+\frac{9}{2}\sigma_{1}\|x_0-x^{*}\|\label{VRefinedBound}.
\end{align}
We now estimate $\|\hat v\|$ when the termination condition
$\text{stopcrit}\leq 0$ is satisfied, where
$\text{stopcrit}=M_s-\sigma_s$.
If the criterion is satisfied with $s=1$, then it follows from step~\ref{ARLOut} of AR-L that it terminates with $(\hat x,\hat v,M)=(\xp{1},v_1,M_1)$.
Moreover, it follows from the fact that $\sigma_0=0$, $\sigma_1\geq M_1$, 
\Cref{VsS1Bound} and \Cref{ContractionBound} with $s=1$, and  the fact that $8^{-1}\leq (0.5)4^{-1}$
that
\begin{align}
\|\hat v\|=\|v_1\|&\overset{}{\leq} 
    (6\sigma_1+8M_1)\|x_1^{*}-x_1\|+2\sigma_1\|x_0-x^{*}\|\nonumber\\
    &\overset{\eqref{ContractionBound}}{\leq} (3\sigma_1+4\sigma_1)*\frac{1}{4}\|x_0-x^{*}\|+2\sigma_1\|x_0-x^{*}\|\nonumber\\
    &\leq 7\sigma_1\|x_0-x^{*}\|\leq 7\sigma_1D\label{SubdiffBound}.
\end{align}
Otherwise, let $S>1$ be the smallest index $s$
for which the termination criterion 
$M_s-\sigma_s\leq 0$ is satisfied.
It follows from
\Cref{VRefinedBound}, the fact that the diameter of $\Dn$ is $D$, the identity $\sigma_s=4^{s-1}\sigma_1$, and the termination condition $\sigma_{S}\geq M_S$ that 
the output
$\hat v=v_{S}$ satisfies
\begin{align}
    \|\hat v\|=\|v_{S}
    &\|\leq (3\sigma_{S}+4M_S)4^{-S}\|x^{*}-x_0\|+\frac{9}{2}\sigma_1\|x_0-x^{*}\| \nonumber\\
    & \leq (7\sigma_{S})4^{-S}\|x^{*}-x_0\|+\frac{9}{2}\sigma_1\|x_0-x^{*}\| \nonumber \\
    &=\frac{7}{4}\sigma_1\|x^{*}-x_0\|+\frac{9}{2}\sigma_1\|x^{*}-x_0\|\leq 7\sigma_1\|x^{*}-x_0\|\leq 7\sigma_1D\label{subdiffBoundFinal}.
\end{align}
Hence, it follows from \Cref{SubdiffBound,subdiffBoundFinal} that the output $\hat v$
always satisfies the inequality in \Cref{Output1}. Moreover, since the inclusion
\Cref{Vinclusion} holds for every $s\geq 1$, it follows from the way the output
$(\hat x,\hat v,\hat M)$ is defined in step~\ref{ARLOut} that the pair $(\hat x,\hat v)$
satisfies the inclusion in \Cref{Output1}. Finally, because \Cref{eq5} holds
at every iteration $s\geq 1$, and $\hat M=M_s$ for some iteration index $s$,
we conclude that $\hat M$ satisfies \Cref{Output3}.

To establish the complexity bound in \Cref{CompResultARL}, assume that
$M_0\leq 2\bar L$ and let $S$ be the smallest index for which termination criterion $\sigma_s\geq M_s$ is satisfied. First, observe that \Cref{eq5} implies that, for every
$s=1,\ldots,S$, the backtracking procedure outputs $M_s$ satisfying $M_s\leq \max\left\{M_0,2\bar L \right\}$. We next derive an upper bound on $S$. Using the bound above, the update
$\sigma_s=4\sigma_{s-1}$ for $s>1$, the assumption $M_0\leq 2\bar L$, and the
fact that $S>1$ is the first index such that $\sigma_S\geq M_S$, we obtain
\[M_S\leq 4^{S-1}\sigma_1=\sigma_{S}=4\sigma_{S-1}<4M_{S-1}\overset{}{\leq}\max \left\{4M_0, 8\bar L \right\}\leq 8 \bar L.\]
It is then immediate from the above relations that
$S\leq 2+\log^{+}_{4}(2\bar L/\sigma_1)$.

We now bound the total number of gradient evaluations performed during AR-L to find such a $\hat v$.
Recall that during each outer iteration of AR-L, algorithm $\mathcal A$ performs $N_s$ gradient evaluations where $N_s$
is bounded as in \Cref{BoundNs}. Furthermore, as shown above
AR-L performs at most $S$ outer iterations, where $S\leq 2+\log_{4}(2\bar L/\sigma_1)$.
Hence, the total number of gradient evaluations that $\mathcal A$ performs throughout the AR-L method can be bounded as
\begin{align*}
    \sum_{s=1}^{S}N_s&\leq \sum_{s=1}^{S}1+8\sqrt{\frac{2 c_{\mathcal A}\bar L}{\sigma_s}}\\
    &\leq S+8\sqrt{\frac{2c_{\mathcal A}\bar L}{\sigma_1}}\sum_{s=1}^{\infty}2^{-(s-1)}\leq S+16\sqrt{\frac{2c_{\mathcal A}\bar L}{\sigma_1}}.
\end{align*}
Now observe that the call made to the Backtracking function in step~\ref{ARL11} of the $s$-th iteration of AR-L performs at most
$2+\log_2(M_s/M_{s-1})$ gradient evaluations.
Thus, it follows from summing and 
using \Cref{eq5} with $s=S$ that
the total number of extra gradient evaluations from the 
Backtracking function is bounded by 
$2S+\log_2(2\bar L/M_0)$.
Thus, summing everything together and using the fact that $3\log_{4} t<2\sqrt{t}$ for all $t>0$, the total number of gradient evaluations that the AR-L method performs is bounded by
\begin{align*}
    &3S+16\sqrt{\frac{2c_{\mathcal A}\bar L}{\sigma_1}}+\log_{2}\left(\frac{2\bar L}{M_0}\right)\leq 3\left[2+\log^{+}_4\left(\frac{2\bar L}{\sigma_1} \right)\right]+16\sqrt{\frac{2c_{\mathcal A}\bar L}{\sigma_1}}+\log_{2}\left(\frac{2\bar L}{M_0}\right)\nonumber\\
&=6+3\log^{+}_4\left(\frac{2\bar L}{\sigma_1} \right)
+16\sqrt{2c_{\mathcal A}}\sqrt{\frac{\bar L}{\sigma_1}}+\log_{2}\left(\frac{2\bar L}{M_0}\right)\\
&<6+2\sqrt{\frac{2\bar L}{\sigma_1}}
+16\sqrt{2c_{\mathcal A}}\sqrt{\frac{\bar L}{\sigma_1}}+\log_{2}\left(\frac{2\bar L}{M_0}\right)
= 6
+C_{1}\sqrt{\frac{\bar L}{\sigma_1}}+\log_{2}\left(\frac{2\bar L}{M_0}\right)
\end{align*}
where $C_1:=2\sqrt{2}+16\sqrt{2c_{\mathcal A}}$.
This immediately proves the gradient evaluation
complexity bound \Cref{CompResultARL} of the AR-L method.

The last conclusion of \Cref{AR-LComplex} follows immediately from
the first conclusion of \Cref{AR-LComplex} and the assumption that $\sigma_1=\epsilon/(7\overline D)$ for some scalar $\overline D>0$.
\end{proof}

\section{Proof of \texorpdfstring{\Cref{prop:nest_complex1}}{Lg}}\label{SC-FISTA Appendix Proof}
\phantomsection 
This appendix section is dedicated to proving \Cref{prop:nest_complex1}. It is divided into two subsections. The first subsection is dedicated to proving
\Cref{prop:nest_complex1}(a)-(b) while the second subsection is dedicated to proving \Cref{prop:nest_complex1}(c).
\phantomsection

\subsection{Proof of \texorpdfstring{\Cref{prop:nest_complex1}(a)-(b)}{Lg}}

The following lemma presents key properties of the iterates generated
during the $\ell$-th cycle of R-FISTA. Its proof is not given as it closely resembles the proofs of Lemmas A.3 and A.4 in \cite{AL38Sujanani}.

\begin{lemma}\label{lem:gamma-sfista0}
Let $\theta$, $\zeta_{\ell}$, and $Q_{\ell}$ be as in \cref{D0 def}. For every inner
iteration index $j \geq 1$ generated during the $\ell$-th cycle of R-FISTA, the following statements hold:
\begin{itemize}
\item[(a)] $\{L_j\}$ is nondecreasing;
\item[(b)] the following relations hold
\begin{align}
&\tau_{j-1} = 1+\frac{\tilde\mu A_{j-1}}{2}, \quad \frac{\tau_{j-1}  A_{j}}{ a_{j-1}^2}=L_{j},\label{tauproperty}\\[0.5em]
&\underbar M_{\ell}\leq L_{j-1}\leq \max\{\underbar M_{\ell},\theta \bar L\}\label{upper bound},\\[0.8em]
& v_{j} \in \nabla \psi_s(y_{j}) + \partial \psi_n(y_{j}), \quad \|v_{j}\|\leq \zeta_{\ell} \|y_{j}-\tilde x_{j-1}\|;\label{ubound-1}
\end{align}
\item[(c)] it holds that 
\begin{equation}\label{sumAkbound}
    A_jL_j \geq \max \left\{\frac{j^2}{4} ,  \left(1+Q_{\ell}^{-1}\right)^{2(j-1)}
 \right\}.
\end{equation}
\end{itemize}
\end{lemma}

The proof of \Cref{prop:nest_complex1}(a)-(b) is now presented. The
proof of \Cref{prop:nest_complex1}(a)-(b) has a similar flavor to the proof of Proposition 2.1(a)-(b) in \cite{AL39Sujanani}, but we include it here for completeness.
\begin{proof}[Proof of \Cref{prop:nest_complex1}(a)-(b)]
(a) Consider the $\ell$-th cycle of R-FISTA and let $m$ denote the first term in the sum in \cref{eq:eq1}. Using 
this definition and the
inequality $\log (1+ \alpha) \ge \alpha/(1+\alpha)$ for
any $\alpha>-1$, it is easy to verify that
\begin{equation}\label{Q bound-2}
\left(1+Q_{\ell}^{-1}\right)^{2(m-1)} \ge 
\frac{D^2\zeta_{\ell}^2}{\chi \epsilon^2}.
\end{equation}
We claim that the $\ell$-th cycle of R-FISTA  either terminates successfully with stopcrit $\leq \epsilon$ or stops in its step~\ref{RestartCheckStep} because the inequality \Cref{restart condition} did not hold (in one of its inner iterations) in at most $m$ inner ACG iterations.
Indeed, it suffices to show that
if the $\ell$-th cycle of R-FISTA has not stopped in step~\ref{RestartCheckStep}
up to (and including) the
$m$-th inner ACG iteration, then
the main while loop of R-FISTA must successfully stop with stopcrit $\leq \epsilon$
at the $m$-th iteration.
So, assume that the $\ell$-th cycle of
R-FISTA has not stopped in its step~\ref{RestartCheckStep} up to (and including) the
$m$-th inner ACG iteration. 
Hence, it follows that
\cref{restart condition} holds with $j=m$.

This observation together with
the inequality \cref{ubound-1} with $j=m$, relation \cref{Q bound-2}, and inequality \cref{sumAkbound} with $j=m$,
then imply that
\begin{align}\label{bound diff-1-2}
D^2&\geq\|y_{m}-x_0\|^{2} \overset{\cref{restart condition}}{\geq} \chi A_{m}L_{m} \|y_{m}-\tilde x_{m-1}\|^2
\overset{\cref{ubound-1}}{\ge}
\frac{\chi}{\zeta_{\ell}^2}A_mL_m \|v_{m}\|^2\\
&\overset{\cref{sumAkbound}}{\geq}\frac{\chi}{\zeta_{\ell}^2}
\left(1+Q_{\ell}^{-1} \right)^{2(m-1)}
\|v_m\|^2 \overset{\cref{Q bound-2}}{\geq} \frac{D^2}{\epsilon^2}\|v_m\|^2
\end{align}
and hence that $\|v_m\|\leq \epsilon$.
In view of the criterion of the main while loop of
R-FISTA, the $\ell$-th cycle of R-FISTA
must thus
successfully stop at the
end of its $m$-th inner ACG iteration.
Hence,
the above claim holds.
Moreover, in view of \cref{upper bound},
it follows that
the second term in 
\cref{eq:eq1} is a bound
on the total number of times $L_j$ is multiplied by $2$ and step~\ref{InnerWhileFISTA} is repeated. Since exactly one gradient evaluation occurs every time step~\ref{InnerWhileFISTA} is executed, the result of part (a) must hold.

(b) 
Suppose that the main while loop of the $\ell$-th cycle of R-FISTA terminates successfully with stopcrit $\leq \epsilon$.
Since
stopcrit=$\|v_j\|$
whenever no restart occurs 
and the inclusion in 
\cref{ubound-1}
holds for every generated iterate,
the output pair 
$(y,v)$
satisfies
\Cref{OutputRFISTA}.
\end{proof}

\subsection{Proof of \texorpdfstring{\Cref{prop:nest_complex1}(c)}{Lg}}
For the remainder of this subsection, consider the $\ell$-th cycle of
R-FISTA. We assume that $j\geq 1$ is an inner ACG iteration index generated during
the cycle and that $\tilde \mu=\mu_{\ell-1}$ is in the interval $(0,\bar \mu]$.

A useful fact that is used in the proof of following result is that if a function $\Psi:\E \to\R\cup\{+\infty\}$ is $\nu$-convex with modulus $\nu>0$, then it has a unique global minimum $x^*$ and
\begin{equation}\label{ineq:nu-convex}
\Psi(x^*) +\frac{\nu}{2}\|\cdot - x^*\|^2\leq \Psi(\cdot). 
\end{equation}

The first lemma below characterizes a key property of the iterate $y_j$.
\begin{lemma}\label{Key Relation}
For every iteration index $j \geq 1$ generated during the $\ell$-th cycle of R-FISTA, it holds that
\begin{equation}
		y_{j}:=\underset{u\in \Dn}\argmin\left\lbrace q_{j-1} (u;\tx_{j-1},L_{j}) 
		:= \ell_{\psi_s}(u;\tilde x_{j-1})+\psi_n(u) + \frac{L_{j}}{2}\|u-\tx_{j-1}\|^2\right\rbrace.
		\label{eq:ynext-sfistaRelation}
		\end{equation}
\end{lemma}
\begin{proof}
It follows from the update rule for $y_j$ in \cref{eq:ynext-sfista1}, the definition of the proximal operator, the definition of the linearization of $\psi_s$ in \cref{eq:defell}, and algebraic manipulation that
\begin{align*}
    y_j&\overset{\cref{eq:ynext-sfista1}}{=}\argmin_{u\in \Dn}\left\{\frac{1}{L_j} \psi_n(u)+\frac{1}{2}\left\|u-\left(\tilde x_{j-1}-\frac{1}{L_j}\nabla \psi_s(\tilde x_{j-1})\right)\right\|^2\right\}\\
    &\overset{\cref{eq:defell}}{=}\argmin_{u\in \Dn} \left\{ \ell_{\psi_s}(u;\tilde x_{j-1})+\psi_n(u) + \frac{L_{j}}{2}\|u-\tx_{j-1}\|^2\right\}
\end{align*}
which immediately implies that \cref{eq:ynext-sfistaRelation} holds.
\end{proof}

It can easily be seen that \Cref{Key Relation} and the fact that $\tilde \mu=\mu_{\ell-1}$ is in the interval $(0,\bar \mu]$ imply that Lemmas A.2-A.6 in \cite{AL39Sujanani} hold. The next lemma thus restates Lemma A.6 of \cite{AL39Sujanani} which lower bounds the difference in our potential function $\sigma_{j}(x)$.

\begin{lemma} \label{lm:hysub-4-fista}
For every $j \ge 1$ and $x \in \Dn$, we have
\begin{equation}\label{recursion}
\sigma_{j-1}(x) - \sigma_{j} (x) \ge   \frac{\chi A_{j}L_{j}}{2} \|y_{j} - \tx_{j-1} \|^2
\end{equation}
where
\[
\sigma_j(x) := A_j [\psi(y_j) - \psi(x) ] + \frac{\tau_j}{2} \|x-x_j\|^2.
\]
\end{lemma}

Next, we state a basic result that is useful in deriving complexity bounds for R-FISTA.
\begin{lemma} \label{lm:hysub-5-ac}
For every $j \ge 2$ and $x \in \Dn$, it holds that
\begin{align}\label{key convergence inequality 2}
 A_{j-1} [\psi(y_{j-1}) - \psi(x) ] + \frac{\tau_{j-1}}2 \|x-x_{j-1}\|^2 \le
 \frac12 \|x-x_0\|^2 - \frac{\chi}{2} \sum_{i=1}^{j-1} A_{i}L_{i} \| y_{i} - \tx_{i-1} \|^2.
\end{align}
\end{lemma}
\begin{proof}
    The proof follows immediately by summing relation \cref{recursion} from $1$ to $j-1$ and noting that $A_0=0$ and $\tau_0=1$.
\end{proof}

We are now ready to prove \Cref{prop:nest_complex1}(c).
\begin{proof}[Proof of \Cref{prop:nest_complex1}(c)]
Consider the $\ell$-th cycle of R-FISTA and assume that it is performed
with $\tilde \mu=\mu_{\ell-1}$ in the interval $ (0,\bar \mu]$. Using relation \cref{key convergence inequality 2} with $x=y_{j-1}$, it follows that
\[
\|y_{j-1}-x_0\|^2 \overset{\cref{key convergence inequality 2}}{\geq}\chi \sum_{i=1}^{j-1}A_{i}L_{i} \|y_{i}-\tilde x_{i-1}\|^2 \geq \chi A_{j-1}L_{j-1}\|y_{j-1}-\tilde x_{j-2}\|^2.
\]
It then follows from the above relation that for any inner iteration index
$j\geq 1$ generated during the $\ell$-th cycle, it holds that
\begin{equation}\label{bound diff-3}
\|y_{j}-x_0\|^2 \geq \chi A_{j}L_{j}\|y_{j}-\tilde x_{j-1}\|^2.
\end{equation}
Hence, relation \cref{bound diff-3} implies that the inequality
\cref{restart condition} checked in step~\ref{RestartCheckStep} of R-FISTA always holds for
every inner iteration index $j\geq 1$ generated during the $\ell$-th cycle.
Hence, the $\ell$-th cycle of R-FISTA never terminates in its step~\ref{RestartCheckStep} and a restart is not executed.
This observation together with
\Cref{prop:nest_complex1}(a)-(b) then immediately imply that
the main while loop of the $\ell$-th cycle must terminate successfully with stopcrit$\leq \epsilon$ and a pair $(y,v)$ satisfying \Cref{OutputRFISTA} in at most \cref{eq:eq1} ACG iterations/gradient evaluations.
\end{proof}

\section{Proof of \texorpdfstring{\Cref{Total Complexity}}{Lg}}
\phantomsection 
This subsection is dedicated to proving \Cref{Total Complexity}. As mentioned in the statement of \Cref{Total Complexity} we assume that $(\mu_0,\bar M_0)$ satisfy $\mu_0\geq \bar \mu$ and $\bar M_0\leq 2\bar L$.

The proposition below establishes an upper bound on the number of cycles that R-FISTA performs and an upper bound on the number of inner ACG iterations that each cycle of R-FISTA performs. Its proof is similar to that of Proposition 2.2 in \cite{AL39Sujanani}, but we include it here for completeness.

\begin{prop}\label{Cycle Complexity}
The following statements about R-FISTA hold:
\begin{itemize}
    \item[(a)] R-FISTA performs at most $\lceil \log^{++}_2\left(2\mu_0/\bar \mu\right) \rceil$ cycles before terminating with a 
    pair $(y,v)$ that is an $\epsilon$-optimal solution of \cref{MainProblem};
    \item[(b)] each cycle of R-FISTA performs at most 
    \begin{equation}\label{ACGCompFISTA}
    \mathcal O_1\left(\sqrt{\frac{\bar L}{\bar \mu}}\log^{++} \left(\frac{D\bar L}{\epsilon}\right) +\log_2^{+}(\bar L/\bar M_0)\right)
    \end{equation}
ACG iterations/gradient evaluations.
\end{itemize}
\end{prop}

\begin{proof}
(a) It follows from \Cref{prop:nest_complex1}(b) that if the main while loop of a cycle of R-FISTA terminates successfully with stopcrit $\leq \epsilon$, then the pair $(y,v)$ that R-FISTA outputs satisfies \Cref{OutputRFISTA}, i.e., the pair is an $\epsilon$-optimal solution of \Cref{MainProblem}. Moreover,  \Cref{prop:nest_complex1}(c) implies that if the $\ell$-th cycle of R-FISTA is performed with $\mu_{\ell-1}\in (0,\bar \mu]$, then the main 
while loop of the cycle always stops successfully with stopcrit$\leq \epsilon$ and hence R-FISTA terminates. The conclusion of part(a) then follows from these two observations and the fact that $\mu_{\ell}$ is updated according to the rule $\mu_{\ell}=\mu_{\ell-1}/2$ if a restart occurs and a new cycle begins.

(b) 
First observe that part(a) and the 
fact that $\mu_0\geq \bar \mu$ imply
that $\mu_{\ell}\geq \min\{\mu_0,\bar \mu/2\}=\bar \mu/2$
for all $\ell\geq 1$. Moreover, it follows from the way $\underbar M_{\ell}$ is set in the Initialization phase of 
R-FISTA that $\bar M_0\leq \underbar M_{\ell}$. This conclusion, the fact that $\bar M_0\leq 2\bar L\leq \theta \bar L$, and a similar argument as the proof of Lemma 2.5
in \cite{AL39Sujanani} then imply that
\begin{equation}\label{Prelim Bound Underbar Ml}
\bar M_0\leq \underbar M_{\ell} \leq \max\left\{\bar M_0, \theta\bar L \right\}\leq \theta\bar L.
\end{equation}
Also observe that it follows from \Cref{prop:nest_complex1}(a) 
and the definitions of $\theta$, $\zeta_{\ell}$, and $Q_{\ell}$ in \Cref{D0 def}
that
each cycle of R-FISTA
performs at most
\begin{equation}\label{prelim complexity bound}
\mathcal O_1\left(\sqrt{\frac{\max\left\{\underbar
M_{\ell},\bar L\right\}}{\mu_{\ell-1}}}\log^{++} \left(\frac{D^2(\bar L^2+\underbar
M_{\ell}^2)}{\epsilon^2}\right) +\log_2^{+}(\bar L/\underbar
M_{\ell})\right)
\end{equation}
ACG iterations/gradient evaluations.
The result of part (b) then follows from the lower bound on $\mu_{\ell}$, the fact 
that $\mathcal O(\log(A))=\mathcal O(\log(A^{2}))$ for a scalar $A>0$,
and bounds \Cref{Prelim Bound Underbar Ml} and \Cref{prelim complexity bound}.
\end{proof}

We are now ready to prove \Cref{Total Complexity}. 
\begin{proof}[Proof of \Cref{Total Complexity}]
\Cref{Cycle Complexity}(a) implies that R-FISTA 
performs at most $\lceil \log^{++}_2\left(2\mu_0/\bar \mu\right) \rceil$ cycles before 
finding a pair $(y,v)$
that 
is an $\epsilon$-optimal solution
of \cref{MainProblem}. The result then follows from this observation 
and the fact that the total number of ACG iterations/gradient evaluations that R-FISTA performs can be upper bounded by the product of the number of 
cycles $\lceil \log^{++}_2\left(2\mu_0/\bar \mu\right) \rceil$ and the bound in \Cref{ACGCompFISTA}.
\end{proof}

\renewcommand{\baselinestretch}{.5}
\tiny
\scriptsize

\printindex
\addcontentsline{toc}{section}{Index}
\label{ind:index}

\typeout{}
\bibliographystyle{siam}
\bibliography{Proxacc_ref}
 \addcontentsline{toc}{section}{Bibliography}

\end{document}

%% file: inputusepackages.tex
\usepackage{xcolor}
\definecolor{OliveGreen}{rgb}{0,0.6,0}
\definecolor{tempblue}{RGB}{36, 56, 231 }
\usepackage{cite}
\usepackage{amssymb}
\usepackage{amsthm}
\usepackage{enumitem}    

\usepackage{mathrsfs}
\usepackage{mathtools}
\usepackage{lineno}
\usepackage{float}
\usepackage{fancyhdr}
\usepackage[us,12hr]{datetime} 
\usepackage{tabularx}
\usepackage{longtable}
\usepackage{latexsym}
\usepackage{fullpage}       
\usepackage{float}
\usepackage{verbatim}
\usepackage{multirow}

\definecolor{tablegray}{RGB}{215, 219, 221 }

\usepackage{makeidx}
\makeindex

\usepackage{algorithm}   
\usepackage{algorithmic}   

\makeindex

\usepackage[pagebackref=true]{hyperref}
\hypersetup{
	plainpages=false,       
	unicode=false,          
	pdftoolbar=true,        
	pdfmenubar=true,        
	pdffitwindow=false,     
	pdfstartview={FitH},    
	pdftitle={
Unassigned Euclidean Distance Geometry, uDGP
	           },    
	pdfnewwindow=true,      
	colorlinks=true,        
	linkcolor=blue,         
	citecolor=green,        
	filecolor=magenta,      
	urlcolor=cyan           
}
\usepackage[noabbrev]{cleveref}
\RemoveFromHook{label}[firstaid/cleveref]
\usepackage{refcount} 

\newboolean{PrintVersion}
\setboolean{PrintVersion}{false} 

\hypersetup{
    plainpages=false,       
    pdfpagelabels=true,     
    bookmarks=true,         
    unicode=false,          
    pdftoolbar=true,        
    pdfmenubar=true,        
    pdffitwindow=false,     
    pdfstartview={FitH},    
    pdftitle={Strict Feasibility and Sensitivity Analysis},    
    pdfnewwindow=true,      
    colorlinks=true,        
    linkcolor=blue,         
    citecolor=green,        
    filecolor=magenta,      
    urlcolor=cyan           
}
\ifthenelse{\boolean{PrintVersion}}{   
\hypersetup{	
    citecolor=black,%
    filecolor=black,%
    linkcolor=black,%
    urlcolor=black}
}{} 

\renewcommand{\baselinestretch}{1} 

\let\origdoublepage\cleardoublepage
\newcommand{\clearemptydoublepage}{%
\clearpage{\pagestyle{empty}\origdoublepage}}
\let\cleardoublepage\clearemptydoublepage

\restylefloat{table}

\crefname{subsection}{subsection}{subsections}
\Crefname{subsection}{Subsection}{Subsections}

\makeindex

\let\oldindex\index
\renewcommand*{\index}[1]{\oldindex{#1}\ignorespaces}

\newtheorem{theorem}{Theorem}[section]
\newtheorem{lemma}[theorem]{Lemma}

\newtheorem{definition}[theorem]{Definition}

\newtheorem{assump}[theorem]{Assumption}
\newtheorem{prop}[theorem]{Proposition}
\newtheorem{proposition}[theorem]{Proposition}

\newtheorem{remark}[theorem]{Remark}
\newtheorem{rem}[theorem]{Remark}

\crefname{thm}{Theorem}{Theorems}
\Crefname{thm}{Theorem}{Theorems}
\crefname{problem}{Problem}{Theorems}
\Crefname{problem}{Problem}{Theorems}
\crefname{conjecture}{Conjecture}{Theorems}
\Crefname{conjecture}{Conjecture}{Theorems}
\crefname{proposition}{Proposition}{Propositions}
\Crefname{proposition}{Proposition}{Propositions}
\crefname{prop}{Proposition}{Propositions}
\Crefname{prop}{Proposition}{Propositions}
\crefname{cor}{Corollary}{Corollaries}
\crefname{lemma}{Lemma}{Lemmas}
\Crefname{lemma}{Lemma}{Lemmas}
\crefname{lem}{Lemma}{Lemmas}
\Crefname{lem}{Lemma}{Lemmas}
\theoremstyle{definition}
\crefname{assump}{Assumption}{assumptions}
\Crefname{assump}{Assumption}{Assumptions}
\crefname{problem}{Problem}{problems}
\Crefname{problem}{Problem}{Problems}
\crefname{definition}{Definition}{definitions}
\Crefname{definition}{Definition}{Definitions}
\crefname{defn}{Definition}{definitions}
\Crefname{defn}{Definition}{Definitions}
\crefname{remark}{Remark}{Remarks}
\Crefname{remark}{Remark}{Remarks}
\crefname{rem}{Remark}{Remarks}
\Crefname{rem}{Remark}{Remarks}
\crefname{rmk}{Remark}{Remarks}
\Crefname{rmk}{Remark}{Remarks}
\crefname{example}{Example}{Examples}
\Crefname{example}{Example}{Examples}
\crefname{align}{}{}
\Crefname{align}{}{}
\crefname{equation}{}{}
\Crefname{equation}{}{}

\DeclareMathOperator{\argmin}{{argmin}}

\DeclareMathOperator{\dom}{{dom}}

\DeclareMathOperator{\dist}{{dist}}

\def\R{{\mathbb{R}}}

\newcommand{\pLC}{p^*_{\scriptscriptstyle LC}}   
\newcommand{\inner}[2]{\langle #1,#2\rangle}
\DeclareMathOperator{\cConv}{\overline{Conv}}
\global\long\def\tx{\tilde{x}}%
\newcommand{\Resk}{\rm{Res}_k}
\newcommand{\Resp}{\rm{Res}_p}
\newcommand{\Resd}{\rm{Res}_d}
\newcommand{\xb}[1]{\bar{x}_{#1}}
\newcommand{\xp}[1]{x_{#1}^{+}}
\newcommand{\xpp}[1]{x_{#1}^{++}}

\newcommand{\Dn}{\dom \psi_n}

\def\R{\mathbb{R}}
\def\E{\mathbb{E}}

\def\Rm{\mathbb{R}^m}

\newcommand{\cA}{{\mathcal A}}

\newcommand{\cL}{{\mathcal L}}

\newcommand{\cN}{{\mathcal N}}

\newcommand{\bbm}{\begin{bmatrix}}
\newcommand{\ebm}{\end{bmatrix}}
\newcommand{\bem}{\begin{pmatrix}}
\newcommand{\eem}{\end{pmatrix}}
\newcommand{\nc}{\newcommand}

\nc{\what}{\widehat}
\nc{\horzbar}{\rotatebox[origin=c]{90}{$\vert$}}
\def\<{\langle}
\def\>{\rangle}

\newcommand{\bd}{\operatorname{bd}}

\newcommand{\textdef}[1]{\index{#1}\textit{#1}}

\numberwithin{algorithm}{section}
\numberwithin{equation}{section}
\numberwithin{figure}{section}
\numberwithin{table}{section}

\DeclareMathOperator{\inte}{int}


%% file: table_QuadraticLC.tex
\begin{tabular}{|cc|ccc|ccc|ccc|} \hline
\multicolumn{2}{|c||}{Dimensions} & \multicolumn{3}{c|}{Time (s)} & \multicolumn{3}{c|}{Primal Feasibility Residual} & \multicolumn{3}{c|}{Dual Stationarity Residual} \\ \cline{1-2}\cline{3-5}\cline{6-8}\cline{9-11}
$m$ & $n$ & O-IAL & APF-IAL & ProxALM \cite{LuM} & O-IAL & APF-IAL & ProxALM \cite{LuM} & O-IAL & APF-IAL & ProxALM \cite{LuM} \\ \hline
     10 &      50 &    0.09 &    \textcolor{purple}{0.07} &    0.09 & 9.1959e-06 & 7.0414e-06 & 1.1720e-06 & 1.3176e-10 & 7.6083e-06 & 6.9513e-06 \\ \hline
     50 &     125 &    0.18 &    \textcolor{purple}{0.03} &    0.14 & 9.9027e-06 & 5.0281e-06 & 6.1081e-09 & 1.7364e-10 & 9.1839e-06 & 7.8900e-06 \\ \hline
    100 &     200 &    0.37 &    \textcolor{purple}{0.03} &    0.52 & 9.0125e-06 & 7.6086e-07 & 2.5729e-09 & 1.5800e-10 & 8.5710e-06 & 5.4161e-06 \\ \hline
    150 &     275 &    0.33 &    \textcolor{purple}{0.04} &    0.85 & 8.2824e-06 & 6.1373e-06 & 6.0852e-10 & 1.6413e-10 & 9.9256e-06 & 9.7421e-06 \\ \hline
    200 &     350 &    0.39 &    \textcolor{purple}{0.03} &    1.33 & 7.6745e-06 & 2.2687e-06 & 4.2170e-10 & 1.6393e-10 & 9.8571e-06 & 6.9411e-06 \\ \hline
    250 &     425 &    0.23 &    \textcolor{purple}{0.03} &    1.99 & 9.1271e-06 & 4.6282e-06 & 1.4520e-10 & 1.4806e-10 & 3.3230e-06 & 9.4211e-06 \\ \hline
    300 &     500 &    0.29 &    \textcolor{purple}{0.02} &    2.79 & 8.7465e-06 & 7.7660e-06 & 1.4695e-10 & 1.5001e-10 & 9.4454e-06 & 9.3530e-06 \\ \hline
    350 &     575 &    0.38 &    \textcolor{purple}{0.07} &    4.13 & 9.0311e-06 & 1.0688e-06 & 3.6827e-11 & 1.3399e-10 & 6.7502e-06 & 4.6030e-06 \\ \hline
    400 &     650 &    0.35 &    \textcolor{purple}{0.05} &    4.66 & 8.4926e-06 & 5.5555e-06 & 1.6717e-10 & 1.3553e-10 & 8.8359e-06 & 9.8516e-06 \\ \hline
    450 &     725 &    0.47 &    \textcolor{purple}{0.09} &    6.35 & 9.3200e-06 & 4.0478e-06 & 1.0157e-10 & 1.1782e-10 & 6.0721e-06 & 8.5942e-06 \\ \hline
    500 &     800 &    0.96 &    \textcolor{purple}{0.19} &    7.31 & 7.4388e-06 & 3.8783e-06 & 7.9476e-10 & 1.7589e-10 & 9.9275e-06 & 9.5685e-06 \\ \hline
    550 &     875 &    1.43 &    \textcolor{purple}{0.21} &   11.42 & 3.5687e-06 & 1.3724e-06 & 4.4192e-09 & 1.5287e-10 & 9.6858e-06 & 9.8851e-06 \\ \hline
    600 &     950 &    1.55 &    \textcolor{purple}{0.24} &   14.07 & 1.2666e-06 & 2.6860e-06 & 4.4561e-09 & 1.2568e-10 & 9.9306e-06 & 9.8263e-06 \\ \hline
    650 &    1025 &    1.93 &    \textcolor{purple}{0.31} &   19.59 & 1.5630e-06 & 6.8438e-06 & 2.6687e-09 & 1.3531e-10 & 9.7562e-06 & 9.8604e-06 \\ \hline
    700 &    1100 &    3.08 &    \textcolor{purple}{0.50} &   26.72 & 2.3606e-06 & 6.0140e-06 & 2.7868e-09 & 1.6383e-10 & 9.9026e-06 & 9.2066e-06 \\ \hline
\end{tabular}

%% file: table_LogReg.tex
\begin{tabular}{|cc|ccc|ccc|ccc|} \hline
\multicolumn{2}{|c||}{Dimensions} & \multicolumn{3}{c|}{Time (s)} & \multicolumn{3}{c|}{Primal Feasibility Residual} & \multicolumn{3}{c|}{Dual Stationarity Residual} \\ \cline{1-2}\cline{3-5}\cline{6-8}\cline{9-11}
$m$ & $n$ & O-IAL & APF-IAL & ProxALM \cite{LuM} & O-IAL & APF-IAL & ProxALM \cite{LuM} & O-IAL & APF-IAL & ProxALM \cite{LuM} \\ \hline
     50 &     100 &    0.65 &    \textcolor{purple}{0.22} &    4.71 & 7.0415e-06 & 6.8516e-06 & 3.6004e-09 & 1.2061e-10 & 9.9789e-06 & 9.4590e-06 \\ \hline
    100 &     150 &    1.19 &    \textcolor{purple}{0.19} &    4.83 & 6.9268e-06 & 5.0017e-06 & 1.2150e-08 & 1.2494e-10 & 9.8587e-06 & 8.9484e-06 \\ \hline
    125 &     200 &    1.51 &    \textcolor{purple}{0.24} &    7.19 & 9.9941e-06 & 6.4578e-06 & 1.6544e-09 & 1.2386e-10 & 9.9646e-06 & 9.5325e-06 \\ \hline
    150 &     300 &    1.37 &    \textcolor{purple}{0.34} &    8.96 & 4.7838e-06 & 4.7497e-06 & 1.2518e-09 & 1.2468e-10 & 9.9642e-06 & 9.1886e-06 \\ \hline
    175 &     400 &    2.34 &    \textcolor{purple}{0.42} &   12.62 & 5.9537e-06 & 4.4589e-06 & 3.4247e-09 & 1.2468e-10 & 9.9959e-06 & 9.9159e-06 \\ \hline
    200 &     500 &    1.25 &    \textcolor{purple}{0.42} &   15.02 & 3.9627e-06 & 6.8020e-06 & 1.3402e-09 & 1.2477e-10 & 9.9725e-06 & 6.8233e-06 \\ \hline
    250 &     750 &    2.84 &    \textcolor{purple}{0.74} &   25.77 & 3.5943e-06 & 5.5857e-06 & 3.5027e-09 & 1.2499e-10 & 9.9918e-06 & 9.5666e-06 \\ \hline
    300 &    1000 &    3.34 &    \textcolor{purple}{1.00} &   28.92 & 3.8711e-06 & 5.5445e-06 & 2.9092e-09 & 1.2496e-10 & 9.9868e-06 & 6.5960e-06 \\ \hline
    350 &    1200 &    3.61 &    \textcolor{purple}{1.31} &   47.98 & 9.9795e-06 & 8.2046e-06 & 5.2400e-10 & 1.2308e-10 & 9.9622e-06 & 9.8896e-06 \\ \hline
    400 &    1400 &    6.04 &    \textcolor{purple}{1.49} &   54.06 & 1.8704e-06 & 6.7112e-06 & 2.4277e-09 & 1.2487e-10 & 9.9592e-06 & 9.7753e-06 \\ \hline
    450 &    1500 &    9.61 &    \textcolor{purple}{1.61} &   47.48 & 2.8357e-06 & 6.4066e-06 & 1.8368e-09 & 1.2477e-10 & 9.9766e-06 & 9.8213e-06 \\ \hline
    500 &    1750 &   14.52 &    \textcolor{purple}{2.29} &   96.82 & 1.8923e-06 & 9.6769e-06 & 2.5087e-09 & 1.2433e-10 & 9.8801e-06 & 7.7157e-06 \\ \hline
    600 &    2000 &   10.51 &    \textcolor{purple}{3.40} &  173.07 & 2.7537e-06 & 4.0733e-06 & 3.2842e-09 & 1.2452e-10 & 9.9636e-06 & 7.6926e-06 \\ \hline
    800 &    2200 &   27.92 &    \textcolor{purple}{4.21} &  230.68 & 7.6531e-06 & 4.3184e-06 & 2.8696e-09 & 1.2441e-10 & 9.9961e-06 & 9.6622e-06 \\ \hline
   1000 &    2500 &   40.77 &    \textcolor{purple}{8.05} &  336.23 & 4.0537e-06 & 6.6977e-06 & 1.1896e-09 & 1.2455e-10 & 9.9896e-06 & 5.6936e-06 \\ \hline
\end{tabular}

%% file: table_ML1_ElasticNet.tex
\begin{tabular}{|cc|ccc|ccc|ccc|} \hline
\multicolumn{2}{|c||}{Dimensions} & \multicolumn{3}{c|}{Time (s)} & \multicolumn{3}{c|}{Primal Feasibility Residual} & \multicolumn{3}{c|}{Dual Stationarity Residual} \\ \cline{1-2}\cline{3-5}\cline{6-8}\cline{9-11}
$m$ & $n$ & O-IAL & APF-IAL & ProxALM \cite{LuM} & O-IAL & APF-IAL & ProxALM \cite{LuM} & O-IAL & APF-IAL & ProxALM \cite{LuM} \\ \hline
     50 &     100 &    0.13 &    \textcolor{purple}{0.04} &    0.29 & 8.6296e-06 & 7.1069e-06 & 7.9688e-08 & 9.9672e-11 & 6.9169e-06 & 3.3123e-06 \\ \hline
    100 &     125 &    0.20 &    \textcolor{purple}{0.02} &    0.35 & 7.1036e-06 & 4.8189e-06 & 9.9291e-09 & 2.1931e-10 & 9.5429e-06 & 5.1250e-07 \\ \hline
    140 &     200 &    0.30 &    \textcolor{purple}{0.03} &    0.55 & 7.0256e-06 & 6.0711e-06 & 6.0293e-08 & 2.4822e-10 & 9.5494e-06 & 2.8900e-06 \\ \hline
    180 &     250 &    0.31 &    \textcolor{purple}{0.04} &    0.76 & 5.5541e-06 & 4.5893e-06 & 8.6723e-09 & 2.4786e-10 & 6.3845e-06 & 6.1334e-07 \\ \hline
    200 &     280 &    0.45 &    \textcolor{purple}{0.05} &    0.87 & 6.0040e-06 & 5.4705e-06 & 8.7956e-08 & 2.3822e-10 & 8.4981e-06 & 7.5677e-06 \\ \hline
    250 &     300 &    0.63 &    \textcolor{purple}{0.04} &    1.09 & 9.2774e-06 & 3.7229e-06 & 9.2569e-09 & 2.2389e-10 & 6.0941e-06 & 9.5241e-07 \\ \hline
    350 &     400 &    1.34 &    \textcolor{purple}{0.05} &    1.91 & 8.9104e-06 & 4.4272e-06 & 8.9300e-09 & 2.4985e-10 & 9.9473e-06 & 8.0528e-06 \\ \hline
    400 &     450 &    2.23 &    \textcolor{purple}{0.07} &    3.01 & 9.7796e-06 & 6.5893e-06 & 1.4934e-08 & 1.9934e-10 & 7.3994e-06 & 4.1762e-06 \\ \hline
    450 &     520 &   15.62 &    \textcolor{purple}{0.30} &  122.30 & 9.9297e-06 & 5.5994e-06 & 3.5779e-06 & 2.4147e-10 & 9.9162e-06 & 6.7195e-06 \\ \hline
    500 &     600 &    3.95 &    \textcolor{purple}{0.24} &    9.85 & 9.3737e-06 & 2.6084e-06 & 1.8283e-08 & 2.4263e-10 & 9.8893e-06 & 9.8151e-06 \\ \hline
    550 &     700 &    2.23 &    \textcolor{purple}{0.34} &   12.16 & 7.5875e-06 & 1.4928e-06 & 1.4847e-08 & 2.4637e-10 & 9.9663e-06 & 5.6968e-06 \\ \hline
    650 &     850 &    2.36 &    \textcolor{purple}{0.44} &   16.77 & 6.3560e-06 & 6.4186e-06 & 3.7237e-09 & 2.4342e-10 & 9.9593e-06 & 9.5102e-06 \\ \hline
    750 &    1000 &    3.32 &    \textcolor{purple}{0.64} &   23.04 & 7.7759e-06 & 9.5629e-06 & 7.3995e-09 & 2.4916e-10 & 9.9151e-06 & 8.7765e-06 \\ \hline
    850 &    1200 &    6.80 &    \textcolor{purple}{1.56} &   54.32 & 5.2259e-06 & 7.5335e-06 & 2.0233e-09 & 2.4707e-10 & 9.9031e-06 & 7.1109e-06 \\ \hline
   1000 &    1500 &   12.25 &    \textcolor{purple}{3.26} &  149.21 & 3.9914e-06 & 5.8409e-06 & 4.3336e-09 & 2.4741e-10 & 9.9514e-06 & 8.1034e-06 \\ \hline
\end{tabular}

%% file: table_ML2_HuberizedSVM.tex
\begin{tabular}{|cc|ccc|ccc|ccc|} \hline
\multicolumn{2}{|c||}{Dimensions} & \multicolumn{3}{c|}{Time (s)} & \multicolumn{3}{c|}{Primal Feasibility Residual} & \multicolumn{3}{c|}{Dual Stationarity Residual} \\ \cline{1-2}\cline{3-5}\cline{6-8}\cline{9-11}
$m$ & $n$ & O-IAL & APF-IAL & ProxALM \cite{LuM} & O-IAL & APF-IAL & ProxALM \cite{LuM} & O-IAL & APF-IAL & ProxALM \cite{LuM} \\ \hline
    100 &     200 &    1.51 &    \textcolor{purple}{0.14} &    0.98 & 9.8647e-06 & 2.8261e-06 & 2.0833e-08 & 2.4535e-10 & 9.9942e-06 & 9.9698e-06 \\ \hline
    125 &     300 &    3.06 &    \textcolor{purple}{0.25} &    2.14 & 9.4398e-06 & 4.4665e-07 & 2.4831e-08 & 2.4959e-10 & 9.9785e-06 & 8.7184e-06 \\ \hline
    160 &     350 &    4.00 &    \textcolor{purple}{0.36} &    3.66 & 8.5234e-06 & 2.8080e-07 & 2.7076e-08 & 2.4924e-10 & 9.9912e-06 & 9.8796e-06 \\ \hline
    220 &     420 &    6.24 &    \textcolor{purple}{0.23} &    3.06 & 9.3216e-06 & 1.8530e-06 & 1.2041e-07 & 2.4800e-10 & 9.9502e-06 & 7.0318e-06 \\ \hline
    300 &     500 &   11.38 &    \textcolor{purple}{0.26} &    2.86 & 9.8340e-06 & 7.1945e-06 & 7.4209e-08 & 2.4857e-10 & 9.8958e-06 & 6.3296e-06 \\ \hline
    340 &     650 &   13.63 &    \textcolor{purple}{0.52} &    5.69 & 9.2159e-06 & 2.1553e-06 & 2.8755e-08 & 2.4860e-10 & 9.9846e-06 & 6.2586e-06 \\ \hline
    410 &     800 &   17.75 &    \textcolor{purple}{0.76} &   10.88 & 9.3374e-06 & 2.2401e-06 & 7.6778e-08 & 2.4880e-10 & 9.9554e-06 & 5.0307e-06 \\ \hline
    500 &    1000 &   38.64 &    \textcolor{purple}{1.76} &   25.09 & 9.4024e-06 & 2.3469e-06 & 1.2538e-07 & 2.4732e-10 & 9.9932e-06 & 7.4665e-06 \\ \hline
    580 &    1300 &   77.52 &    \textcolor{purple}{4.40} &   50.60 & 8.3424e-06 & 2.3128e-06 & 3.4474e-08 & 2.4784e-10 & 9.9934e-06 & 9.1382e-06 \\ \hline
    630 &    1800 &  103.71 &    \textcolor{purple}{9.61} &  158.06 & 7.5187e-06 & 3.4824e-06 & 4.8961e-08 & 2.4937e-10 & 9.9756e-06 & 5.3972e-06 \\ \hline
    720 &    2200 &  185.82 &   \textcolor{purple}{16.02} &  256.62 & 9.8946e-06 & 5.4031e-06 & 1.1028e-08 & 2.4799e-10 & 9.9957e-06 & 7.2998e-06 \\ \hline
    800 &    2500 &  165.66 &   \textcolor{purple}{16.03} &  310.64 & 8.0908e-06 & 3.3600e-06 & 6.8596e-09 & 2.4931e-10 & 9.9595e-06 & 6.7552e-06 \\ \hline
   1000 &    2800 &  248.27 &   \textcolor{purple}{16.39} &  360.11 & 7.8138e-06 & 3.8748e-06 & 1.6719e-08 & 2.4804e-10 & 9.8386e-06 & 5.0048e-06 \\ \hline
   1200 &    3500 &  387.89 &   \textcolor{purple}{25.50} &  639.73 & 9.7420e-06 & 9.5725e-06 & 5.6077e-09 & 2.4937e-10 & 9.9801e-06 & 5.3442e-06 \\ \hline
   2000 &    4000 & 1513.25 &   \textcolor{purple}{42.62} & 1333.21 & 9.7640e-06 & 3.8331e-06 & 6.5937e-08 & 2.4522e-10 & 9.9615e-06 & 6.6974e-06 \\ \hline
\end{tabular}

%% file: table_ML3_BayesianDOptimal.tex
\begin{tabular}{|cc|ccc|ccc|ccc|} \hline
\multicolumn{2}{|c||}{Dimensions} & \multicolumn{3}{c|}{Time (s)} & \multicolumn{3}{c|}{Primal Feasibility Residual} & \multicolumn{3}{c|}{Dual Stationarity Residual} \\ \cline{1-2}\cline{3-5}\cline{6-8}\cline{9-11}
$m$ & $n$ & O-IAL & APF-IAL & ProxALM \cite{LuM} & O-IAL & APF-IAL & ProxALM \cite{LuM} & O-IAL & APF-IAL & ProxALM \cite{LuM} \\ \hline
      5 &      25 &    0.45 &    \textcolor{purple}{0.09} &    1.25 & 2.7883e-07 & 5.9706e-07 & 4.5457e-09 & 1.2415e-10 & 9.9057e-06 & 9.4089e-06 \\ \hline
     10 &      30 &    0.61 &    \textcolor{purple}{0.27} &    2.55 & 4.5228e-07 & 8.2275e-07 & 3.9348e-08 & 1.1904e-10 & 9.9933e-06 & 8.7112e-06 \\ \hline
     20 &      50 &    1.12 &    \textcolor{purple}{0.60} &    5.23 & 7.9678e-06 & 5.8942e-06 & 8.5390e-08 & 1.2430e-10 & 9.9840e-06 & 9.3609e-06 \\ \hline
     30 &      70 &    5.88 &    \textcolor{purple}{2.10} &   24.69 & 9.8960e-06 & 1.0200e-06 & 2.2628e-07 & 1.2487e-10 & 9.9899e-06 & 3.0808e-06 \\ \hline
     50 &      80 &    7.15 &    \textcolor{purple}{0.51} &    4.75 & 8.2063e-06 & 1.5803e-06 & 8.4101e-08 & 1.2217e-10 & 9.9959e-06 & 9.8783e-06 \\ \hline
     60 &      90 &    9.33 &    \textcolor{purple}{0.59} &    5.57 & 8.9226e-06 & 6.4077e-06 & 8.0434e-08 & 1.2488e-10 & 9.9937e-06 & 4.0106e-08 \\ \hline
     75 &     115 &   21.45 &    \textcolor{purple}{0.85} &   11.64 & 9.1407e-06 & 4.1737e-06 & 1.3451e-07 & 1.2238e-10 & 9.9774e-06 & 8.2729e-06 \\ \hline
     85 &     125 &   23.16 &    \textcolor{purple}{1.03} &   12.00 & 9.7713e-06 & 1.8210e-06 & 1.9272e-07 & 1.2352e-10 & 9.9849e-06 & 9.3506e-06 \\ \hline
    100 &     175 &   44.07 &    \textcolor{purple}{2.41} &   34.99 & 9.1592e-06 & 4.1352e-06 & 4.2160e-07 & 1.2367e-10 & 9.9955e-06 & 5.8512e-08 \\ \hline
    120 &     200 &   71.97 &    \textcolor{purple}{3.58} &   73.93 & 9.6673e-06 & 2.3166e-06 & 9.1791e-06 & 1.2487e-10 & 9.9984e-06 & 6.9781e-07 \\ \hline
    130 &     220 &   72.12 &    \textcolor{purple}{4.48} &   78.33 & 9.7896e-06 & 4.1079e-06 & 3.5511e-06 & 1.2437e-10 & 9.9813e-06 & 2.0627e-07 \\ \hline
    140 &     230 &   97.34 &    \textcolor{purple}{5.40} &  117.74 & 9.7533e-06 & 4.9281e-06 & 8.5930e-06 & 1.2353e-10 & 9.9925e-06 & 3.5834e-06 \\ \hline
    160 &     250 &  135.96 &    \textcolor{purple}{7.10} &  194.70 & 9.7997e-06 & 6.4149e-06 & 5.2830e-06 & 1.2497e-10 & 9.9970e-06 & 1.8520e-07 \\ \hline
    180 &     275 &  192.57 &    \textcolor{purple}{8.58} &  317.97 & 9.9753e-06 & 5.9361e-06 & 5.2906e-06 & 1.2344e-10 & 9.9924e-06 & 1.4934e-07 \\ \hline
    200 &     300 &  204.77 &   \textcolor{purple}{11.30} &  282.49 & 9.9368e-06 & 9.2390e-06 & 7.9677e-06 & 1.2329e-10 & 9.9961e-06 & 3.3229e-07 \\ \hline
\end{tabular}

%% file: table_quantSDP4_QuantumTomography.tex
\begin{tabular}{|cc|ccc|ccc|ccc|} \hline
\multicolumn{2}{|c||}{Dimensions} & \multicolumn{3}{c|}{Time (s)} & \multicolumn{3}{c|}{Primal Feasibility Residual} & \multicolumn{3}{c|}{Dual Stationarity Residual} \\ \cline{1-2}\cline{3-5}\cline{6-8}\cline{9-11}
$m$ & $n$ & O-IAL & APF-IAL & ProxALM \cite{LuM} & O-IAL & APF-IAL & ProxALM \cite{LuM} & O-IAL & APF-IAL & ProxALM \cite{LuM} \\ \hline
      8 &     144 &    0.68 &    \textcolor{purple}{0.44} &   10.18 & 1.5906e-12 & 6.9649e-06 & 9.8046e-10 & 1.7489e-10 & 9.9437e-06 & 6.4462e-06 \\ \hline
     10 &     196 &    0.45 &    \textcolor{purple}{0.28} &   11.67 & 5.2495e-13 & 4.2385e-06 & 2.2599e-09 & 1.7676e-10 & 9.9640e-06 & 6.3797e-06 \\ \hline
     12 &     256 &    0.49 &    \textcolor{purple}{0.41} &   14.21 & 9.9633e-13 & 3.2955e-06 & 2.3523e-09 & 1.7645e-10 & 9.9992e-06 & 5.4347e-06 \\ \hline
     15 &     324 &    6.78 &    \textcolor{purple}{5.90} &  358.13 & 8.2771e-06 & 9.0839e-07 & 3.5424e-09 & 1.7670e-10 & 9.9997e-06 & 8.1455e-06 \\ \hline
     18 &     400 &    4.24 &    \textcolor{purple}{2.17} &  181.61 & 2.3007e-12 & 1.5707e-06 & 2.4364e-09 & 1.7641e-10 & 9.9970e-06 & 7.7798e-06 \\ \hline
     21 &     484 &    3.27 &    \textcolor{purple}{1.88} &   55.63 & 9.7588e-08 & 1.6571e-06 & 5.1399e-10 & 1.7654e-10 & 9.9968e-06 & 6.5909e-06 \\ \hline
     24 &     576 &    2.55 &    \textcolor{purple}{1.45} &   35.36 & 6.6984e-07 & 2.7532e-06 & 1.5051e-09 & 1.7600e-10 & 9.9852e-06 & 6.7207e-06 \\ \hline
     27 &     676 &    3.92 &    \textcolor{purple}{1.27} &   35.01 & 1.7832e-06 & 8.3042e-07 & 1.9601e-09 & 1.7504e-10 & 9.9963e-06 & 9.1124e-06 \\ \hline
     30 &     784 &    3.70 &    \textcolor{purple}{1.33} &   37.50 & 3.2894e-06 & 1.6693e-06 & 1.3051e-10 & 1.7633e-10 & 9.9162e-06 & 7.9736e-06 \\ \hline
     35 &    1024 &    5.80 &    \textcolor{purple}{1.38} &   34.42 & 8.1250e-06 & 1.8668e-06 & 3.7588e-10 & 1.7563e-10 & 9.9952e-06 & 6.8599e-06 \\ \hline
     42 &    1296 &    4.57 &    \textcolor{purple}{1.65} &   38.22 & 1.6795e-07 & 1.6715e-06 & 4.2726e-10 & 1.7626e-10 & 9.9749e-06 & 9.3537e-06 \\ \hline
     50 &    1600 &    4.16 &    \textcolor{purple}{1.84} &   45.34 & 2.6703e-07 & 1.7037e-06 & 3.6872e-10 & 1.7611e-10 & 9.9895e-06 & 9.2976e-06 \\ \hline
     61 &    2304 &    7.75 &    \textcolor{purple}{2.51} &   82.04 & 3.4542e-07 & 1.0970e-06 & 5.2736e-10 & 1.7339e-10 & 9.9901e-06 & 9.1086e-06 \\ \hline
     73 &    3136 &   11.14 &    \textcolor{purple}{3.66} &  115.43 & 4.8447e-07 & 1.4990e-06 & 6.9661e-10 & 1.7567e-10 & 9.9307e-06 & 9.9191e-06 \\ \hline
     86 &    4096 &   18.02 &    \textcolor{purple}{6.10} &  232.20 & 7.0560e-07 & 1.8597e-06 & 1.1957e-09 & 1.7302e-10 & 9.9470e-06 & 9.3482e-06 \\ \hline
\end{tabular}